\documentclass[11pt]{amsart}
\usepackage{amssymb, amsmath, amsthm, mathrsfs, amsopn, bm, color}
\usepackage{amsrefs, framed}

\usepackage{tikz}

\usepackage{enumitem}

\usepackage[a4paper, centering]{geometry}
\usepackage{bookmark} 

\usepackage{hyperref} 

\numberwithin{equation}{section}

\def\DM{{\mathcal{DM}^{\infty}}}
\def\DMA{{{\rm DM}^\infty}}

\DeclareMathOperator{\Tr}{Tr}
\newcommand*{\Trace}[3][\pm]{\Tr^{#1}(#2, #3)}
\newcommand*{\Trp}[2]{\Trace[+]{#1}{#2}}
\newcommand*{\Trm}[2]{\Trace[-]{#1}{#2}}

\newcommand*{\jump}[1]{\Theta_{#1}}
\DeclareMathOperator{\Trv}{\bf{Tr}}
\newcommand*{\Tracev}[3][\pm]{\Trv^{#1}(#2, #3)}
\newcommand*{\Trpv}[2]{\Tracev[+]{#1}{#2}}
\newcommand*{\Trmv}[2]{\Tracev[-]{#1}{#2}}
\newcommand*{\Trpmv}[2]{\Tracev[\pm]{#1}{#2}}

\newcommand*{\DMAloc}{{\rm DM}^\infty_{{\rm loc}}}

\newcommand{\A}{{\boldsymbol{{A}}}}
\newcommand{\B}{\boldsymbol B}
\newcommand{\C}{\boldsymbol C}

\DeclareMathOperator{\DIV}{{\rm Div}}

\newcommand{\bfu}{{\bf u}}

\def\R{\mathbb{R}}

\def\tr{{\rm tr}}

\newcommand{\uxz}{u^{\bm\xi}_y}
\newcommand{\xz}{^{\bm\xi}_y}

\def\MA{{\bf A}}

\DeclareMathOperator{\Div}{div}

\newcommand{\medint}{-\kern  -,375cm\int}
\newcommand{\medintinrigo}{-\kern  -,315cm\int}

\newcommand{\eps}{\varepsilon}
 \newcommand{\hh}{\Haus{N-1}}

\newcommand{\res}{\mathop{\hbox{\vrule height 7pt width .5pt depth 0pt
\vrule height .5pt width 6pt depth 0pt}}\nolimits} 

\def\pscal#1#2{\left\langle #1\,,\, #2 \right\rangle}

\def\sfera{{\mathchoice
{\setbox0=\hbox{$\displaystyle     \rm S$}\hbox{\raise0.5\ht0\hbox
to0pt{\kern0.35\wd0\vrule height0.45\ht0\hss}\hbox
to0pt{\kern0.55\wd0\vrule height0.5\ht0\hss}\box0}}
{\setbox0=\hbox{$\textstyle        \rm S$}\hbox{\raise0.5\ht0\hbox
to0pt{\kern0.35\wd0\vrule height0.45\ht0\hss}\hbox
to0pt{\kern0.55\wd0\vrule height0.5\ht0\hss}\box0}}
{\setbox0=\hbox{$\scriptstyle      \rm S$}\hbox{\raise0.5\ht0\hbox
to0pt{\kern0.35\wd0\vrule height0.45\ht0\hss}\raise0.05\ht0\hbox
to0pt{\kern0.5\wd0\vrule height0.45\ht0\hss}\box0}}
{\setbox0=\hbox{$\scriptscriptstyle\rm S$}\hbox{\raise0.5\ht0\hbox
to0pt{\kern0.4\wd0\vrule height0.45\ht0\hss}\raise0.05\ht0\hbox
to0pt{\kern0.55\wd0\vrule height0.45\ht0\hss}\box0}}}}
\def\q{{\mathchoice {\setbox0=\hbox{$\displaystyle\rm
Q$}\hbox{\raise
0.15\ht0\hbox to0pt{\kern0.4\wd0\vrule height0.8\ht0\hss}\box0}}
{\setbox0=\hbox{$\textstyle\rm Q$}\hbox{\raise
0.15\ht0\hbox to0pt{\kern0.4\wd0\vrule height0.8\ht0\hss}\box0}}
{\setbox0=\hbox{$\scriptstyle\rm Q$}\hbox{\raise
0.15\ht0\hbox to0pt{\kern0.4\wd0\vrule height0.7\ht0\hss}\box0}}
{\setbox0=\hbox{$\scriptscriptstyle\rm Q$}\hbox{\raise
0.15\ht0\hbox to0pt{\kern0.4\wd0\vrule height0.7\ht0\hss}\box0}}}}

\def\BV{{\rm{BV}}}

\newcommand{\Haus}[1]{{\mathscr H}^{#1}} 
\newcommand{\Leb}[1]{{\mathscr L}^{#1}} 

\newtheorem{definition}{Definition}[section]

\newtheorem{lemma}[definition]{Lemma}

\newtheorem{theorem}[definition]{Theorem}

\newtheorem{proposition}[definition]{Proposition}

\theoremstyle{remark}
\newtheorem{remark}[definition]{Remark}

\makeatletter
\def\@settitle{\begin{center}%
		\baselineskip14\p@\relax
		\bfseries
		\uppercasenonmath\@title
		\@title
		\ifx\@subtitle\@empty\else
		\\[5ex]
		\normalsize\mdseries\@subtitle
		\fi
	\end{center}%
}
\def\subtitle#1{\gdef\@subtitle{#1}}
\def\@subtitle{}
\makeatother

\begin{document}

\title[Stress-strain duality via slicing]
{A slicing approach to stress-strain duality}

\author[V.~De Cicco]{Virginia De Cicco}
\address{Department of Basic and Applied Sciences for Engineering, Sapienza University of Rome \\
	Via A.\ Scarpa 10 - I-00185 Rome (Italy)}
\email{virginia.decicco@uniroma1.it}
\author[G.~Scilla]{Giovanni Scilla 
}
\address{Department of Mathematics and Applications ``R. Caccioppoli'', University of Naples Federico II\\
	Via Cintia, Monte Sant'Angelo - I-80126 Naples (Italy)}
\email{giovanni.scilla@unina.it}



\begin{abstract}
The classical Kohn-Temam stress-strain pairing $(\MA:E\bfu)$ for symmetric tensors $\MA$ and
$\bfu\in BD$ is typically formulated
under summability assumptions on the divergence of $\MA$.
This excludes stress fields whose divergence has singular surface
contributions, as occurs at cracks and material interfaces in continuum mechanics. We define
and study stress-strain pairings for bounded symmetric
divergence-measure tensor fields. 
For general
$\bfu\in BD$, we introduce a slicing
pairing $((\MA:E\bfu))_\Xi$ for tensor fields satisfying a
directional $BV$-type condition with respect to a finite frame
$\Xi$. 
The definition is based on a one-dimensional disintegration strategy, and despite this construction, the new pairing enjoys analogous properties of the usual pairing $(\MA:E\bfu)$, such as the absolute continuity with respect to $|E\bfu|$ and the Gauss-Green formulas.
We also identify several situations in which the pairing is independent of the choice of frame $\Xi$, including the relevant case in which the stress field $\MA$ belongs to $BV$.
While a distributional stress-strain pairing can be defined naturally for bounded $BD$ functions, the truncation techniques used in $BV$ to extend it to the unbounded setting fail in $BD$.
The slicing pairing is consistent with the distributional one under a natural frame-wise compatibility condition, while being more general even for bounded $\bfu$. Indeed, its existence does not require the compatibility condition $|\DIV\MA|(S_{\bfu}\setminus J_\bfu)=0$ which is necessary for the distributional definition. We also observe that this condition is automatically satisfied when the singular part of $\DIV\MA$ is concentrated on a countably rectifiable set, such as a crack or an interface, and that it might fail only if $\Haus{N-1}(S_\bfu\setminus J_\bfu)>0$, which is an open problem. 
\end{abstract}

\maketitle

\section{Introduction}
The mathematical formulation of the duality between stress and strain tensor fields represents a cornerstone in the variational analysis of inelastic behaviors, such as perfect plasticity, damage, and fracture mechanics. In the classical framework established by Kohn and Temam in \cite{KT}, the definition of a pairing $(\MA : E\bfu)$ between a bounded symmetric stress tensor $\MA \in L^\infty(\Omega; \mathbb{R}_{\text{sym}}^{N \times N})$ and a strain tensor $E\bfu$ associated with a displacement $\bfu \in BD(\Omega)$ crucially requires that the divergence of the stress tensor satisfies $\DIV \MA \in L^N(\Omega;\mathbb{R}^N)$.
The only available notion is the classical distributional definition of pairing (see \cite[Ch. II, Sec. 7.3]{Temam} and \cite{KT})
\begin{equation}\label{eq:pairingBDTemintro}
\pscal{(\MA:E\bfu)}{\varphi} :=
-\int_\Omega \varphi \bfu\cdot \DIV \MA\, dx - \int_\Omega  \, \MA:[\bfu \odot\nabla\varphi]\, dx\,,
\end{equation}
for every $\varphi \in C^\infty_c(\Omega)$, where $\odot$ denotes the symmetrized tensor product between vectors. Most developments in elasto-plasticity have been formulated under assumptions ensuring that the admissible stresses remain within the scope of the classical theory, in particular requiring summable divergence. Within the framework of functions with bounded deformation, these issues have been extensively investigated through variational approaches. In the context of plasticity, a foundational contribution is provided by Dal Maso, De Simone and Mora~\cite{DMDSM}, who developed a general theory of quasistatic evolution for linearly elastic--perfectly plastic materials. 

Nevertheless, in several areas of applied sciences and engineering, tensor fields naturally arise whose divergence contains jump (surface) terms. 
For instance, in composite elastic bodies, jumps in the stress across interfaces describe surface forces associated with adhesion and debonding phenomena, and the corresponding pairing with the displacement field represents interfacial work. In fracture mechanics, cracks induce discontinuities in the stress field, while the pairing $(\MA : E\bfu)$ accounts for bulk elastic energy as well as fracture and crack-opening energy (see \cite{Gr,Bar,FrMa}). In fluid dynamics, shocks in conservation laws produce concentration effects in the divergence of stress-flux fields (see \cite{ChenFrid,ChFr1}). 
These examples show that the symmetric stress-strain pairing in $BD$ is not only mathematically natural, but also essential for modeling realistic systems involving cracks, interfaces, plasticity, and concentrated forces. In particular, they illustrate how jump contributions in the divergence operator naturally encode surface and interface forces within a unified variational framework.

On the other hand, starting from the seminal paper \cite{Anz}, the pairing  theory in $BV$ has undergone significant developments in recent years, driven by  applications in several contexts (\cite{ChenFrid, CD3, CDCM, CCDM, CDCS, CD5, CDM2, DCSTensor}), which suggest possible extensions of the classical $BD$ theory to more irregular stresses and divergences. 
Indeed, the original motivation of \cite{Anz} was the variational approach to elasto-plasticity (see \cite{Anz3,AnzGia,Anzellotti1984}). 

Inspired by these recent advances, we address the problem of defining a general pairing $(\MA:E\bfu)$. The ideas developed in the aforementioned works allow one to treat the case of bounded $BD$ functions rather naturally through a distributional formulation. The extension to arbitrary $BD$ functions is more delicate, as the available techniques break down in the absence of truncation arguments and must be replaced by a different approach.

\emph{Our results.\ } Motivated by this perspective, we aim to extend the classical definition \eqref{eq:pairingBDTemintro} in order to include bounded symmetric tensor fields $\MA$ with divergence measure.

We start by introducing a pairing in $BD$ for \emph{bounded} functions (see, for instance, \cite[Theorem 3.1]{CFI} for an example of a function of this type). In this case, a natural definition of pairing between $\MA$ and $E\bfu$ would be the distribution $(\MA : E\bfu) \in \mathcal{D}'(\Omega)$ acting as
\begin{equation}\label{eq:pairingBDintro}
\langle (\MA:E\bfu), \varphi \rangle
:=
-\int_\Omega \varphi \,\bfu^* \cdot d \DIV \MA
-\int_\Omega \MA : (\bfu \odot \nabla \varphi)\, dx,
\end{equation}
for every $\varphi \in C^\infty_c(\Omega)$, where $\bfu^*$ is the precise representative of $\bfu$ defined on $\Omega\setminus(S_\bfu\setminus J_\bfu)$ by \eqref{eq:preciserepresentative}. 
However, although $|\DIV\MA|\ll\Haus{N-1}$, the term $\int_\Omega \varphi \, \bfu^* \cdot d\DIV \MA$ may fail to be well defined, since it is not known whether $\Haus{N-1}(S_\bfu \setminus J_\bfu)=0$ (see Section~\ref{sec:preliminaries} (§~\hyperref[sec:BD]{\emph{Functions of bounded deformation}})).
To overcome this issue, we require the assumption 
\begin{equation}\label{assumintro} \tag{${\rm H}_\ast$}
|\DIV^s \MA|(S_{\bfu}\setminus J_\bfu)=0,
\end{equation}
where $\DIV^s \MA$ denotes the singular part of the measure $\DIV \MA$,
so that $\bfu^*$ turns out to be defined $|\DIV\MA|$-a.e. (see discussion in Section \ref{sec:pairingBDbounded}).  
From a physical standpoint, condition \eqref{assumintro} requires that the singular force concentrations, represented by $\DIV^s\MA$, do not charge the set $S_\bfu\setminus J_\bfu$ of discontinuity points of $\bfu$ which are not jump points. This is not a restriction when $\DIV^s\MA$ is concentrated on interfaces or crack surfaces: since $\Haus{N-1}(\Sigma\cap(S_\bfu\setminus J_\bfu))=0$ for every countably $\Haus{N-1}$-rectifiable set $\Sigma$, condition \eqref{assumintro} holds whenever $|\DIV^s\MA|$ is concentrated on such a set (see Remark~\ref{rem:Hstar}). It could fail only if $\Haus{N-1}(S_\bfu\setminus J_\bfu)>0$, which is an open problem (see \cite[Remark 6.2]{AmbCosDalM}), and $|\DIV^s\MA|$ charges a purely unrectifiable subset of $S_\bfu\setminus J_\bfu$.

The distribution \eqref{eq:pairingBDintro} is shown to be a Radon measure (see Theorem \ref{thm:main2}) and to satisfy the Gauss-Green formula
(see Theorem \ref{main1}).
 
On the other hand, the boundedness of $\bfu$
 is a rather restrictive modeling assumption, especially in brittle fracture mechanics, so that it seems natural to treat the case of \emph{unbounded} $BD$ functions.
However, a further difficulty arises in this setting, since the distributional definition \eqref{eq:pairingBDintro} is no longer directly applicable, as it would require a truncation argument which is not available in $BD$. The use of slicing techniques to overcome this type of difficulty is reminiscent of the definitions of the spaces of generalized functions of bounded deformation $GBD$ and $GSBD$ introduced by Dal Maso \cite{DalMaso2013}, where generalized deformations are characterized through suitable one-dimensional sections.

This observation naturally leads us to introduce,
for general $BD$ functions, a slicing-based definition of pairing. 
We impose a structural assumption ensuring one-dimensional $BV$ regularity of the components of a symmetric matrix $\MA$ along a ``frame'' - that is, a suitable finite set of directions $\bm\xi \in \Xi$ - namely, $\MA\in \BV^\infty_{\Xi}(\Omega;\R^{N\times N}_{\rm sym})$ (see \eqref{eq:newspace}). This allows us to establish a Fubini-type formula that reduces the pairing to one-dimensional contributions.
The argument relies on slicing along directions $\bm\xi \in \Xi$ and on the fact that, for a.e.\ $y \in \bm\xi^\perp$, the scalar projections
$
\hat{u}_y^{\bm\xi}(t) := \bfu(y+t\bm\xi)\cdot \bm\xi
$
belong to $BV_{\rm loc}(\Omega^{{\bm\xi}}_y)$, where $\Omega^{{\bm\xi}}_y$ is a one-dimensional slice of $\Omega$. 

A preliminary, but crucial, result is a slicing formula for the distributional pairing \eqref{eq:pairingBDintro} for $\bfu\in L^\infty(\Omega;\R^N)$, as
\begin{equation}
\label{slicingBDintro}
(\MA:E\bfu)=\sum_{{\bm\xi}\in\Xi}\Leb{N-1}\res \Omega^{\bm\xi}\otimes_{\mathcal{M}}({a}_y^{{\bm\xi}},D{\hat{u}}_{y }^{{\bm\xi}})\res\Omega^{{\bm\xi}}_y,
\end{equation}
in the sense of measures (see Theorem \ref{thm:slicingBDbdd}), where $\otimes_{\mathcal{M}}$ denotes the generalized product of measures, see Section~\ref{sec:preliminaries} (§~\hyperref[sec:measures]{\emph{Measures}}). Here, ${a}_y^{\bm\xi}(s):={a}^{\bm\xi}(y+s\bm\xi)$ and ${a}^{\bm\xi}$ are the coefficients of $\MA$ with respect to the dual basis of $\{\bm\xi\otimes\bm\xi\}_{\bm\xi\in\Xi}$. 
In fact, the proof is obtained by disintegrating separately the Lebesgue, Cantor and jump parts. The handling of the Cantor part is more delicate and requires that $|E^c\bfu|(S_{\MA})=0$; i.e. the Cantor part of $E\bfu$ does not charge the set of the essential discontinuities of $\MA$.

Now, for a general $BD$ function, we introduce a new pairing, denoted by $((\MA:E\bfu))_{\Xi}$, taking \eqref{slicingBDintro} as the starting point. More precisely, given \(\MA\in\BV^\infty_{\Xi}(\Omega;\R^{N\times N}_{\rm sym})\) and \(\bfu\in BD(\Omega)\),
we define the linear functional
\(((\MA: E\bfu))_{\Xi} \colon  C^\infty_c(\Omega) \to \R\) by
\begin{equation}\label{eq:pairingBDnewintro}
\pscal{((\MA:E\bfu))_{\Xi}}{\varphi} 
:=
\sum_{{\bm\xi}\in\Xi}\int_{{\bm\xi}^\perp}
\left(\int_
{\Omega^{{\bm\xi}}_y} 
\varphi_y^{{\bm\xi}}(s)\,
d({a}_y^{{\bm\xi}},D{\hat{u}}_{y }^{{\bm\xi}})(s)
\right)\,d\Leb{N-1}(y).
\end{equation}
The resulting object is well defined, a priori depends on the frame $\Xi$ and  is consistent with the bounded case. This is the content of the following theorem (see Theorem \ref{consistenza} and Theorem \ref{mmmmm} 
for more precise statements), which is the main result of the paper.
\begin{theorem}\label{consistenzaintro}
Let \(\MA\in\BV^\infty_{\Xi}(\Omega;\R^{N\times N}_{\rm sym})$ and \(\bfu\in BD(\Omega)\).
Then the slicing pairing
 $((\MA:E\bfu))_{\Xi}$ is a Radon measure in $\Omega$, absolutely continuous with respect to \(|E\bfu|\). If, in addition, $\MA$ and $\bfu$ satisfy the frame-wise compatibility condition \eqref{assumXi} and the integrability conditions \eqref{eq:integrabslice1333}, then a Gauss-Green formula holds on bounded sets of finite perimeter. 
 Moreover,  if \(\bfu\in BD(\Omega)\cap L^\infty(\Omega;\R^N)\) such that $|E^c\bfu|(S_{\MA})=0$, and \eqref{assumintro} holds, or if \eqref{assumXi} holds, then
 $((\MA:E\bfu))_{\Xi}$ coincides with the pairing defined in \eqref{eq:pairingBDintro}; i.e.,
 $$
 ((\MA:E\bfu))_{\Xi}=(\MA:E\bfu),
 $$
 in the sense of measures.
\end{theorem}
The construction is not intrinsic in this level of generality, since both the class $\BV^\infty_\Xi$ and the definition of $((\MA:E\bfu))_\Xi$ involve the frame $\Xi$, the directional $BV$ assumptions being imposed only along the directions of $\Xi$. However, the absolutely continuous part and the jump part of $((\MA:E\bfu))_\Xi$ do not depend on $\Xi$ (see Proposition~\ref{t:framecomparison}), so that two frames can only produce different Cantor parts; we do not know whether this may actually happen. Although such a possible coordinate dependence may at first seem at odds with the objectivity expected in continuum mechanics, it provides the flexibility needed to accommodate singularities of $\DIV\MA$. We identify several situations in which the pairing becomes intrinsic and is therefore independent of the chosen frame (see Remarks~\ref{invariance} and~\ref{rem:W11LD}). In particular, this occurs when $\MA\in BV(\Omega;\R^{N\times N}_{\rm sym})\cap L^\infty(\Omega;\R^{N\times N}_{\rm sym})$, in which case the slicing construction reconstructs the intrinsic measure $\MA^*:E\bfu$.

Notice that the slicing pairing is well-defined without requiring the compatibility assumption \eqref{assumintro}, even for bounded $\bfu$. By the discussion above, this is relevant only if $|\DIV^s\MA|$ charges a purely unrectifiable subset of $S_\bfu\setminus J_\bfu$, which requires $\Haus{N-1}(S_\bfu\setminus J_\bfu)>0$.

A particularly illuminating step for developing the previous analysis, which is of independent interest, was a preliminary slicing decomposition of the pairing $(\MA, D\bfu)$ between bounded tensor fields with measure-valued divergence and $\bfu \in BV(\Omega;\mathbb{R}^N)$, as studied in \cite{DCSTensor}. This is achieved by exploiting the slicing theory in $BV$, which characterizes $BV$ functions through their one-dimensional sections and allows for a corresponding disintegration of the measure $D\bfu$. This study relies on several key features of $BV$ functions, which are not available in $BD$: the coarea formula, the chain rule together with truncation arguments, the fact that a vector-valued $BV$ function can be characterized componentwise in terms of scalar $BV$ functions (which in turn implies that the pairing can be decomposed as the sum of the corresponding scalar pairings), and the property $\Haus{N-1}(S_\bfu \setminus J_\bfu)=0$. 

In the last part  of the paper, we introduce two new notions of pairing $((\MA^D:E^D\bfu))^*_{\Xi}$ and $((\MA^D:E^D\bfu))_{\Xi}$ between the deviatoric parts of $\MA$ and $E\bfu$,
which can be compared with the one studied in
 \cite[Lemma~7.3]{Temam} and \cite[Theorem~3.2]{KT}.
The main difficulty is that, to the best of the authors' knowledge, no slicing result is currently available for $E^D\bfu$ (see \cite[Example 10.1]{ArroyoR}). This lack of a slicing theory is consistent with the fact that the deviatoric part of the strain does not satisfy the rank-one property, which is essential for slicing arguments. 
Nevertheless, we are able to give two ``mixed'' definitions for $\bfu\in BD(\Omega)$ and
$\MA\in\BV^\infty_{\Xi}(\Omega;\R^{N\times N}_{\rm sym})$. The first one, $((\MA^D:E^D\bfu))^*_{\Xi}$, uses the precise representative of ${\rm tr}\,\MA$ and requires a compatibility condition between $\Div\bfu$ and the discontinuity set of ${\rm tr}\,\MA$ (see \eqref{bbbbb}). The second one, $((\MA^D:E^D\bfu))_{\Xi}$, is defined for every $\bfu\in BD(\Omega)$ by means of a Borel representative $\MA^\Xi$ of $\MA$ adapted to the frame $\Xi$ (see \eqref{f:AXi} and \eqref{f:deviatoricslicing}), but in general it depends on the choice of this representative on the support of the singular part of $\Div\bfu$ (see Remark~\ref{t:remdevcomparison}).

Finally, in Section~\ref{sec:explicit_example}, we discuss a simple two-dimensional example of Mode~I opening of an interface carrying a concentrated normal force, as in cohesive or adhesive interfaces. The stress tensor considered therein has a singular $\DIV\MA$ with respect to $\Leb{2}$, hence the classical Kohn–Temam pairing theory does not apply and it does not allow one to compute the mechanical work across the interface, whereas the pairing $(\MA:E\bfu)$ captures the interfacial contribution. Since the displacement field in this example is a special function of bounded variation, the example is also covered by the pairing for $BV$ functions of \cite{DCSTensor}; its purpose is to illustrate, in an explicit case, the interfacial work, the Gauss--Green formula and the deviatoric pairings. \\
\\
\emph{Organization of the paper.\ } In Section~\ref{sec:preliminaries} we collect preliminaries on $BV$, $BD$, and divergence-measure fields, including slicing and disintegration results. Section~\ref{sec:slicingdive} is devoted to slicing formulas for divergence and normal traces. In Section~\ref{pimo}  we establish the slicing representation of the pairing in the $BV$ setting for scalar and vector-valued functions. Sections~\ref{sec:pairingBDbounded}-\ref{sec:pairingunbounded} develop the extension to tensor fields and $BD$ functions, including the definition of the pairing via slicing and the analysis of the unbounded case. 
In Section~\ref{sec:comparisonTemam} we compare our results with the classical theory, and extend it to (bounded) tensor fields with divergence measure. Finally, in Section~\ref{sec:explicit_example} we discuss an example. 

\section{Preliminaries} \label{sec:preliminaries}

\subsection*{\emph{Vectors and matrices}} \label{sec:vectormatrices}  If ${\bf a}, {\bf b} \in \R^N$, we write ${\bf a} \cdot {\bf b}$ for the Euclidean scalar product, and we denote by $|{\bf a}|=\sqrt{{\bf a} \cdot {\bf a}}$ the associated norm.  
\medskip

We write $\mathbb R^{N \times N}$ for the set of real $N \times N$ matrices, and denote by $\mathbb R^{N \times N}_{\rm sym}$ that of all real symmetric $N \times N$ matrices. The components $A_{ij}$ of a matrix $\MA$ are defined by
\begin{equation*}
A_{ij}:= {\bf e}_i \cdot \MA{\bf e}_j\,.
\end{equation*} Given two matrices $\MA$ and ${\bf B} \in \mathbb R^{N \times N}$, we consider the Frobenius scalar product $\MA:{\bf B}={\rm tr}(\MA^T {\bf B})$, where $\MA^T$ is the transpose of $\MA$, ${\rm tr }\MA$ is its trace, i.e. ${\rm tr }\MA=\sum_{i=1}^NA_{ii}$,
 and we denote by $|\MA|=\sqrt{\MA:\MA}$ the associated norm. Note that
\begin{equation}
|{\rm tr }\MA| \leq \sqrt{N} |\MA|\,.
\label{eq:traceineq}
\end{equation}
 
We recall that for any two vectors ${\bf a}, {\bf b}\in \R^N$, ${\bf a} \otimes {\bf b} :={\bf a} {\bf b}^T\in \mathbb R^{N \times N}$ stands for the tensor product, while ${\bf a}\odot {\bf b}$ denotes the symmetric matrix $\frac{1}{2}({\bf a} \otimes {\bf b} + {\bf b} \otimes {\bf a})$. 
We may also write
\begin{equation*}
\MA= \sum_{i,j=1}^N A_{ij} {\bf e}_i \otimes {\bf e}_j\,.
\end{equation*}

The following identities (see, e.g., \cite[Section 3]{Gurtin1973}) easily follow from the definitions and the properties of the tensor and scalar products:
\begin{equation}
\MA : ({\bf v} \odot {\bf w})=\MA : ({\bf v} \otimes {\bf w}) = {\bf v} \cdot (\MA {\bf w}) = {\bf w} \cdot (\MA {\bf v}) \,,
\label{eq:ident12}
\end{equation}
for every ${\bf v}, {\bf w}\in\R^N$ and $\MA \in \mathbb \R^{N \times N}_{\rm sym}$. 

For every ${\MA} \in \mathbb R^{N \times N}$ let us denote by $\MA^D$ the deviatoric part of $\MA$, defined by
\begin{equation}
\MA^D = \MA - \frac{1}{N}({\rm tr}\MA){\bf I} \,,
\label{eq:dev}
\end{equation}
where ${\bf I}$ is the identity matrix.
Recalling that the Frobenius scalar product is bilinear, and that $\MA:{\bf I}={\rm tr}\,\MA$ (so that, in particular, ${\bf I}:{\bf I}=N$), we easily deduce from \eqref{eq:dev} the well-known identity (see, e.g., \cite[eq. (7.41)]{Temam})
\begin{equation}
\MA^D:{\bf B}^D = \MA:{\bf B} - \frac{1}{N}{\rm tr}\,\MA\,{\rm tr}\,{\bf B} \,. 
\label{eq:identitydev}
\end{equation}

\subsection*{\emph{Measures}} \label{sec:measures} The following definitions and basic facts about measures can be found, e.g., in \cite[Chapters 1 and 2]{AFP}.

We denote by
$\Leb{N}$ 
and $\Haus{\alpha}$
the Lebesgue measure 
and the $\alpha$-dimensional 
Hausdorff measure in $\R^N$ for some $\alpha \in [0, N]$, respectively. Unless otherwise stated, a measurable set is a $\Leb{N}$-measurable set. 
It is well known that $\Haus{N-1}=\Leb{N-1}$ on each hyperplane $\Omega^{\bm\xi}\subset \R^N$ (see \cite[eq. (2.37)]{AFP}).
Unless otherwise specified, $\Omega \subseteq \mathbb{R}^N$ is an open set. 

Following the notation of \cite{AFP}, we denote by $\mathcal{M}_{\rm loc}(\Omega;\R^m)$ and $\mathcal{M}(\Omega;\R^m)$ the spaces of vector-valued Radon measures and finite Radon measures on $\Omega$, respectively. In the case of scalar measures ($m=1$), we will shorten the notation and use the symbols $\mathcal{M}_{\rm loc}(\Omega)$ and $\mathcal{M}(\Omega)$. 

For matrix valued measures we will use the notation $\mathcal{M}_{\rm loc}(\Omega;\R^{N\times N})$, $\mathcal{M}(\Omega;\R^{N\times N})$, $\mathcal{M}_{\rm loc}(\Omega;\R^{N\times N}_{\rm sym})$ and $\mathcal{M}(\Omega;\R^{N\times N}_{\rm sym})$.

Given open sets $E \subset \mathbb{R}^N$ and $F \subset \mathbb{R}^M$, a positive Radon measure $\mu$ on $E$, and a family $x \mapsto {\bm\nu}_x$ of finite $\mathbb{R}^m$-valued Radon measures on $F$ such that $x \mapsto {\bm\nu}_x(B)$ is $\mu$-measurable for every Borel set $B\subset F$ and $\int_{E'} |{\bm\nu}_x|(F)\, d\mu(x) < +\infty$ for every open set $E' \Subset E$, we denote by $\mu \otimes_{\mathcal{M}} {\bm\nu}_x$ their \emph{generalized product}, that is, the $\mathbb{R}^m$-valued Radon measure on $E\times F$ defined by $(\mu \otimes_{\mathcal{M}} {\bm\nu}_x)(B):=\int_E {\bm\nu}_x(\{y\in F:\ (x,y)\in B\})\,d\mu(x)$ for every Borel set $B\subset K\times F$, with $K\subset E$ compact (see \cite[Definition 2.27]{AFP}). If $f$ is a positive or negative Borel function, or if $f \in L^1(E \times F, \mu \otimes_{\mathcal{M}} {\bm\nu}_x)$, then
\begin{equation}
\int_{E \times F} f(x,y)\, d(\mu \otimes_{\mathcal{M}} {\bm\nu}_x)(x,y)
=
\int_E \left( \int_F f(x,y)\, d{\bm\nu}_x(y) \right) d\mu(x).
\label{eq:disintegrationf}
\end{equation}

We denote by $L^1_\mu(\Omega; \mathbb{R}^N)$ the space of $\mathbb{R}^N$-valued functions that are integrable with respect to $\mu$ on $\Omega$, and by $L^1_{\mu,\mathrm{loc}}(\Omega; \mathbb{R}^N)$ the space of functions that belong to $L^1_\mu(K; \mathbb{R}^N)$ for every compact set $K \subset \Omega$. In the scalar-valued case, we omit the target space and simply write $L^1(\Omega, \mu)$ and $L^1_{\mathrm{loc}}(\Omega, \mu)$.

\subsection*{\emph{Approximate limits}} \label{sec:approxlimits} The following basic definitions and results can be found, e.g., in \cite[Sections 3.6 and 4.5]{AFP}.

For $\bfu\in L^1_{\rm loc}(\Omega;\R^N)$ we denote by $S_\bfu$ the \emph{approximate discontinuity set} of $\bfu$, i.e., the set of points of $\Omega$ which are not Lebesgue points of $\bfu$, and by $\tilde\bfu(x)$ the approximate limit of $\bfu$ at $x\in\Omega\setminus S_\bfu$; by Lebesgue's differentiation theorem, $\Leb{N}(S_\bfu)=0$. A point $x\in\Omega$ is an \emph{approximate jump point} of $\bfu$, and we write $x\in J_\bfu$, if there exist ${\bf a}\neq{\bf b}$ in $\R^N$ and a unit vector $\bm\nu\in\R^N$ such that the averages of $|\bfu-{\bf a}|$ on $B_r^+(x):=\{y\in B_r(x):\ (y-x)\cdot\bm\nu>0\}$ and of $|\bfu-{\bf b}|$ on $B_r^-(x):=\{y\in B_r(x):\ (y-x)\cdot\bm\nu<0\}$ tend to $0$ as $r\to0^+$; the triplet $({\bf a},{\bf b},\bm\nu)$, uniquely determined up to a permutation of $({\bf a},{\bf b})$ and a change of sign of $\bm\nu$, is denoted by $(\bfu^+(x),\bfu^-(x),\bm\nu_\bfu(x))$. Clearly $J_\bfu\subset S_\bfu$; moreover, $J_\bfu$ and $S_\bfu$ are Borel sets, and $\tilde\bfu$ and $\bfu^\pm$ are Borel functions. Finally, the \emph{precise representative} $\bfu^*(x)$ is the limit, as $r\to0^+$, of the averages of $\bfu$ on $B_r(x)$, whenever it exists, so that
\begin{equation}
\bfu^*(x)=
\begin{cases}
\tilde{\bfu}(x) & x\in \Omega \setminus S_{\bfu}, \\
\displaystyle \frac{\bfu^+(x)+ \bfu^-(x)}{2} & x\in J_{\bfu}.
\end{cases}
\label{eq:preciserepresentative}
\end{equation} 
A priori, it is not clear whether $\bfu^*$ is well posed in $S_{\bfu} \setminus J_{\bfu}$, in general. However, for sufficiently regular functions it is known that $S_{\bfu} \setminus J_{\bfu}$ is suitably small (for instance, when $\bfu$ is a function of bounded variation, see \cite[Theorem 3.78]{AFP}). 

Let $\bfu\in L^1(\Omega;\R^N)$, and consider the sequence 
\begin{equation}
\bfu_\varepsilon:= \eta_\varepsilon * \bfu = (\eta_\varepsilon * u_1, \eta_\varepsilon * u_2, \dots, \eta_\varepsilon * u_N)\,, 
\label{eq:mollification}
\end{equation}
where $\eta_\varepsilon(x):= \frac{1}{\varepsilon^N} \eta(\frac{x}{\varepsilon})$ and $\eta$ is a positive symmetric mollifier. Then, combining \cite[Proposition 3.64(b) and Proposition 3.69(b)]{AFP} we get
\begin{equation}
\bfu_\varepsilon \to \bfu^* \quad \mbox{ pointwise in $\Omega\backslash(S_\bfu\backslash J_\bfu)$ as $\varepsilon \to 0$.}
\label{eq:convtoprecise}
\end{equation}

All the previous definitions extend verbatim to the setting of matrix-valued functions.

\subsection*{\emph{Slicing of functions of bounded variation}} \label{sec:BV} 
For a detailed treatment of the theory of $BV$ functions, we refer the reader to the monograph \cite{AFP}.
The vector space of all functions of bounded variation in \(\Omega\)
will be denoted by \(BV(\Omega;\R^N)\). In addition, \(BV_{\rm{loc}}(\Omega;\R^N)\) will denote the class of functions belonging to \(BV(\Omega';\R^N)\) for every $\Omega'\Subset\Omega$.

Let us fix $\bm\xi\in\R^{N}$ and $\bm\xi^\perp$ the hyperplane orthogonal to $\bm\xi$.
For every $E\subset \R^N$  and for every $y\in\bm\xi^\perp$ we denote  $E^{\bm\xi}_y:=\{s\in\R: y+s{\bm\xi}\in E\}$ and by  $E^{\bm\xi}:=\{y\in \bm\xi^\perp: E^{\bm\xi}_y\not=\emptyset\}$.
Given $u:\Omega\to\R$, for any ${\bm\xi}\in\R^N\setminus \{0\}$ and
$y\in\bm\xi^\perp$, we define
\[
\uxz(s) := u(y+s{\bm\xi}) \quad \text{for }s\in \Omega\xz.
\] 
If $\bfu:\Omega\to\R^N$, for any ${\bm\xi}\in\R^N\setminus \{0\}$ and
$y\in\bm\xi^\perp$, we define
\[
\hat{u}_y^{\bm\xi}(s) := \bfu(y+s{\bm\xi})\cdot \bm\xi \quad \text{for }s\in \Omega\xz.
\]

We first state a well-known result on the slicing of directional derivatives (see \cite[Theorem 3.103 and 3.107]{AFP}).

\begin{proposition}
Let $v \in L^1(\Omega)$ and let $\bm\xi \in \mathbb{S}^{N-1}$. The following conditions are equivalent:
\begin{enumerate}
\item[$(a)$] $D_{\bm\xi} v \in \mathcal{M}(\Omega)$;
\item[$(b)$] for $\Leb{N-1}$-a.e. $y \in \Omega^{\bm\xi}$ the function $v_y^{\bm\xi}$ belongs to $BV(\Omega_y^{\bm\xi})$ and
\begin{equation*}
\int_{\Omega^{\bm\xi}} |D v_y^{\bm\xi}|(\Omega_y^{\bm\xi}) \, d\Leb{N-1}(y) < \infty.
\end{equation*}
\end{enumerate}
If these conditions are satisfied, then
\begin{equation}
D_{\bm\xi} v = Dv \cdot \bm{\xi}
= \Leb{N-1}\!\res \Omega^{\bm{\xi}} \otimes_{\mathcal{M}} Dv_y^{\bm{\xi}} \,;
\label{eq:slicdirectder}
\end{equation}
i.e., 
\begin{equation}
\int_\Omega \varphi d D_{\bm\xi} v = \int_{\Omega^{\bm\xi}} \left(\int_{\Omega_y^{\bm\xi}}\varphi_y^{\bm\xi} dDv_y^{\bm{\xi}} \right) \, d\Leb{N-1}(y)\,, \quad \mbox{ for every $\varphi\in C_c(\Omega)$.}
\label{eq:slicdirectder2}
\end{equation}
Moreover, if $f\in L^1(\Omega,|D_{\bm\xi} v|)$, then
\begin{equation}
\int_\Omega f d D_{\bm\xi} v = \int_{\Omega^{\bm\xi}} \left(\int_{\Omega_y^{\bm\xi}}f_y^{\bm\xi} dDv_y^{\bm{\xi}} \right) \, d\Leb{N-1}(y)\,. 
\label{eq:slicdirectder3}
\end{equation}
\label{prop:slicingdirectd}
\end{proposition}

The following lemma adapts the argument in the proof of \cite[Proposition 3.2]{AmbCosDalM} to functions whose distributional derivative is a Radon measure in a prescribed direction. We use the pointwise variation $Vf$ of a function $f$ defined on a subset of full measure of an interval, see Theorem~\ref{structure} below and \cite[Section 2]{AmbCosDalM}.

\begin{lemma}\label{t:directionalpointwise}
Let $v\in L^1_{\rm loc}(\Omega)$ and $\bm\xi\in\mathbb{S}^{N-1}$ be such that $D_{\bm\xi}v$ is a Radon measure in $\Omega$, and let $\tilde v$ be the Lebesgue representative of $v$, defined in $\Omega\setminus S_v$. Then, for $\Leb{N-1}$-a.e.\ $y\in\Omega^{\bm\xi}$, we have $v^{\bm\xi}_y\in BV_{\rm loc}(\Omega^{\bm\xi}_y)$ and
\begin{equation}\label{f:directionalpointwise}
(v^{\bm\xi}_y)^+(t)=(v^{\bm\xi}_y)^-(t)=\tilde v(y+t\bm\xi)\qquad\text{for every }t\in\Omega^{\bm\xi}_y\setminus J_{v^{\bm\xi}_y}\text{ such that }y+t\bm\xi\notin S_v .
\end{equation}
Consequently, if $\lambda=\Leb{N-1}\res\Omega^{\bm\xi}\otimes_{\mathcal M}\lambda_y$ is a positive Radon measure in $\Omega$ such that $\lambda(S_v)=0$ and $\lambda_y$ has no atoms for $\Leb{N-1}$-a.e.\ $y$, then, for $\Leb{N-1}$-a.e.\ $y$, $(\tilde v)^{\bm\xi}_y=(v^{\bm\xi}_y)^*$ $\lambda_y$-a.e.\ in $\Omega^{\bm\xi}_y$.
\end{lemma}
\begin{proof}
Let $\psi\in C^\infty_c(\Omega)$ with $0\le\psi\le1$, and let $w:=\psi v$, extended by $0$ outside $\Omega$. Then $w\in L^1(\R^N)$ and $D_{\bm\xi}w=\psi D_{\bm\xi}v+v\,\partial_{\bm\xi}\psi\,\Leb{N}$ is a finite measure in $\R^N$. For every $y\in\bm\xi^\perp$ let $T_y:=\{t\in\R:\ y+t\bm\xi\notin S_w\}$, and let $V$ denote the pointwise variation on $\R$ computed using only points of $T_y$. The argument in the proof of \cite[Proposition 3.2]{AmbCosDalM}, based on regularization, on the lower semicontinuity of $V$ with respect to pointwise convergence on $T_y$, on Fatou's lemma and on \cite[(2.9)]{AmbCosDalM}, uses only the fact that $D_{\bm\xi}(\bfu\cdot\bm\xi)=E\bfu\,\bm\xi\cdot\bm\xi$ is a finite measure. Applied to $w$, together with \eqref{eq:slicdirectder} and \cite[Corollary 2.29]{AFP}, it shows that, for $\Leb{N-1}$-a.e.\ $y$, the function $\tilde w^{\bm\xi}_y$ is defined $\Leb1$-a.e., coincides a.e.\ with $w^{\bm\xi}_y$, and $V(\tilde w^{\bm\xi}_y)=|Dw^{\bm\xi}_y|$ as measures on $\R$. For such $y$, let $t\in T_y\setminus J_{w^{\bm\xi}_y}$. Since $\tilde w^{\bm\xi}_y$ has finite pointwise variation and coincides a.e.\ with $w^{\bm\xi}_y$, its one-sided limits at $t$ along $T_y$ exist and coincide with those of $w^{\bm\xi}_y$, which have a common value $c$. As $V(\tilde w^{\bm\xi}_y)(\{t\})\ge2|\tilde w^{\bm\xi}_y(t)-c|$ and $V(\tilde w^{\bm\xi}_y)(\{t\})=|Dw^{\bm\xi}_y|(\{t\})=0$, we obtain $\tilde w(y+t\bm\xi)=c$. Applying this to a sequence $\psi_k$ with $\psi_k=1$ on $\Omega_k$, where $(\Omega_k)$ is an increasing sequence of open sets with $\Omega_k\Subset\Omega$ and $\bigcup_k\Omega_k=\Omega$, we obtain \eqref{f:directionalpointwise}, since $w=v$, $S_w=S_v$ and $\tilde w=\tilde v$ in $\Omega_k$.

The last assertion follows, since $\lambda_y\big((S_v)^{\bm\xi}_y\big)=0$ for $\Leb{N-1}$-a.e.\ $y$, the set $J_{v^{\bm\xi}_y}$ is countable and $\lambda_y$ has no atoms.
\end{proof}

\subsection*{\emph{Functions of bounded deformation}} \label{sec:BD} We collect some definitions and basic results about functions of bounded deformation. We refer the reader to \cite{Temam} for a comprehensive treatment on the topic. 

We denote by $BD(\Omega)$ the space of \emph{functions of bounded deformation} in $\Omega$, {\it i.e.}, $\bfu\in BD(\Omega)$ if $\bfu \in L^1(\Omega;\R^N)$ and $E\bfu\in \mathcal{M} (\Omega;\R^{N \times N}_{\rm sym})$, where $E\bfu:=(D\bfu+D\bfu^T)/2$ and $D\bfu$ is the distributional derivative of $\bfu$.  In addition, \(BD_{\rm{loc}}(\Omega)\) will denote the class of functions belonging to \(BD(\Omega')\) for every $\Omega'\Subset\Omega$.
  In general we have that $J_{\bfu} \subset S_{\bfu}$. 
While for a function $\bfu\in BV(\Omega;\R^N)$ we have
$\Haus{N-1}(S_\bfu  \setminus J_\bfu) = 0$,
it is not known whether the same property holds for every $\bfu\in BD(\Omega)$ (see \cite[Remark 6.2]{AmbCosDalM}).
However, by \cite[Theorem 6.1]{AmbCosDalM}, the set $S_{\bfu}\setminus J_{\bfu}$ of discontinuity points of $\bfu$ which are not jump points is $|E{\bf v}|$-negligible for every ${\bf v}\in BD(\Omega)$; i.e., $|E {\bf v}|$-almost every $x\in\Omega$ is either a Lebesgue point or a jump point for $\bfu$.

The properties of the trace operator on open sets with Lipschitz boundaries (see also, e.g., \cite[Chapter II]{Temam} for the analogous classical result on $C^1$ boundaries) are collected in the theorem below, which is due to Babadjian \cite{Babadjian}. 

\begin{theorem}{\cite[Theorem 3.2 and Proposition 3.4]{Babadjian}}\label{Babadjian}
Let $\Omega \subset \mathbb{R}^N$ be a bounded open set with Lipschitz boundary. There exists a unique linear continuous mapping 
\begin{equation*}
\gamma : BD(\Omega) \to L^1(\partial \Omega; \mathbb{R}^N)
\end{equation*}
such that the following integration by parts formula holds: for every $\bfu \in BD(\Omega)$ and $\varphi \in C^1_c(\mathbb{R}^N)$,
\begin{equation}\label{GGbab}
\int_\Omega \bfu \odot \nabla \varphi \, dx + \int_\Omega \varphi \, dE\bfu = \int_{\partial \Omega}\gamma(\bfu) \odot  \bm\nu_\Omega\,\varphi \, d\Haus{N-1},
\end{equation}
where $\bm\nu_\Omega$ is the outer unit normal to $\partial \Omega$. 
In addition, 
\[
\gamma(\bfu) = \bfu|_{\partial \Omega} \quad \text{for all } \bfu \in C^0(\overline{\Omega}; \mathbb{R}^N) \cap BD(\Omega).
\]
\end{theorem}

\begin{remark}
\label{rem:extension}
Notice that formula \eqref{GGbab} implies the following
\begin{equation}\label{eq:ident12345}
\int_{\Omega} \bfu \cdot \DIV{\bm{\Phi}}\,dx + \int_{\Omega}{\bm{\Phi}}:\,dE\bfu
= \int_{\partial \Omega} {\bm{\Phi}}: (\gamma(\bfu) \odot \bm\nu_{\Omega}) \,d\Haus{N-1},
\end{equation}
with \({\bm{\Phi}}\in C^1(\R^N;\mathbb{R}^{N\times N}_{{\rm sym}})\); note that, since $\Omega$ is bounded, \eqref{GGbab} holds for every $\varphi\in C^1(\R^N)$, as one sees by multiplying $\varphi$ by a cut-off function in $C^1_c(\R^N)$ equal to $1$ in a neighbourhood of $\overline\Omega$. Indeed, it is enough to write \eqref{GGbab} componentwise with $\varphi=\varphi_{ij}$, where $\bm\Phi=(\varphi_{ij})$, to sum over $i,j$, and to use the symmetry $\varphi_{ij}=\varphi_{ji}$, which gives $\sum_{i,j}\int_\Omega(\bfu\odot\nabla\varphi_{ij})_{ij}\,dx=\sum_{i}\int_\Omega u_i\sum_j\partial_j\varphi_{ij}\,dx=\int_\Omega\bfu\cdot\DIV\bm\Phi\,dx$.
\end{remark}

The existence of one-sided Lebesgue limits on countably $\Haus{N-1}$-rectifiable subsets of $\Omega$ for functions of bounded deformation is proved in the following result.

\begin{proposition}{\cite[Proposition 4.1]{Babadjian}}\label{prop:rectifiable}
Let $\bfu \in BD(\Omega)$, and $\Sigma$ be an oriented countably $\Haus{N-1}$-rectifiable subset of $\Omega$. Then, for $\Haus{N-1}$-almost every $x \in \Sigma$, there exist one-sided Lebesgue limits $\bfu_{\Sigma}^{\pm}(x)$ with respect to the approximate unit normal $\bm\nu_{\Sigma}(x)$ to $\Sigma$; that is,
\[
\lim_{\rho \to 0^+} \frac1{\rho^N}\int_{B_{\rho}^{\pm}(x, \bm\nu_{\Sigma}(x))} |\bfu(y) - \bfu_{\Sigma}^{\pm}(x)| \, dy = 0,
\]
where $B_{\rho}^{\pm}(x, \bm\nu_{\Sigma}(x)) := \{ y \in B_{\rho}(x): \,\, \pm (y - x) \cdot \bm\nu_{\Sigma}(x) > 0 \}$. In addition, we have the representation 
\begin{equation*}
E\bfu \res \Sigma = (\bfu_{\Sigma}^{+} - \bfu_{\Sigma}^{-}) \odot \bm\nu_{\Sigma} \, \Haus{N-1} \res \Sigma \,.
\end{equation*}
\end{proposition}

Since for every $\bfu\in BD(\Omega)$ the jump set $J_\bfu$ is a countably $\Haus{N-1}$-rectifiable Borel set (see \cite[Proposition 3.5]{AmbCosDalM}), by Proposition \ref{prop:rectifiable} applied with $\Sigma=J_\bfu$, we obtain the following decomposition of the Radon measure $E \bfu$
\begin{equation}
E \bfu =  e(\bfu) \Leb{N} + E^s \bfu = e(\bfu) \Leb{N} + ([\bfu^+-\bfu^-] \odot \bm\nu_\bfu) \Haus{N-1} \res J_\bfu + E^c \bfu,
\end{equation}
where $e(\bfu)$ is the density of the absolutely continuous part of $E \bfu$ with respect to $\Leb{N}$, $E^s \bfu$ is the singular part, $E^c \bfu$ is the Cantor part and vanishes on Borel sets $B$ with $\Haus{N-1}(B) < +\infty$.
We also introduce the space $LD(\Omega)$ of such functions $\bfu\in BD(\Omega)$ complying with
\begin{equation}\label{mmmm}
E \bfu =  e(\bfu) \Leb{N}.
\end{equation}

Recalling the notation of Section~\ref{sec:preliminaries} (§~\hyperref[sec:BV]{\emph{Slicing of functions of bounded variation}}), the space $BD(\Omega)$ can be characterized using one-dimensional sections (see~\cite[Section~3]{AmbCosDalM}, \cite[Chap. II, Section 2.2]{Temam} and \cite[Proposition 2.26 and Theorem 2.28]{AFP}). The slicing properties of $BD$ functions that we shall use are collected in the following theorem (see \cite[Proposition 3.2 and Theorem 4.5]{AmbCosDalM}), where, for $\bm\xi\in\R^N$, we set
\[
J^{\bm\xi}_\bfu:=\{x\in J_\bfu:\ (\bfu^+(x)-\bfu^-(x))\cdot\bm\xi\neq0\}.
\]
\begin{theorem}\label{structure}
Let $\bfu \in BD(\Omega)$ and let $\bm\xi \in \mathbb{S}^{N-1}$. Then the following properties hold for $\Haus{N-1}$-almost every $y \in \Omega^{\bm\xi}$:
\begin{itemize}
\item[(i)] $\tilde{\hat u}^{\bm\xi}_y$ is defined and coincides with $\hat{u}_y^{\bm\xi}$ $\Leb{1}$-almost everywhere in $\Omega^{\bm\xi}_y$, where $\tilde{\hat u}^{\bm\xi}_y(t):=\tilde\bfu(y+t\bm\xi)\cdot\bm\xi$ is the one-dimensional section of the Lebesgue representative $\tilde\bfu$, defined for every $t\in\Omega^{\bm\xi}_y$ such that $y+t\bm\xi\notin S_\bfu$;
\item[(ii)] $\hat{u}_y^{\bm\xi} \in BV(\Omega^{\bm\xi}_y)$ (and so $\hat{u}_y^{\bm\xi}\in L^\infty_{\rm{loc}}(\Omega^{\bm\xi}_y)$), the measure $|D\hat u^{\bm\xi}_y|$ coincides on $\Omega^{\bm\xi}_y$ with the pointwise variation $V\tilde{\hat u}^{\bm\xi}_y$, and $e(\bfu)(y + t\bm\xi)\bm\xi \cdot \bm\xi = \nabla \hat{u}_y^{\bm\xi}(t)$ for $\Leb{1}$-a.e.\ $t \in \Omega^{\bm\xi}_y$. Here, for a function $v$ defined on a set $E\subset\Omega^{\bm\xi}_y$ of full measure, $Vv(I):=\sup\big\{\sum_{i=1}^{k-1}|v(t_{i+1})-v(t_i)|:\ t_1<\dots<t_k,\ t_i\in I\cap E\big\}$ for every open interval $I\subset\Omega^{\bm\xi}_y$, and $Vv$ is extended to a Borel measure as in \cite[Section~2]{AmbCosDalM};
\item[(iii)] $(J_{\bfu}^{\bm\xi})^{\bm\xi}_y = J_{\hat{u}_y^{\bm\xi}}$ and, for every $t \in (J_{\bfu}^{\bm\xi})_y^{\bm\xi}$,
\[
\bfu^\pm(y + t\bm\xi) \cdot \bm\xi = (\hat{u}_y^{\bm\xi})^\pm(t) = \lim_{s \to t^\pm} \tilde{\hat{u}}_y^{\bm\xi}(s),
\]
where the normals to $J_{\bfu}$ and $J_{\hat{u}_y^{\bm\xi}}$ are oriented so that $\bm\nu_{\bfu} \cdot \bm\xi \geq 0$ and $\nu_{\hat{u}_y^{\bm\xi}} = 1$.
\end{itemize}
Moreover,
\begin{equation}\label{eq:slicefq}
E\bfu\,\bm\xi \cdot \bm\xi = \int_{\Omega^{\bm\xi}} D\hat{u}_y^{\bm\xi} \, d\Haus{N-1}(y),
\qquad
\big|E\bfu\,\bm\xi \cdot \bm\xi\big| = \int_{\Omega^{\bm\xi}} |D\hat{u}_y^{\bm\xi}| \, d\Haus{N-1}(y),
\end{equation}
\begin{equation}
 E^c \bfu\, \bm\xi \cdot \bm\xi 
= \int_{\Omega^{\bm\xi}} D^c \hat{u}_y^{\bm\xi} \, d\Haus{N-1}(y), 
\qquad
| E^c \bfu\, \bm\xi \cdot \bm\xi | 
= \int_{\Omega^{\bm\xi}} |D^c \hat{u}_y^{\bm\xi}| \, d\Haus{N-1}(y),
\label{eq:disintEc}
\end{equation}
as measures in $\Omega$. In particular, for every bounded Borel function $g:\Omega\to\R$ the map
\begin{equation}\label{measurability}
y\mapsto \int_{\Omega^{\bm\xi}_y} g^{\bm\xi}_y(s)\,dD\hat{u}_y^{\bm\xi}(s)
\end{equation}
is $\Haus{N-1}$-measurable on $\Omega^{\bm\xi}$.
\end{theorem}

The following pointwise properties of the slices are contained in \cite{AmbCosDalM}: assertion $(i)$ follows from Theorem~\ref{structure}$(ii)$ and $(iii)$, while assertion $(ii)$ extends Theorem~\ref{structure}$(iii)$ from $J^{\bm\xi}_\bfu$ to $J_\bfu$, by the localization argument used in the proof of \cite[Theorem 4.5]{AmbCosDalM}.

\begin{lemma}\label{t:BDpointwise}
Let $\bfu\in BD(\Omega)$ and $\bm\xi\in\mathbb{S}^{N-1}$. Then for $\Leb{N-1}$-a.e.\ $y\in\Omega^{\bm\xi}$ the following properties hold:
\begin{itemize}
\item[$(i)$] for every $t\in\Omega^{\bm\xi}_y$ such that $y+t\bm\xi\notin S_\bfu$, we have $t\notin J_{\hat u^{\bm\xi}_y}$ and $(\hat u^{\bm\xi}_y)^+(t)=(\hat u^{\bm\xi}_y)^-(t)=\tilde\bfu(y+t\bm\xi)\cdot\bm\xi$;
\item[$(ii)$] for every $t\in\Omega^{\bm\xi}_y$ such that $x:=y+t\bm\xi\in J_\bfu$ and ${\bm\nu}_\bfu(x)\cdot\bm\xi>0$, we have $(\hat u^{\bm\xi}_y)^\pm(t)=\bfu^\pm(x)\cdot\bm\xi$, where $(\hat u^{\bm\xi}_y)^+(t)$ and $(\hat u^{\bm\xi}_y)^-(t)$ denote the right and left limits of $\hat u^{\bm\xi}_y$ at $t$;
\item[$(iii)$] $(\hat u^{\bm\xi}_y)^*(t)=\bfu^*(y+t\bm\xi)\cdot\bm\xi$ for every $t\in\Omega^{\bm\xi}_y$ such that $y+t\bm\xi\notin S_\bfu\setminus J_\bfu$.
\end{itemize}
Moreover, let $\Sigma\subset\Omega$ be an oriented countably $\Haus{N-1}$-rectifiable Borel set and let $\bfu^\pm_\Sigma$ be the traces of $\bfu$ on $\Sigma$ given by Proposition~\ref{prop:rectifiable}. Then, for $\Leb{N-1}$-a.e.\ $y\in\Omega^{\bm\xi}$,
\begin{itemize}
\item[$(iv)$] $(\hat u^{\bm\xi}_y)^\pm(t)=\bfu^\pm_\Sigma(y+t\bm\xi)\cdot\bm\xi$ for every $t\in\Sigma^{\bm\xi}_y$, where now the one-sided limits of $\hat u^{\bm\xi}_y$ are taken with respect to the orientation $\nu_{\Sigma^{\bm\xi}_y}(t)={\rm sign}\big({\bm\nu}_\Sigma(y+t\bm\xi)\cdot\bm\xi\big)$, that is, $(\hat u^{\bm\xi}_y)^+(t)$ is the limit from the side towards which $\nu_{\Sigma^{\bm\xi}_y}(t)$ points.
\end{itemize}
\end{lemma}
\begin{proof}
$(i)$ By Theorem~\ref{structure}$(iii)$, $J_{\hat u^{\bm\xi}_y}=(J^{\bm\xi}_\bfu)^{\bm\xi}_y\subset(S_\bfu)^{\bm\xi}_y$, so that $\hat u^{\bm\xi}_y$ has no jump at $t$. By Theorem~\ref{structure}$(ii)$, $\tilde{\hat u}^{\bm\xi}_y$ is defined at $t$ and $V\tilde{\hat u}^{\bm\xi}_y=|D\hat u^{\bm\xi}_y|$, hence $V\tilde{\hat u}^{\bm\xi}_y(\{t\})=0$, which forces $\tilde{\hat u}^{\bm\xi}_y(t)$ to coincide with the common value of the one-sided limits of $\hat u^{\bm\xi}_y$ at $t$.

$(ii)$ For $x\in J^{\bm\xi}_\bfu$ this is Theorem~\ref{structure}$(iii)$. For $x\in J_\bfu\setminus J^{\bm\xi}_\bfu$, the localization argument leading to (5.12)--(5.16) in the proof of \cite[Theorem 4.5]{AmbCosDalM} (with ${\bf v}=\bfu\chi_{\Omega_k}$, where $\Omega_k$ are bounded open sets with $C^1$ boundary such that $\bigcup_k\partial\Omega_k$ covers $\Haus{N-1}$-almost all of $J_\bfu$) applies verbatim to the points of $J_\bfu\cap\partial^\pm_{\bm\xi}\Omega_k$, where $\partial^\pm_{\bm\xi}\Omega_k:=\{x\in\partial\Omega_k:\ \pm{\bm\nu}_{\Omega_k}(x)\cdot\bm\xi>0\}$: it only uses that ${\bf v}$ vanishes on one side of $\partial\Omega_k$ and distinguishes the cases $\bfu^\mp(x)\cdot\bm\xi\neq0$ and $\bfu^\mp(x)\cdot\bm\xi=0$, never the condition $(\bfu^+(x)-\bfu^-(x))\cdot\bm\xi\neq0$. It gives $(\hat u^{\bm\xi}_y)^\mp(t)=\bfu^\mp(x)\cdot\bm\xi$ on $\partial^\pm_{\bm\xi}\Omega_k$; since $t\notin J_{\hat u^{\bm\xi}_y}$ by Theorem~\ref{structure}$(iii)$ and $(\bfu^+(x)-\bfu^-(x))\cdot\bm\xi=0$, both one-sided limits are equal to $\bfu^\pm(x)\cdot\bm\xi$.

$(iii)$ follows from $(i)$, $(ii)$ and \eqref{eq:preciserepresentative}, since, by the coarea formula \cite[Theorem 2.93]{AFP}, for $\Leb{N-1}$-a.e.\ $y$ the line $y+\R\bm\xi$ does not meet the points of $J_\bfu$ with ${\bm\nu}_\bfu\cdot\bm\xi=0$.

$(iv)$ By \cite[Section 3, property (ii)]{AmbCosDalM}, for $\Haus{N-1}$-a.e.\ $x\in\Sigma$ the traces $\bfu^\pm_\Sigma(x)$ are one-sided Lebesgue limits of $\bfu$ at $x$ with respect to ${\bm\nu}_\Sigma(x)$. Hence either $\bfu^+_\Sigma(x)=\bfu^-_\Sigma(x)=\tilde\bfu(x)$ and $x\notin S_\bfu$, or $x\in J_\bfu$ and $(\bfu^+_\Sigma,\bfu^-_\Sigma,{\bm\nu}_\Sigma)(x)=(\bfu^+,\bfu^-,{\bm\nu}_\bfu)(x)$ up to a change of sign of the normal and a permutation of the one-sided limits. Since the exceptional set is $\Haus{N-1}$-negligible, its projection onto $\bm\xi^\perp$ is $\Leb{N-1}$-negligible, and $(iv)$ follows from $(i)$, $(ii)$ and the coarea formula as in $(iii)$.
\end{proof}

We say that $\bfu\in BD(\Omega)$ is a \emph{special function with bounded deformation}, and we write $\bfu\in SBD(\Omega)$, if the measure $E^c \bfu$ is zero; that is, $|E^s \bfu|(\Omega\setminus J_{\bfu})=0$. $SBD^p(\Omega)$, $p>1$, is defined as the space of $SBD$ functions such that $e(\bfu)\in L^p(\Omega; \R^{N\times N}_{\rm sym})$ and $\Haus{N-1}(J_\bfu)<+\infty$.

\subsection*{\emph{Vector and matrix fields with divergence measure}} \label{sec:vectmatrfields}

We introduce the spaces of bounded vector and matrix fields whose distributional divergence is a finite Radon measure in $\Omega$.  Namely,
\begin{equation*}
\DM(\Omega;\R^N) :=\left\{ \A \in L^\infty(\Omega;\R^N) : \Div\A\in \mathcal M(\Omega)\right\},
\end{equation*}
and
\begin{equation}
\begin{split}
\DMA(\Omega;\mathbb{R}^{N\times N}) & :=\left\{ \MA \in L^\infty(\Omega;\mathbb{R}^{N\times N}) : \DIV\MA\in \mathcal M(\Omega;\R^N)\right\}\,, \\
\DMA(\Omega;\mathbb{R}^{N\times N}_{\rm sym}) & :=\left\{ \MA \in L^\infty(\Omega;\mathbb{R}^{N\times N}_{\rm sym}) : \DIV\MA\in \mathcal M(\Omega;\R^N)\right\}\,.
\end{split}
\end{equation}
Here, $\DIV \MA$ denotes the distributional divergence of $\MA$, defined componentwise by
\[
\DIV \MA := (\Div \MA_1, \dots, \Div \MA_N),
\]
where $\MA_j$ is the $j$-th column of $\MA$ and
\[
\Div \MA_j := \sum_{i=1}^N \frac{\partial A_{ij}}{\partial x_i}, \qquad j = 1, \dots, N.
\]
Consistently with this convention, for a vector-valued function $\bfu$ we denote by $D\bfu$ (or $\nabla\bfu$) the matrix with entries $(D\bfu)_{ij}=D_iu_j$, so that the $j$-th column of $D\bfu$ is $Du_j$ and $\MA:D\bfu=\sum_{j=1}^N\MA_j\cdot Du_j$; in the same way, $(\nabla\bm\varphi)_{ij}=\partial_i\varphi_j$ for $\bm\varphi\in C^1_c(\Omega;\R^N)$. For symmetric matrix fields $\MA$, which are the main object of this paper, this convention plays no role.

We will denote by $\mathcal{DM}^\infty_{\rm loc}(\Omega;\R^N)$, $\DMAloc(\Omega;\mathbb{R}^{N\times N})$ and $\DMAloc(\Omega;\mathbb{R}^{N\times N}_{\rm sym})$ the corresponding local spaces.

We recall that, if $\A\in\DM(\Omega;\R^N)$, then $|\Div\A|(B)=0$ for every Borel set $B\subset\Omega$ with $\Haus{N-1}(B)=0$.

For matrix-valued fields, we recall that, for every Borel set $B$, the estimate
\begin{equation}
|\DIV \MA|(B) \leq \sum_{i=1}^N |\Div \MA_i|(B)
\label{eq:totvariationineq}
\end{equation}
holds (see, e.g., \cite[p.~4]{AFP}).

The following result shows that the divergence measure of $\MA$ is absolutely continuous with respect to $\Haus{N-1}$, see \cite[Proposition 2.5]{DCSTensor}.

\begin{proposition}\label{ABS}
Let $\MA \in \DMA(\Omega;\mathbb{R}^{N\times N})$, and let $B\subset\Omega$ be a Borel measurable set such that $\Haus{N-1}(B)=0$. Then $|\DIV \MA|(B)=0$. 
\label{prop:absolutecont}
\end{proposition}

As a consequence, the set
\begin{equation*}\label{f:jump}
\jump{\MA} 
:= \left\{
x\in\Omega:\
\limsup_{r \to 0^+}
\frac{|\DIV \MA| (B_r(x))}{r^{N-1}} > 0
\right\},
\end{equation*} 
is a Borel set, \(\sigma\)-finite with respect to \(\Haus{N-1}\). Moreover,
the measure \(\DIV \MA\) admits the decomposition
\[
\DIV\MA = \DIV^a\MA + \DIV^c\MA + \DIV^j\MA,
\]
where \(\DIV^a\MA\ll\Leb{N}\),
\(\DIV^c\MA (B) = 0\) for every Borel set \(B\) with \(\Haus{N-1}(B) < +\infty\),
and $\DIV^j\MA \ll \Haus{N-1}\res\jump{\MA}$. 

The following lemma provides an approximation result for stress fields. 
\begin{lemma}
Given $\MA\in \DMA(\Omega;\mathbb{R}^{N\times N})$, there exists $\{\MA_k\}\subset C^\infty({\Omega};\mathbb{R}^{N\times N})$ such that 
\begin{enumerate}
\item[$(i)$] $\MA_k \to \MA$ in $L^1(\Omega;\mathbb{R}^{N\times N})$;
\item[$(ii)$]  $\|\MA_k\|_{L^\infty(\Omega;\mathbb{R}^{N\times N})}\leq 3\sqrt{N}\, \|\MA\|_{L^\infty(\Omega;\mathbb{R}^{N\times N})}$ for every $k$;
\item[$(iii)$] $\displaystyle \sup_{k\in\mathbb{N}} \int_\Omega |\DIV\MA_k|\, dx <+\infty$; 
\item[$(iv)$] $\DIV\MA_k \rightharpoonup^* \DIV\MA$ as Radon measures on $\Omega$.
\end{enumerate}
\label{lem:approxim}
\end{lemma}
\begin{proof}
Applying the approximation theorem of Chen-Frid \cite[Theorem 1.2]{ChenFrid}
to each column $\MA_j$ of $\MA$, $j=1,\dots, N$, we find $\MA_{j,k} \in C^\infty(\Omega;\mathbb{R}^N)$ such that
$\MA_{j,k}\to \MA_j$ in $L^1$ and 
$\int_\Omega |\Div \MA_{j,k}|\,dx \to |\Div \MA_j|(\Omega)$ for each $j$.
Define $\MA_k := (\MA_{1,k},\dots,\MA_{N,k})$. 
Then $(i)$ and $(ii)$ follow immediately, while $(iii)$ is a consequence of \eqref{eq:totvariationineq} and the fact that $$\displaystyle\sup_{k\in\mathbb{N}}\int_\Omega |\Div \MA_{j,k}|\,dx <+\infty\,.$$ For what concerns $(iv)$, from $(i)$ we have
\begin{equation*}
\lim_{k\to+\infty}\int_{\Omega} \MA_k: \nabla \bm\varphi \,dx = \int_\Omega \MA: \nabla \bm\varphi \,dx \quad \mbox{for every $\bm\varphi\in C^1_c(\Omega;\R^N)$,}
\end{equation*}  
so that the assertion follows from a standard density argument of $C_c^1(\Omega;\R^N)$ in $C_c(\Omega;\R^N)$ with respect to the $L^\infty(\Omega;\R^N)$ norm, and the bound $(iii)$.
\end{proof}

\subsection*{\emph{Normal traces}} \label{sec:distrtraces}

We first recall the definition of traces of the normal component of a vector field \(\C\in \DM(\Omega;\R^N)\), which can be defined as distributions
\(\Trace{\C}{\Sigma}\)
on every oriented countably \(\Haus{N-1}\)-rectifiable set
\(\Sigma\subset\Omega\)
(see \cite{AmbCriMan,Anz,ChenFrid}).
According to \cite{AmbCriMan}, given a domain \(\Omega'\Subset\Omega\) of class \(C^1\), the trace of the normal component of \(\C\) on \(\partial\Omega'\) is defined in the distributional sense by
\begin{equation}\label{f:disttr}
\pscal{\Trace[]{\C}{\partial\Omega'}}{\varphi}
:= \int_{\Omega'} \C\cdot \nabla\varphi\, dx + \int_{\Omega'} \varphi\, d\Div\C,
\qquad
\forall\varphi\in C^\infty_c(\Omega).
\end{equation}
It is shown in \cite[Proposition 3.2]{AmbCriMan} that this distribution is in fact induced by an \(L^\infty\) function on \(\partial\Omega'\),
which we continue to denote by \(\Trace[]{\C}{\partial\Omega'}\), i.e.,
\begin{equation}\label{eq:identification}
\pscal{\Trace[]{\C}{\partial\Omega'}}{\varphi}
=\int_{\partial\Omega'}\varphi \Trace[]{\C}{\partial\Omega'}\,d\Haus{N-1},
\end{equation}
and satisfies the estimate
\[
\|\Trace[]{\C}{\partial\Omega'}\|_{L^\infty(\partial\Omega')}
\leq \|\C\|_{L^\infty(\Omega';\R^N)}.
\]
Now, since \(\Sigma\) is oriented and countably $\Haus{N-1}$-rectifiable,
we can find countably many {oriented} \(C^1\) hypersurfaces \(\Sigma_i\),
with classical normal vectors \(\bm\nu_{\Sigma_i}\),
and pairwise disjoint Borel sets \(E_i\subseteq \Sigma_i\)
such that \(\Haus{N-1}(\Sigma\setminus \bigcup_i E_i) = 0\).
Moreover, we may assume, without loss of generality, that for each $i$ there exist two open, bounded sets \(\Omega_i, \Omega'_i\) with \(C^1\) boundaries
and exterior normal vectors \(\bm\nu_{\Omega_i}\) and \(\bm\nu_{\Omega_i'}\) respectively,
such that
\(E_i\subseteq \partial\Omega_i \cap \partial\Omega'_i\)
and
\[
\bm\nu_{\Sigma_i}(x) = \bm\nu_{\Omega_i}(x) = -\bm\nu_{\Omega'_i}(x)
\qquad \forall x\in E_i.
\]
We then define the orientation on \(\Sigma\) by setting
\(\bm\nu_{\Sigma}(x) := \bm\nu_{\Sigma_i}(x)\) for
\(\Haus{N-1}\)-a.e.\ $x\in E_i$.
Using the localization property established in \cite[Proposition 3.2]{AmbCriMan},
the normal traces of \(\C\) on \(\Sigma\) are defined by
\begin{equation}
\Trm{\C}{\Sigma} := \Tr(\C, \partial\Omega_i),
\quad
\Trp{\C}{\Sigma} := -\Tr(\C, \partial\Omega'_i),
\qquad
\Haus{N-1}-\text{a.e.\ on}\ E_i.
\label{f:traces}
\end{equation}
These two normal traces belong to
\(L^{\infty}(\Sigma, \Haus{N-1}\res\Sigma)\) (see \cite[Proposition 3.2]{AmbCriMan})
and satisfy
\begin{equation}\label{mmm}
\Div \C \res\Sigma =
\left[\Trp{\C}{\Sigma} - \Trm{\C}{\Sigma}\right]
\, \Haus{N-1} \res\Sigma\,.
\end{equation}

For a matrix field $\MA\in\DMA(\Omega;\mathbb{R}^{N\times N})$, the normal traces are defined column by column. Namely, given a domain $\Omega'$ as above, we set
\[
\Tracev[]{\MA}{\partial\Omega'}:= \big(\Trace[]{\MA_1}{\partial\Omega'},\dots, \Trace[]{\MA_N}{\partial\Omega'}\big),
\]
so that, by \eqref{f:disttr},
\begin{equation}\label{f:disttrv}
\pscal{\Tracev[]{\MA}{\partial\Omega'}}{\bm \varphi}
=\int_{\Omega'} \MA:\nabla\bm\varphi \, dx + \int_{\Omega'} \bm\varphi \cdot d\DIV\MA
\qquad
\forall \bm\varphi\in C^\infty_c(\Omega;\R^N);
\end{equation}
this distribution is induced by a function in $L^\infty(\partial\Omega';\R^N)$, still denoted by $\Tracev[]{\MA}{\partial\Omega'}$, with $\|\Tracev[]{\MA}{\partial\Omega'}\|_{L^\infty(\partial\Omega';\mathbb{R}^{N})}\leq \|\MA\|_{L^\infty(\Omega';\mathbb{R}^{N\times N})}$, which satisfies the same localization property as in the scalar case. Accordingly, the normal traces of $\MA$ on $\Sigma$ are defined by
\begin{equation}
\Trmv{\MA}{\Sigma} := \Trv(\MA, \partial\Omega_i),
\quad
\Trpv{\MA}{\Sigma} := -\Trv(\MA, \partial\Omega'_i),
\qquad
\Haus{N-1}-\text{a.e.\ on}\ E_i\,,
\label{eq:tracesvc}
\end{equation}
which extends \cite[Definition 3.3]{AmbCriMan} to matrix-valued fields. These traces belong to the space $L^{\infty}_{\Haus{N-1}\res\Sigma}(\Sigma; \R^N)$ and, applying \cite[Proposition 3.4(ii)]{AmbCriMan} to each column, we obtain
\begin{equation}\label{mmmv}
\DIV^j \MA \res\Sigma =
\left[\Trpv{\MA}{\Sigma} - \Trmv{\MA}{\Sigma}\right]
\, \Haus{N-1} \res\Sigma\,.
\end{equation}

\subsection*{\emph{A brief review for the pairing}} \label{sec:reviewpairing}
In \cite{CD3} (see also \cite{Anz, ChenFrid}), the pairing between a vector field $\A \in \DM(\Omega)$ and the measure gradient of a scalar function $u \in BV(\Omega)$ is defined as the linear functional $(\A, Du) : C^\infty_c(\Omega) \to \mathbb{R}$ given by
\begin{equation}\label{f:pairing}
\pscal{(\A, Du)}{\varphi} :=
-\int_\Omega u^*\varphi\, d \Div \A - \int_\Omega u \, \A\cdot \nabla\varphi\, dx,
\end{equation}
under the assumption $u^*\in L^1(\Omega,|\Div\A|)$.

The distribution \((\A, Du)\) is a Radon measure in \(\Omega\),
absolutely continuous with respect to \(|Du|\)
(see \cite[Theorem 4.12]{CD3}),
hence the identity
\begin{equation}\label{f:anz}
\Div(u\A) = u^* \Div\A + (\A, Du)
\end{equation}
holds in the sense of measures in \(\Omega\).

The following representation result can be found in \cite[Theorem 4.12 and Remark 3.4]{CD3} and \cite[Theorem 3.5]{CCDM}.

\begin{theorem}\label{t:pairingCD3}
Let \(\A\in \mathcal{DM}^\infty(\Omega;\R^N)\) and $u\in BV_{\rm loc}(\Omega)$ with $u^* \in L^1_{\rm loc}(\Omega, |\Div \A|)$.
Then the measure \((\A, Du)\) admits the following decomposition:
\begin{itemize}
	\item[(i)]
	absolutely continuous part: 
	\((\A, Du)^a = \A \cdot \nabla u\, \Leb{N}\);
	
	\item[(ii)]
	jump part:
	\(\displaystyle
	(\A, Du)^j = 
	\frac{\Trp{\A}{J_u}+\Trm{\A}{J_u}}{2}
	\, (u^+-u^-) \, \hh \res J_u
	\);

	\item[(iii)]
	Cantor part:
		\begin{equation*}\label{eq:cantorpart1}
	(\A, Du)^c\res (\Omega\setminus S_{\A})= \widetilde{\A} \cdot D^c u\res (\Omega\setminus S_{\A}),
	\end{equation*}
	where \(S_{\A}\) is the approximate discontinuity set  of \(\A\).	
\end{itemize}
In particular, if \begin{equation}
|D^c u|(S_{\A})=0\,, 
\label{eq:Dcu}
\end{equation} then
\begin{equation*}\label{eq:cantorpart}
	(\A, Du)^c= \widetilde{\A} \cdot D^c u, \ \ 	{\text\ in\ }\Omega.	
	\end{equation*} 
	Moreover, if $N=1$, the above assumption is not needed, since in this case $\A\in BV_{\rm loc}(\Omega)$, the set $S_{\A}=J_{\A}$ is at most countable, and $D^cu$ has no atoms.
	Finally, if \(\A\in BV(\Omega;\R^N)\cap L^\infty(\Omega;\R^N)\), (and so $\eqref{eq:Dcu}$ is satisfied), then
	$$
	(\A,Du)=\A^*\cdot Du,
	$$
	in the sense of measures.
	\end{theorem}
	
We shall use the Gauss--Green formulas of \cite[Theorem 5.1]{CD3} only in dimension one, where they read as follows.
\begin{theorem}\label{t:GGN=1}
Let $a,u\in BV_{{\rm loc}}(\R)$ and assume that $u^*\in L^1_{\rm{loc}}(\R,|Da|)$.
Let \( F\subset \R\) be a bounded set of finite perimeter (i.e., a finite union of bounded intervals).
Then the
following Gauss--Green formulas hold:
\begin{gather}
\int_{F^1} u^* \, dDa + \int_{F^1} (a, u') = -
\sum_{s\in \partial^* F} a^+(s)\nu(s)u^+(s)\,,\label{GreenIBN=1} \nonumber
\\
\int_{F^1\cup \partial ^*F} u^* \, dDa + \int_{F^1\cup \partial ^*F} (a, u') =
-\sum_{s\in \partial^* F} a^-(s)\nu(s)u^-(s)\,.\label{GreenIB2N=1} \nonumber
\end{gather}
Here $F^1$ denotes the measure-theoretic interior of $F$, $\Tr^+(a,\partial^* F)=a^+(s)\nu(s)$ and \\ $\Tr^-(a,\partial^* F)=a^-(s)\nu(s)$ are the normal traces of \(a\) and for every $s\in \partial^* F$,  $\nu(s)$
is the interior unit normal.
Here, for every $s\in\partial^*F$, $a^+(s)$ and $u^+(s)$ denote the one-sided limits of $a$ and $u$ at $s$ from the side of $F$ (i.e., from the side towards which $\nu(s)$ points), while $a^-(s)$ and $u^-(s)$ denote the one-sided limits from the opposite side. More explicitly, if $F=\bigcup_{k=1}^m(\alpha_k,\beta_k)$ with $\alpha_1<\beta_1<\alpha_2<\dots<\alpha_m<\beta_m$, then $\partial^*F=\{\alpha_1,\beta_1,\dots,\alpha_m,\beta_m\}$, $\nu(\alpha_k)=1$, $\nu(\beta_k)=-1$, and
\[
a^\pm(\alpha_k)=\lim_{t\to\alpha_k^\pm}a(t)\,,\qquad a^\pm(\beta_k)=\lim_{t\to\beta_k^\mp}a(t)\,,
\]
and similarly for $u$.
\end{theorem}

\section{Slicing properties of vector fields with measure-valued diagonal derivatives}
\label{sec:slicingdive}

In this section we introduce the class of essentially bounded vector fields with measure-valued diagonal derivatives, namely vector fields whose diagonal components of the distributional Jacobian are Radon measures. We investigate their slicing properties and show that the measure-valued diagonal derivatives admit a one-dimensional representation in terms of the distributional derivatives of the corresponding slices. This yields a slicing formula for the divergence measure of the vector field and, as a consequence, a representation formula for the normal trace on oriented rectifiable sets.

Let $\A\in L^1(\Omega;\R^N)$, and assume that 
\begin{equation}
D_iA^i\in \mathcal M(\Omega) \mbox{ for every } i\in\{1,\dots,N\}\,.
\tag{H}
\label{eq:conditionH}
\end{equation}
In this case we have $$\Div\A=\sum_{i=1}^ND_iA^i \in \mathcal M(\Omega)\,.$$

\begin{remark}
By Proposition \ref{prop:slicingdirectd}, applied with $v=A^i$ and $\bm\xi={\bf e}_i$,
condition \eqref{eq:conditionH} is equivalent to the following slicing characterization: for every $i\in\{1,\dots,N\}$ 
\begin{equation}
\begin{cases}
(A^i)_y^{{\bf e}_i}\in BV(\Omega^{{\bf e}_i}_y)\, \mbox{ for $\Leb{N-1}$-a.e.  $y\in \Omega^{{\bf e}_i}$},\\
\\
 \displaystyle \int_{\Omega^{{\bf e}_i}} \big| D (A^i)_y^{{\bf e}_i} \big|(\Omega_y^{{\bf e}_i}) \, d\Leb{N-1}(y) < +\infty\,.
\end{cases}
\tag{$\rm \tilde{H}$}
\label{eq:conditionoldH}
\end{equation}
This implies that $(A^i)_y^{{\bf e}_i}\in  L^\infty_{\rm{loc}}(\Omega^{{\bf e}_i}_y)$ for $\Leb{N-1}$-a.e.  $y\in \Omega^{{\bf e}_i}$.
\label{rem:equivalence}
\end{remark}

We introduce the class of essentially bounded vector fields satisfying \eqref{eq:conditionH}:
\begin{equation}
\label{eq:newspace2}
BV^\infty_{\mathrm{diag}}(\Omega;\R^{N}):=\left\{\A\in L^\infty(\Omega;\R^{N}):
D_iA^i\in \mathcal M(\Omega) \mbox{ for every } i\in\{1,\dots,N\}
\right\}.
\end{equation}
The notation ``diag'' emphasizes that only the diagonal entries of the distributional Jacobian are required to be Radon measures, while no information is assumed on the off-diagonal derivatives $D_j A^i$, $i\neq j$.

It is immediate to check that
$$BV(\Omega;\R^{N})\cap L^\infty(\Omega;\R^{N})\subset 
BV^\infty_{\mathrm{diag}}(\Omega;\R^{N})\subset\DM(\Omega;\R^{N}).$$

\begin{remark}
The second inclusion above is, in general, strict. For instance, let $\Omega:=(-1,1)^2$ and $\A=\A(x_1,x_2):=(f(x_1-x_2),f(x_1-x_2))$, where 
\begin{equation*}
f(s):= 
\begin{cases}
\sin(\frac{1}{s})\,, & \mbox{ if } s\neq 0\,, \\
0 \,,  & \mbox{ if } s=0 \,.
\end{cases}
\end{equation*}
Then, $\Div\A=0\in \mathcal{M}(\Omega)$, whereas $D_1A^1 \not\in \mathcal{M}(\Omega)$ and $D_2A^2 \not\in \mathcal{M}(\Omega)$. Indeed, $f\not\in BV(-1,1)$. Moreover, for every $x_2\in(-1,1)$, $(A^1)_{x_2}^{{\bf e}_1}(t)=f(t-x_2) = \sin(\frac{1}{t-x_2})$ for $t\neq x_2$, hence $(A^1)_{x_2}^{{\bf e}_1}\not \in BV(-1,1)$. Therefore, the first condition in \eqref{eq:conditionoldH} fails.
\end{remark}

We first prove the following representation formula for the divergence measure.

\begin{proposition}\label{p:slicingdiv}
Let $\A\in BV^\infty_{\mathrm{diag}}(\Omega;\R^{N})$. 
Then for every $\varphi\in C_c(\Omega)$ we have
 \begin{equation}
\int_\Omega \varphi\, d\Div\A = \sum_{i=1}^N \int_{\Omega^{{\bf e}_i}}
\left(
\int_{\Omega_{y }^{{\bf e}_i}} \varphi_y^{{\bf e}_i}(s)\, d\,D((A^i)_y^{{\bf e}_i})(s)
\right)\,d\Leb{N-1}(y)\,;
\label{eq:stima2}
\end{equation}
i.e., 
\begin{equation}
\Div\A =  \sum_{i=1}^N\Leb{N-1}\res\Omega^{{\bf e}_i}\otimes_{\mathcal{M}} D((A^i)_y^{{\bf e}_i})\res\Omega^{{\bf e}_i}_y.
\label{eq:stima2bis}
\end{equation}
as measures on $\Omega$. 
\end{proposition}
\proof 
For every $i=1,\dots,N$, by assumption \eqref{eq:conditionH} and \eqref{eq:slicdirectder}, applied with $v=A^i$ and $\bm\xi={\bf e}_i$, we get
\begin{equation*}
D_iA^i = \Leb{N-1}\res\Omega^{{\bf e}_i}\otimes_{\mathcal{M}} D((A^i)_y^{{\bf e}_i})\res\Omega^{{\bf e}_i}_y
\end{equation*}
as measures on $\Omega$. Summing over \(i=1,\dots,N\), and recalling that
\(\Div \A = \displaystyle \sum_{i=1}^N D_i A^i\),
we deduce \eqref{eq:stima2bis}. Formula \eqref{eq:stima2} follows by testing against \(\varphi \in C_c(\Omega)\).
\endproof

We next derive a slicing formula for the normal trace of fields in $BV^\infty_{\mathrm{diag}}(\Omega;\R^N)$.

\begin{proposition}
Let $\A\in BV^\infty_{\mathrm{diag}}(\Omega;\R^{N})$. 
Let $\Omega'\Subset\Omega$ be an open set with 
\(C^1\) boundary.
Then   for every test function  $\varphi\in C_c(\Omega)$ 
\begin{equation}\label{eq:tracesl}
\int_{\partial\Omega'}\varphi \Trace[]{\A}{\partial\Omega'}\,d\Haus{N-1}
= -
 \sum_{i=1}^{N}\int_{{\bf e}_i^\perp}
\left( 
\sum_{s\in\partial(\Omega')^{{\bf e}_i}_y} 
((A^i)_y^{{\bf e}_i})^+(s)\,\nu_{(\Omega')^{{\bf e}_i}_y}(s)\,\varphi_y^{{\bf e}_i}(s)
\right)
\,d\Leb{N-1}(y)
\,,
\end{equation}
where $\nu_{(\Omega')^{{\bf e}_i}_y}\in \{-1, 1\}$ is the normal of $\partial(\Omega')^{{\bf e}_i}_y$ oriented with respect to the interior unit normal vector
and satisfies the relation
\begin{equation}
\nu_{(\Omega')^{{\bf e}_i}_y}(s) = -\operatorname{sign}\left( {\bm\nu}_{\Omega'}(y + s{{\bf e}_i}) \cdot {{\bf e}_i} \right).
\label{eq:normal_sign_slicing}
\end{equation}
The sign convention for  a
  one-dimensional slice interval $(a,b)\subset (\Omega')^{{\bf e}_i}_y$ is the following: the interior normal is
  $1$ at $a$ and $-1$ at $b$; moreover, $((A^i)_y^{{\bf e}_i})^+(s)$ denotes the one-sided limit of $(A^i)_y^{{\bf e}_i}$ at $s$ from the side of $(\Omega')^{{\bf e}_i}_y$, that is, from the side towards which $\nu_{(\Omega')^{{\bf e}_i}_y}(s)$ points.

\end{proposition}
\proof
First, we assume that $\varphi\in C^1_c(\Omega)$. In order to prove \eqref{eq:tracesl}, we recall that by \eqref{f:disttr} and \eqref{eq:identification} we have  
\begin{equation}\label{f:disttr1}
\int_{\partial\Omega'}\varphi \Trace[]{\A}{\partial\Omega'}\,d\Haus{N-1}
= \int_{\Omega'} \A\cdot \nabla\varphi\, dx + \int_{\Omega'} \varphi\, d\Div\A.
\end{equation}
For every fixed $i\in\{1,2,\dots,N\}$ and for $\Leb{N-1}$-a.e. $y\in {\bf e}_i^\perp$, integrating by parts (see \cite[Theorem 6.3]{CDCM}) we have
\begin{align*}
\int_{(\Omega')_{y }^{{\bf e}_i}} (\varphi_y^{{\bf e}_i})'(s)(A^i)_y^{{\bf e}_i}(s)\, ds = &-\int_{(\Omega')_{y }^{{\bf e}_i}} \varphi_y^{{\bf e}_i}(s)\, d\,D((A^i)_y^{{\bf e}_i})(s)
\\ &- 
 \sum_{s\in\partial(\Omega')^{{\bf e}_i}_y} \,\varphi_y^{{\bf e}_i}(s)
((A^i)_y^{{\bf e}_i})^+(s)\,\nu_{(\Omega')^{{\bf e}_i}_y}(s)
\,.
\end{align*}
Furthermore, by Fubini's Theorem,
\begin{equation}
\int_{\Omega'}  A^i\,\partial_i \varphi\,dx =\int_{{\bf e}_i^\perp}
\left(\int_
{(\Omega')^{{\bf e}_i}_y} 
(A^i)_y^{{\bf e}_i}(s)(\varphi_{y }^{{\bf e}_i})'(s)\,ds
\right)\,d\Leb{N-1}(y)\,.
\label{eq:stima1}
\end{equation}
Now, integrating over ${\bf e}_i^\perp$ and summing up $i$, taking into account \eqref{eq:stima1} and \eqref{eq:stima2} we obtain
$$
\int_{\Omega'} \A\cdot \nabla\varphi\, dx = -  \int_{\Omega'} \varphi\, d\Div\A
- \sum_{i=1}^{N}\int_{{\bf e}_i^\perp}
\left( 
\sum_{s\in\partial(\Omega')^{{\bf e}_i}_y} 
((A^i)_y^{{\bf e}_i})^+(s)\,\nu_{(\Omega')^{{\bf e}_i}_y}(s)\,\varphi_y^{{\bf e}_i}(s)
\right)
\,d\Leb{N-1}(y)
\,,
$$
whence, in view of \eqref{f:disttr1}, \eqref{eq:tracesl} follows.  
For a general $\varphi\in C_c(\Omega)$, using the coarea formula \cite[Theorem 2.93]{AFP}, 
we find
\begin{equation}
\begin{split}
&\left|\sum_{i=1}^{N}\int_{{\bf e}_i^\perp} 
\left(\sum_{s\in \partial(\Omega')^{{\bf e}_i}_y} 
((A^i)^{{\bf e}_i}_y)^+(s)\,\nu_{(\Omega')^{{\bf e}_i}_y}(s)\varphi_y^{{\bf e}_i}(s)
\right)d\Leb{N-1}(y) \right| \\
& \,\,\,\,\,\, \leq \|\A\|_{L^\infty(\Omega;\R^N)} \|\varphi\|_{L^\infty(\Omega)}\sum_{i=1}^{N}\int_{{\bf e}_i^\perp}
\#(\partial(\Omega')^{{\bf e}_i}_y)\,d\Leb{N-1}(y) \\
&  \,\,\,\,\,\, = \|\A\|_{L^\infty(\Omega;\R^N)} \|\varphi\|_{L^\infty(\Omega)}\sum_{i=1}^{N}\int_{\partial \Omega'}
|\nu_{ \Omega'}\cdot {\bf e}_i|\,d\Haus{N-1}(y) \\
& \,\,\,\,\,\, \leq  N \|\A\|_{L^\infty(\Omega;\R^N)} \Haus{N-1}(\partial\Omega') \|\varphi\|_{L^\infty(\Omega)}\,,
\end{split}
\label{eq:stimacont}
\end{equation}
so that the conclusion follows by a standard density argument. 
\endproof

\begin{proposition}\label{plplo}
Let $\A\in BV^\infty_{\mathrm{diag}}(\Omega;\R^{N})$ and
let \(\Sigma\subset\Omega\) be an oriented countably \(\Haus{N-1}\)-rectifiable set with unit normal ${\bm\nu}_\Sigma$, such that $\Haus{N-1}(\Sigma\cap K)<+\infty$ for every compact set $K\subset\Omega$.
Then  for every test function  $\varphi\in C_c(\Omega)$ 
we have
\begin{equation}
 \int_{\Sigma}\,  \varphi\Tr^\pm(\A,\Sigma)
 \, d\Haus{N-1}  =\sum_{i=1}^{N}\int_{{\bf e}_i^\perp}
\left(\sum_{s\in \Sigma^{{\bf e}_i}_{y }} 
((A^i)^{{\bf e}_i}_y)^\pm(s)\,\nu_{\Sigma^{{\bf e}_i}_{y }}(s)\varphi_y^{{\bf e}_i}(s)
\right)d\Leb{N-1}(y), 
\label{eq:traces}
\end{equation}
where $\nu_{\Sigma^{{\bf e}_i}_{y }}\in \{-1, 1\}$ is the normal of $\Sigma^{{\bf e}_i}_{y }$ oriented with respect to ${\bm\nu}_\Sigma$, $((A^i)^{{\bf e}_i}_y)^\pm(s)$ are the one-sided limits of $(A^i)^{{\bf e}_i}_y$ at $s$ with respect to this orientation (i.e., $((A^i)^{{\bf e}_i}_y)^+(s)$ is the limit from the side towards which $\nu_{\Sigma^{{\bf e}_i}_y}(s)$ points),
and satisfies the relation
\begin{equation*}
\nu_{\Sigma^{{\bf e}_i}_y}(s) = \operatorname{sign}\left({\bm\nu}_\Sigma(y + s{{\bf e}_i}) \cdot {{\bf e}_i} \right).
\label{eq:normal_sign_slicing2}
\end{equation*}
\end{proposition}
\proof
The assertion \eqref{eq:traces} follows by combining \eqref{eq:tracesl}
and \eqref{f:traces}. Both sides of \eqref{eq:traces} are well defined, since $\Haus{N-1}(\Sigma\cap{\rm supp}\,\varphi)<+\infty$ and, by the coarea formula \cite[Theorem 2.93]{AFP}, the absolute value of the right-hand side is bounded by $N\|\A\|_{L^\infty(\Omega;\R^N)}\|\varphi\|_{L^\infty(\Omega)}\Haus{N-1}(\Sigma\cap{\rm supp}\,\varphi)$.
\endproof

We now turn to matrix-valued fields. Since the divergence of a matrix field is defined columnwise, the relevant assumption is that each column satisfies the condition \eqref{eq:conditionH}. 
More precisely, given $\MA\in L^1(\Omega;\R^{N\times N})$, we assume that 
\begin{equation}
D_iA_{ij}\in \mathcal M(\Omega) \mbox{ for every } i,j\in\{1,\dots,N\}\,.
\tag{H$^\prime$}
\label{eq:conditionH'}
\end{equation}
In this case we have $$(\DIV\MA)_j=\sum_{i=1}^ND_iA_{ij} \in \mathcal M(\Omega)\,.$$
We introduce the class 
\begin{equation*}
\label{eq:newspace1}
\BV^\infty_{\mathrm{diag}}(\Omega;\R^{N\times N}):=\left\{\MA\in L^\infty(\Omega;\R^{N\times N}):\ \ 
D_iA_{ij}\in \mathcal M(\Omega) \mbox{ for every } i,j\in\{1,\dots,N\}
\right\}.
\end{equation*}
In view of definition \eqref{eq:newspace2}, $\MA \in \BV^\infty_{\mathrm{diag}}(\Omega;\R^{N\times N})$ if and only if each column $\MA_j\in BV^\infty_{\mathrm{diag}}(\Omega;\R^{N})$, $j=1,\dots,N$.

It is immediate to check that
$$BV(\Omega;\R^{N\times N})\cap L^\infty(\Omega;\R^{N\times N})\subset 
\BV^\infty_{\mathrm{diag}}(\Omega;\R^{N\times N})\subset\DMA(\Omega;\R^{N\times N}).$$

Slicing formulas for the divergence and the normal traces of matrix fields in $\BV^\infty_{\mathrm{diag}}(\Omega;\R^{N\times N})$ can be directly obtained, column by column, from Propositions~\ref{p:slicingdiv} and~\ref{plplo}.

\section{Slicing formulas for pairings with $BV$ functions}\label{pimo}

In this section we establish slicing formulas for the pairing measure between divergence-measure fields and $BV$ functions. First, for vector fields belonging to $BV^\infty_{\mathrm{diag}}(\Omega;\mathbb R^N)$, we show that the pairing measure $(\A,Du)$ admits a representation in terms of the one-dimensional pairings associated with the corresponding slices. 

\begin{theorem} \label{slicBV} Let $\A\in BV^\infty_{\mathrm{diag}}(\Omega;\R^{N})$.  If $u\in BV(\Omega)$ with $u^*\in L^1(\Omega,|\Div\A|)$ and 
\eqref{eq:Dcu} holds,
then we have
\begin{equation*}
\int_{\Omega}\varphi d(\A,Du)
=\sum_{i=1}^{N}\int_{{\bf e}_i^\perp}
\left(\int_
{\Omega^{{\bf e}_i}_y} 
\varphi_y^{{\bf e}_i}(s) \,
d((A^i)_y^{{\bf e}_i}, Du_{y }^{{\bf e}_i})\,(s)
\right)\,d\Leb{N-1}(y)\,,
\label{eq:slicingpair}
\end{equation*}
for every $\varphi\in C_c(\Omega)$, i.e.
\begin{equation*}
(\A,Du)
=\sum_{i=1}^{N}(\Leb{N-1}\res \Omega^{{\bf e}_i})\otimes_{\mathcal{M}}
((A^i)_y^{{\bf e}_i},Du_{y }^{{\bf e}_i})\res(\Omega^{{\bf e}_i}_y)
\label{eq:slicingpair2}
\end{equation*}
as measures in $\Omega$. 
\end{theorem}

Since the pairing measure admits the standard decomposition into absolutely continuous, Cantor, and jump parts, it suffices to establish the result separately for each component. We observe that the only point where the additional assumption \eqref{eq:Dcu} is required is the slicing formula for the Cantor part.

\begin{proposition} Let $\A\in BV^\infty_{\mathrm{diag}}(\Omega;\R^{N})$.  If $u\in BV(\Omega)$ with $u^*\in L^1(\Omega,|\Div\A|)$, we have
\begin{equation}
\int_{\Omega}\varphi d(\A,Du)^a
=\sum_{i=1}^{N}\int_{{\bf e}_i^\perp}
\left(\int_
{\Omega^{{\bf e}_i}_y} 
\varphi_y^{{\bf e}_i}(s)
((A^i)_y^{{\bf e}_i},Du_{y }^{{\bf e}_i})^a\,ds
\right)\,d\Leb{N-1}(y)\,,
\label{eq:abspartpair}
\end{equation}
for every $\varphi\in C_c(\Omega)$. 
\end{proposition}
\proof
From Theorem \ref{t:pairingCD3}(i) and Fubini's Theorem we have
\begin{equation*}
\begin{split}
\int_{\Omega}\varphi d(\A,Du)^a &=\int_{\Omega}\varphi \A\cdot \nabla u\, dx
 =\sum_{i=1}^{N}\int_{\Omega}\varphi  A^i\,\partial_i u\,dx\, \\
& =\sum_{i=1}^{N}\int_{{\bf e}_i^\perp}
\left(\int_
{\Omega^{{\bf e}_i}_y} 
\varphi_y^{{\bf e}_i}(s)
(A^i)_y^{{\bf e}_i}(u_{y }^{{\bf e}_i})'\,ds
\right)\,d\Leb{N-1}(y)\,,
\end{split}
\end{equation*}
and \eqref{eq:abspartpair} follows again from Theorem \ref{t:pairingCD3} (i) with $N=1$.  
\endproof

\begin{proposition} Let $\A\in BV^\infty_{\mathrm{diag}}(\Omega;\R^{N})$.  If $u\in BV(\Omega)$ with $u^* \in L^1(\Omega,|\Div\A|)$, then for every $\varphi\in C_c(\Omega)$ we have
\begin{equation*}\label{mnbv}
 \int_{\Omega}\varphi d(\A,Du)^j 
=\sum_{i=1}^{N}\int_{{\bf e}_i^\perp} 
\left(\int_{J_{u_{y }^{{\bf e}_i}}} \varphi_y^{{\bf e}_i}(s)
d((A^i)^{{\bf e}_i}_y,Du_{y }^{{\bf e}_i})^j(s)
\right)  \, d\Leb{N-1}(y).
\end{equation*}
\end{proposition}
\proof
Let $v:=u^+-u^-$ on $J_u$, oriented by $\nu_u$, so that $\int_{J_u}|v|\,d\Haus{N-1}=|D^ju|(\Omega)<+\infty$. By Theorem~\ref{t:pairingCD3}$(ii)$,
\[
\int_\Omega\varphi\,d(\A,Du)^j=\int_{J_u}\varphi\,\frac{\Trp{\A}{J_u}+\Trm{\A}{J_u}}{2}\,v\,d\Haus{N-1}.
\]
For $k\in\mathbb N$ let $\Sigma_k:=\{x\in J_u:\ |v(x)|>1/k\}$, so that $\Haus{N-1}(\Sigma_k)\le k\,|D^ju|(\Omega)<+\infty$. By Proposition~\ref{plplo} with $\Sigma=\Sigma_k$, both sides of \eqref{eq:traces} are integrals with respect to finite Radon measures on $\Sigma_k$ (the right-hand side by the coarea formula \cite[Theorem 2.93]{AFP}), whose total variations are bounded by $N\|\A\|_{L^\infty}\Haus{N-1}\res\Sigma_k$; hence \eqref{eq:traces} holds with $\varphi$ replaced by any function in $L^1(\Sigma_k,\Haus{N-1})$, in particular by $\varphi\,v$. Letting $k\to\infty$, by dominated convergence on both sides (on the right-hand side we use again the coarea formula), we obtain
\[
\begin{split}
&\int_\Omega\varphi\,d(\A,Du)^j\\
&\quad=\sum_{i=1}^N\int_{{\bf e}_i^\perp}\left(\sum_{s\in(J_u)^{{\bf e}_i}_y}\varphi^{{\bf e}_i}_y(s)\,\frac{((A^i)^{{\bf e}_i}_y)^+(s)+((A^i)^{{\bf e}_i}_y)^-(s)}{2}\,\nu_{(J_u)^{{\bf e}_i}_y}(s)\,v(y+s{\bf e}_i)\right)d\Leb{N-1}(y),
\end{split}
\]
where $\nu_{(J_u)^{{\bf e}_i}_y}(s)={\rm sign}(\nu_u(y+s{\bf e}_i)\cdot{\bf e}_i)$. In the inner sum, $\frac12\big(((A^i)^{{\bf e}_i}_y)^+(s)+((A^i)^{{\bf e}_i}_y)^-(s)\big)=((A^i)^{{\bf e}_i}_y)^*(s)$, whatever the orientation. Moreover, by \cite[Theorem 3.108]{AFP} and the coarea formula, for $\Leb{N-1}$-a.e.\ $y$ the line $y+\R{\bf e}_i$ does not meet the points of $J_u$ with $\nu_u\cdot{\bf e}_i=0$, $J_{u^{{\bf e}_i}_y}=(J_u)^{{\bf e}_i}_y$, and $\nu_{(J_u)^{{\bf e}_i}_y}(s)\,v(y+s{\bf e}_i)=u^{{\bf e}_i}_y(s^+)-u^{{\bf e}_i}_y(s^-)$ for every $s\in(J_u)^{{\bf e}_i}_y$. Therefore, by Theorem~\ref{t:pairingCD3}$(ii)$ for $N=1$,
\[
\int_\Omega\varphi\,d(\A,Du)^j=\sum_{i=1}^N\int_{{\bf e}_i^\perp}\left(\int_{J_{u_{y }^{{\bf e}_i}}}\varphi_y^{{\bf e}_i}(s)\,d((A^i)^{{\bf e}_i}_y,Du_{y }^{{\bf e}_i})^j(s)\right)d\Leb{N-1}(y).
\]
\endproof

As a final step, we derive a slicing formula for the Cantor part of the pairing measure. 

\begin{proposition}  Let $\A\in BV^\infty_{\mathrm{diag}}(\Omega;\R^{N})$.  If $u\in BV(\Omega)$ with $u^* \in L^1(\Omega,|\Div\A|)$ and \eqref{eq:Dcu} holds, then for every $\varphi\in C_c(\Omega)$ we have
\begin{equation}\label{mnbvc}
\int_{\Omega}\varphi\,d(\A,Du)^c=
\sum_{i=1}^{N}\int_{{{\bf e}_i}^\perp} \left(\int_{\Omega^{{{\bf e}_i}}_y} \varphi_y^{{\bf e}_i}(s)\ d((A^i)_{y }^{{{\bf e}_i}},Du_{y }^{{{\bf e}_i}})^c(s) \right)\, d\Leb{N-1}(y).
\end{equation}
\end{proposition}
\begin{proof}
We recall that, since \eqref{eq:Dcu} holds, one has $(\A , Du)^c = \widetilde{\A} \cdot D^c u$ on $\Omega$, as stated in Theorem \ref{t:pairingCD3}(iii). Then, by \cite[Theorem 3.108]{AFP} for $D^c u\cdot {\bf e}_i$, we obtain
\begin{equation*}
\begin{split}
\int_{\Omega}\varphi d(\A,Du)^c &=\int_{\Omega}\varphi \widetilde{\A} \cdot dD^c u
 =\sum_{i=1}^N\int_{\Omega}\varphi  (\widetilde\A\cdot {\bf e}_i)\,d(D^c u\cdot {\bf e}_i)\, 
 \\
& 
=\sum_{i=1}^N\int_{{\bf e}_i^\perp}
\left(\int_
{\Omega^{{\bf e}_i}_y} 
\varphi_y^{{\bf e}_i}(s)
\widetilde{(A^i)_{y }^{{{\bf e}_i}}}(s)\,dD^c(u_{y }^{{\bf e}_i})(s)
\right)\,d\Leb{N-1}(y)\,,
\end{split}
\end{equation*}
where in the last equality we used, besides \cite[Theorem 3.108]{AFP}, that for $\Leb{N-1}$-a.e.\ $y$ the slice of $\widetilde{A^i}=\widetilde\A\cdot{\bf e}_i$ coincides with $\widetilde{(A^i)^{{\bf e}_i}_y}$ $|D^c u^{{\bf e}_i}_y|$-a.e. This follows from Lemma~\ref{t:directionalpointwise}, applied with $v=A^i$ (recall that $D_{{\bf e}_i}A^i\in\mathcal M(\Omega)$) and $\lambda=|D^c u\cdot{\bf e}_i|=\Leb{N-1}\res\Omega^{{\bf e}_i}\otimes_{\mathcal M}|D^c u^{{\bf e}_i}_y|$: the measures $D^c u^{{\bf e}_i}_y$ have no atoms, and $\lambda(S_{A^i})\le|D^c u|(S_\A)=0$ by \eqref{eq:Dcu}, since $S_{A^i}\subset S_\A$,
so that 
applying Theorem \ref{t:pairingCD3}(iii) in dimension one, we conclude that
\begin{equation*}
\begin{split}
\sum_{i=1}^N & \int_{{\bf e}_i^\perp}
\left(\int_
{\Omega^{{\bf e}_i}_y} 
\varphi_y^{{\bf e}_i}(s)
\widetilde{(A^i)_{y }^{{{\bf e}_i}}}(s)\,dD^c(u_{y }^{{\bf e}_i})(s)
\right)\,d\Leb{N-1}(y) \\
&=\sum_{i=1}^N\int_{{\bf e}_i^\perp}
 \left(\int_{\Omega^{{{\bf e}_i}}_y} \varphi_y^{{\bf e}_i}(s)\ d((A^i)_{y }^{{{\bf e}_i}},Du_{y }^{{{\bf e}_i}})^c(s) \right)\,d\Leb{N-1}(y)\,.
\end{split}
\end{equation*}
This proves \eqref{mnbvc}. 
\end{proof}

We now address the pairing between a matrix-valued field $\MA$ and the matrix derivative of a function in $BV(\Omega;\R^N)$, recently introduced in \cite{DCSTensor}. 

Let \(\MA\in\DMA(\Omega;\mathbb{R}^{N\times N})\) and $\bfu\in BV(\Omega;\R^N) $ be such that $\bfu^* \in  L^1_{|\DIV\MA|}(\Omega;\R^N)$. Then
the linear functional
\((\MA: D\bfu) \colon C^\infty_c(\Omega) \to \R\) defined by
\begin{equation}\label{eq:pairingvettoriale}
\pscal{(\MA:D\bfu)}{\varphi} :=
-\int_\Omega \varphi \bfu^*\cdot d \DIV\MA - \int_\Omega  \MA:[\nabla\varphi\otimes\bfu]\, dx
\end{equation}
is well-defined. Moreover, it is straightforward to check that
\begin{equation}
(\MA:D\bfu)=\sum_{j=1}^N(\MA_j, Du_j)\,,
\label{eq:pairingsum}
\end{equation}
where for every $j\in\{1,\dots,N\}$ $(\MA_j,Du_j)$ is the pairing \eqref{f:pairing} between the $j$-th column $\MA_j$ of $\MA$ and the $j$-th component $u_j$ of $\bfu$. Note that each scalar pairing $(\MA_j,Du_j)$ on the right-hand side is well defined, since the assumption $\bfu^* \in  L^1_{|\DIV\MA|}(\Omega;\R^N)$ implies $u_j^*\in L^1(\Omega; |\Div\MA_j|)$ for every $j\in\{1,\dots,N\}$.  

As an immediate consequence of the representation \eqref{eq:pairingsum}, the pairing \eqref{eq:pairingvettoriale} inherits the main properties of the scalar pairing \eqref{f:pairing}; in particular, it is a Radon measure in $\Omega$, absolutely continuous with respect to $|D\bfu|$ (see \cite[Theorems 3.2 and 3.7]{DCSTensor}).

We investigate whether an analogous slicing representation can be obtained.
Formula \eqref{eq:pairingsum} plays a fundamental role 
in this analysis, as it allows us to reduce the vectorial 
pairing to a finite sum of scalar pairings. Consequently, the desired 
slicing formula follows from 
Theorem \ref{slicBV}, provided that each column of $\MA$ satisfies assumption \eqref{eq:conditionH} stated above.

By applying Theorem \ref{slicBV} to every vector field $\MA_j$ and using \eqref{eq:pairingsum}, we obtain the following result. 

\begin{theorem} \label{th:sliceBVvector}
Let $\MA\in \BV^\infty_{\mathrm{diag}}(\Omega;\R^{N\times N})$. If ${\bfu}\in BV(\Omega;\R^N)$ with ${\bfu}^*\in L^1_{|\DIV\MA|}(\Omega;\R^N)$ and $|D^c\bfu|(S_{\MA})=0$, then we have
$$
\int_{\Omega}\varphi d(\MA:D\bfu)
=\sum_{i,j=1}^N\int_{{\bf e}_i^\perp}
\left(\int_
{\Omega^{{\bf e}_i}_y} 
\varphi_y^{{\bf e}_i}(s)
\,d((A_{ij})_y^{{\bf e}_i},D(u_j)_{y }^{{\bf e}_i})(s)
\right)\,d\Leb{N-1}(y)\,,
$$
for every $\varphi\in C_c(\Omega)$.
\end{theorem}

\section{A pairing for tensor fields and bounded $BD$ functions} \label{sec:pairingBDbounded}

We start by introducing a pairing in $BD$ for bounded functions. 
Given a symmetric matrix field \(\MA\in\DMA(\Omega;\R^{N \times N}_{\rm sym})\) and $\bfu \in BD(\Omega)\cap L^\infty(\Omega;\R^N)$, a natural definition would be the analogue of \eqref{eq:pairingvettoriale}, where the term $\MA : [\nabla \varphi\otimes\bfu]$ is replaced by $\MA : [\bfu \odot \nabla \varphi]$ in view of the symmetry of $\MA$. 
However, since it is not known whether $\Haus{N-1}(S_\bfu \setminus J_\bfu)=0$ for every $\bfu \in BD(\Omega)$, the term $\int_\Omega \varphi \, \bfu^* \cdot d\DIV \MA$ may fail to be well-defined.
This motivates the assumption
\begin{equation}\label{assum} \tag{${\rm H}_\ast$}
|\DIV^s \MA|(S_{\bfu}\setminus J_\bfu)=0.
\end{equation}
Here, $\DIV^s \MA$ denotes the singular part of the measure $\DIV \MA$. 
Since $\Leb{N}(S_{\bfu}\setminus J_\bfu)=0$, condition \eqref{assum} is equivalent to   
\begin{equation*}\label{assum1}
|\DIV \MA|(S_{\bfu}\setminus J_\bfu)=0.
\end{equation*}

\begin{remark}\label{rem:Hstar}
Assumption \eqref{assum} can be interpreted as a natural compatibility condition between the singular part of the divergence of the stress tensor $\MA$ and the fine structure of the displacement field $\bfu$. Indeed, a characteristic property of $BD$ functions ensures that
$
|E\bfu|(S_\bfu \setminus J_\bfu) = 0,
$
so that the strain measure does not detect the non-jump part of the singular set of $\bfu$. In this perspective, condition \eqref{assum} requires that the singular part of $\DIV\MA$ does not charge the same set, namely the portion of $S_\bfu$ that is invisible to the strain. This ensures that concentrated forces are confined to the physically relevant jump set $J_\bfu$, where a well-defined notion of interfacial interaction is available, thereby guaranteeing a consistent coupling between kinematic and static singularities.
\\
Assumption \eqref{assum} is satisfied, for instance, in the following cases:
\begin{enumerate}
\item[(i)] $\bfu \in BV(\Omega;\R^N)$, by Proposition \ref{ABS}, since $\Haus{N-1}(S_\bfu \setminus J_\bfu)=0$;
\item[(ii)] $|\DIV \MA| \ll \Leb{N}$;
\item[(iii)] $\bfu \in SBD^p(\Omega;\R^N)$ for some $p>1$, since also in this case $\Haus{N-1}(S_\bfu \setminus J_\bfu)=0$ (see \cite[Theorem 4.1]{CFI});
\item[(iv)] $|\DIV^s\MA|$ is concentrated on a countably $\Haus{N-1}$-rectifiable Borel set $\Sigma$ (for instance, an interface or a crack surface), for every $\bfu\in BD(\Omega)$. Indeed, by \cite[Section 3, property (ii)]{AmbCosDalM}, for $\Haus{N-1}$-a.e.\ $x\in\Sigma$ the function $\bfu$ has one-sided Lebesgue limits $\bfu^\pm_\Sigma(x)$ with respect to ${\bm\nu}_\Sigma(x)$; if they coincide, then $x\notin S_\bfu$, otherwise $x\in J_\bfu$. Hence $\Haus{N-1}(\Sigma\cap(S_\bfu\setminus J_\bfu))=0$, and Proposition~\ref{ABS} gives $|\DIV^s\MA|(S_\bfu\setminus J_\bfu)=|\DIV^s\MA|(\Sigma\cap(S_\bfu\setminus J_\bfu))=0$.
\end{enumerate}
Therefore, \eqref{assum} can fail only if $\Haus{N-1}(S_\bfu\setminus J_\bfu)>0$, which is an open problem (see \cite[Remark 6.2]{AmbCosDalM}), and $|\DIV^s\MA|$ charges a purely unrectifiable subset of $S_\bfu\setminus J_\bfu$.
\end{remark}

Under assumption \eqref{assum}, the precise representative $\bfu^*$ of $\bfu$ (see \eqref{eq:preciserepresentative}) is well defined $|\DIV \MA|$-a.e.\ in $\Omega$.

This allows us to introduce the following definition. For every \(\MA\in\DMA(\Omega;\R^{N\times N}_{\rm sym})\) and \(\bfu\in BD(\Omega)\cap L^\infty(\Omega;\R^N)\)  satisfying \eqref{assum} we define the linear functional
\((\MA: E\bfu) \colon C^\infty_c(\Omega) \to \R\) by
\begin{equation}\label{eq:pairingBD}
\pscal{(\MA:E\bfu)}{\varphi} 
:=
-\int_\Omega \varphi \bfu^*\cdot d \DIV \MA - \int_\Omega  \, \MA:[\bfu \odot\nabla\varphi]\, dx. 
\end{equation}

The definition above is consistent with the classical notion of pairing introduced by Temam in the case of fields with summable divergence, as shown in the following remark.

\begin{remark}\label{pipo1} 
If $\DIV \MA\ll\Leb{N}$  we retrieve Temam's definition of pairing
\begin{equation}\label{eq:pairingBDTem}
\pscal{(\MA:E\bfu)}{\varphi} :=
-\int_\Omega \varphi \bfu\cdot \DIV \MA\, dx - \int_\Omega  \, \MA:[\bfu \odot\nabla\varphi]\, dx\,,
\end{equation}
(see \cite[Ch. II, Sec. 7.3]{Temam} and \cite{KT}). Indeed, in this case $\bfu^*=\bfu$ $\Leb{N}$-a.e.\ and $d\DIV \MA= \DIV \MA\,dx$, so that \eqref{eq:pairingBD} reduces to \eqref{eq:pairingBDTem}.
Moreover, the distribution is well defined under the assumptions that $\Omega$ is a bounded Lipschitz subset of $\R^N$  and $\DIV \MA\in L^N(\Omega;\R^N)$, since $\bfu\in BD(\Omega)\subset L^{\frac N{N-1}}(\Omega;\R^N)$ (see \cite[Ch. II, Sec. 2.4, Theorem 2.2]{Temam}).  Hence, by  H\"{o}lder's inequality, $\bfu\in L^1_{|\DIV\MA|}(\Omega ;\R^N)$.
Temam proved that this distribution 
 is a bounded Radon measure on $\Omega$ (see \cite[Proposition 1.1]{KT} and \cite[Lemma 7.3]{Temam}). 
\end{remark}

\begin{theorem}\label{thm:main2}
Let $\MA\in\DMA(\Omega;\R^{N\times N}_{\rm sym})$ and $\bfu\in BD(\Omega)\cap L^\infty(\Omega;\R^N)$ satisfy \eqref{assum}. 
Then the distribution \((\MA: E\bfu)\) defined in \eqref{eq:pairingBD} is a Radon measure in \(\Omega\), absolutely continuous with respect to \(|E\bfu|\), and it holds that
\begin{equation}
|(\MA: E\bfu)|(B) \leq \|\MA\|_{L^\infty(U)} |E\bfu|(B) \quad \mbox{for every Borel set $B \subset U \Subset \Omega$.}
\label{eq:abscontBD}
\end{equation}

Moreover, $\Div(\MA\bfu)\in \mathcal{M}(\Omega)$ and the following integration by parts formula holds
\begin{equation}\label{f:integrBD}
\Div(\MA\bfu)= \bfu^* \cdot \DIV\MA + (\MA: E\bfu)
\end{equation}
in the sense of measures in \(\Omega\). 
\end{theorem}

\begin{proof}
The fact that $\Div(\MA\bfu)$ is a Radon measure, i.e.\ $\MA\bfu\in\mathcal{DM}^\infty(\Omega;\R^N)$, will follow from \eqref{f:integrBD}.
We subdivide the proof into steps. \\
\noindent
{\emph{Step 1.}} First, we assume that $\bfu$ is Lipschitz continuous on compact subsets of $\Omega$. We claim that \eqref{f:integrBD} holds in the form
\begin{equation}
\Div(\MA\bfu)= \bfu \cdot \DIV\MA + \MA: e(\bfu)\,\Leb{N}\,,
\label{eq:leibnizBD}
\end{equation}
in the sense of measures on $\Omega$. 
Let $(\MA_k)$ be the sequence given by Lemma~\ref{lem:approxim}. From Lemma \ref{lem:approxim}$(iv)$, we have $\DIV\MA_k \rightharpoonup^* \DIV\MA$ as Radon measures on $\Omega$. Then
\begin{equation*}
\bfu \cdot \DIV\MA_k +\MA_k:\nabla\bfu\,\Leb{N} \rightharpoonup^* \bfu \cdot \DIV\MA + \MA:e(\bfu)\,\Leb{N} \quad \mbox{ in $\mathcal{M}(\Omega)$.}
\end{equation*}
Moreover, $\Div(\MA_k\bfu)\rightharpoonup\Div(\MA\bfu)$ in $\mathcal{D}'(\Omega)$, since $\MA_k\to\MA$ in $L^1(\Omega)$ and $\bfu$ is locally bounded; here we also used that $\MA_k:\nabla\bfu\to\MA:\nabla\bfu=\MA:e(\bfu)$ in $L^1_{\rm loc}(\Omega)$, since $\MA$ is symmetric. Collecting these facts, we can compute the limit, in the sense of distributions, of the identity
\begin{equation*}
\Div(\MA_k\bfu)= \bfu \cdot \DIV\MA_k + \MA_k:\nabla\bfu,
\end{equation*}
so that
\[
\Div(\MA\bfu)= \bfu \cdot \DIV\MA + \MA:e(\bfu)
\]
holds in the sense of distributions. Since $C^1_c(\Omega)$ is dense in $C_c(\Omega)$, the same identity holds in the sense of measures.
\\
\noindent
{\emph{Step 2.}} Let $\bfu\in BD(\Omega)\cap L^\infty(\Omega;\R^N)$, and let $(\bfu_\varepsilon)_\varepsilon$ be a sequence of mollifications of $\bfu$ as in \eqref{eq:mollification}. 
Then, by {\emph{Step 1}}, we have
\begin{equation*}
\Div(\MA\bfu_\varepsilon)= \bfu_\varepsilon \cdot \DIV\MA + \MA: e(\bfu_\varepsilon) \,\Leb{N} \quad \mbox{ in $\mathcal{M}(\Omega)$.}    
\end{equation*}
Moreover, combining \eqref{eq:convtoprecise} with \eqref{assum}, we deduce $\bfu_\varepsilon \to \bfu^*$ $|\DIV \MA|$-a.e.\ in $\Omega$ as $\varepsilon \to 0$.

By the Dominated Convergence Theorem applied with respect to the measure $\DIV\MA$, we obtain
\begin{equation}
\bfu_\varepsilon \cdot \DIV\MA \rightharpoonup^* \bfu^* \cdot \DIV\MA \,\, \mbox{ as $\varepsilon\to0$  in $\mathcal{M}(\Omega)$.}
\label{eq:stima1444}
\end{equation}  
By \emph{Step 1}, the sequence $(\Div(\MA\bfu_\varepsilon))_{\varepsilon}$ is equibounded in $\mathcal{M}(\Omega')$ for every $\Omega'\Subset \Omega$: indeed, for $\varepsilon$ small enough,
\begin{equation*}
\begin{split}
|\Div(\MA\bfu_\varepsilon)|(\Omega')&\le\|\bfu_\varepsilon\|_{L^\infty(\Omega';\R^N)}|\DIV\MA|(\Omega')+\|\MA\|_{L^\infty(\Omega;\R^{N\times N}_{\rm sym})}\int_{\Omega'}|e(\bfu_\varepsilon)|\,dx\\
&\le\|\bfu\|_{L^\infty(\Omega;\R^N)}|\DIV \MA|(\Omega)+ \|\MA\|_{L^\infty(\Omega; \R^{N\times N}_{\rm sym})}|E\bfu|(\Omega).
\end{split}
\end{equation*}
Since $\Div(\MA\bfu_\varepsilon)$ converges to $\Div(\MA\bfu)$ in the sense of distributions, it follows that
\begin{equation}
\Div(\MA\bfu_\varepsilon) \rightharpoonup^* \Div(\MA\bfu) \,\, \mbox{ as $\varepsilon\to0$  in $\mathcal{M}(\Omega)$.}
\label{eq:stima211}
\end{equation}
Collecting \eqref{eq:stima1444} and \eqref{eq:stima211}, and using \eqref{eq:leibnizBD} with $\bfu=\bfu_\varepsilon$ together with the definition of pairing \eqref{eq:pairingBD}, we obtain 
 \begin{equation}
\MA: e(\bfu_\varepsilon) \rightharpoonup^* \Div(\MA\bfu) - \bfu^* \cdot \DIV\MA = (\MA: E\bfu) \,\, \mbox{ as $\varepsilon\to0$  in $\mathcal{M}(\Omega)$,}
\label{eq:approxpairconv}
\end{equation}
which shows that $(\MA: E\bfu)$ is a Radon measure on $\Omega$.  This also yields \eqref{f:integrBD}. 
The proof of \eqref{eq:abscontBD} 
relies on an approximation argument based on the sequence $(\bfu_\varepsilon)_{\varepsilon>0}$ introduced above. Let $U\subset \Omega$  be an open set, $K\Subset U$ be a compact set, and let $\varphi\in C_c(U)$ be a test function with ${\rm supp}\,\varphi \subset K$. There exists $r_0>0$ such that $K_r:=K+B_r(0)\subset U$ for every $r\in (0,r_0)$. Let $r\in (0,r_0)$ be such that $|E\bfu|(\partial K_r)$=0. Indeed, this property holds for almost every $r$. Then, by \eqref{eq:approxpairconv}, 
\begin{equation*}
\begin{split}
|\langle (\MA: E\bfu), \varphi \rangle| & \leq \|\varphi\|_{L^\infty(K)} \|\MA\|_{L^\infty(U;\R^{N\times N}_{\rm sym})} \displaystyle \mathop{\lim\inf}_{\varepsilon\to0}\int_{K_r}| e(\bfu_\varepsilon)|\,dx   \\
& =\|\varphi\|_{L^\infty(K)} \|\MA\|_{L^\infty(U;\R^{N\times N}_{\rm sym})} |E\bfu|(K_r)\,, 
\end{split} 
\end{equation*}
whence, passing to the limit as $r\to0$,
\begin{equation*}
|(\MA: E\bfu)|(K) \leq \|\MA\|_{L^\infty(U;\R^{N\times N}_{\rm sym})} |E\bfu|(K)\,. 
\end{equation*}
Finally, \eqref{eq:abscontBD} follows by the regularity of measures $|(\MA: E\bfu)|$ and $|E\bfu|$ for every Borel set $B\subset U$. 
\end{proof}
\begin{proposition}\label{p:tracesbBD}
	Let \(\MA\in\DMA(\Omega;\R^{N\times N}_{{\rm sym}})\),  \({\bfu}\in BD(\Omega)\cap L^\infty(\Omega;\R^N)\) such that \eqref{assum} holds, and let $\Sigma\subset\Omega$ be an oriented countably $\Haus{N-1}$-rectifiable set. 
	Then \(\MA\bfu\in\DM(\Omega;\R^N)\) and the normal traces of \(\MA\bfu\) on \(\Sigma\) are given by
	\begin{equation}\label{f:trusBD}
		\Trace{\MA\bfu}{\Sigma} = 
			\bfu^\pm\cdot \Trpmv{\MA}{\Sigma},\ \ 
			 \Haus{N-1}-\text{a.e.\ in}\ \Sigma,
	\end{equation}
where $\bfu^\pm$ are the traces of $\bfu$ on $\Sigma$ defined in Proposition~\ref{prop:rectifiable}.
Moreover,
	\begin{equation}\label{f:truABD}
		\Div(\MA\bfu)\res J_\bfu =
		\left[
		\bfu^+ \cdot \Trpv{\MA}{J_\bfu} - \bfu^- \cdot \Trmv{\MA}{J_\bfu}
		\right]\, \hh \res  J_\bfu\,.
	\end{equation}
\end{proposition}

\begin{proof} 
The fact that $\MA\bfu \in \DM(\Omega;\R^N)$ is due to \eqref{f:integrBD}. We will show
\eqref{f:trusBD} only for \(\Tr^-\), since the argument for \(\Tr^+\) is analogous. Recalling the notation of Section~\ref{sec:preliminaries} (§~\hyperref[sec:distrtraces]{\emph{Normal traces}}), we may assume that $J_\bfu$ is oriented with $\bm\nu_\Sigma$ on $J_\bfu\cap\Sigma$. We adapt the argument of \cite[Proposition 3.4]{DCSTensor} for $\bfu\in BV(\Omega;\R^N)\cap L^\infty(\Omega;\R^N)$.

Namely, we perform a blow-up argument around each point \(x\in \Sigma\) such that
\begin{itemize}
	\item[(a)] \(x\in\Omega\setminus (S_{\bfu}\setminus J_{\bfu})\), 
	\(x\in E_i\) for some \(i\), the set \(E_i\) has density \(1\) at \(x\), 
	and \(x\) is a Lebesgue point of both \(\Trmv{\MA}{\Sigma}\) and \(\Trm{\MA\bfu}{\Sigma}\)
	with respect to
	\(\Haus{N-1}\res\partial\Omega_i\);
	\\
	\item[(b)]
	\(|\DIV\MA| \res\Omega_i (B_\varepsilon(x)) = o(\varepsilon^{N-1})\)  \quad \mbox{ and }\quad \(|\Div (\MA\bfu)| \res\Omega_i  (B_\varepsilon(x)) = o(\varepsilon^{N-1})\)
	as $\varepsilon\to0$.
\end{itemize}
Note that \(\Haus{N-1}\)-a.e.\ \(x\in \Sigma\setminus(S_\bfu\setminus J_\bfu)\) satisfies (a) and (b), as a consequence of \cite[Theorem~2.56 and (2.41)]{AFP} since $\MA$ and $\MA\bfu$ have measure divergence. Moreover, $\Haus{N-1}(\Sigma\cap(S_\bfu\setminus J_\bfu))=0$: indeed, by \cite[Section 3, property (ii)]{AmbCosDalM}, for $\Haus{N-1}$-a.e.\ $x\in\Sigma$ the function $\bfu$ has one-sided Lebesgue limits $\bfu^\pm_\Sigma(x)$ with respect to ${\bm\nu}_\Sigma(x)$, so that either $x\notin S_\bfu$ or $x\in J_\bfu$ (see Remark~\ref{rem:Hstar}$(iv)$). Hence $\Haus{N-1}$-a.e.\ $x\in\Sigma$ satisfies (a) and (b), and it is enough to prove \eqref{f:trusBD} at these points.
In order not to overburden the notation, we set
 \(\bfu^-(x) := \widetilde{\bfu}(x)\) if \(x\in \Omega\setminus S_{\bfu}\).

Let \(\eta\in C^{\infty}_c(\R^N)\) be a cut-off function, with ${\rm supp}\,\eta\subseteq B_1(0)$, such that \(0\leq \eta \leq 1\), and for every \(\varepsilon > 0\) set \(\eta_{\varepsilon}(y) := \eta\left(\frac{y-x}{\varepsilon}\right)\). We will use both the definitions of trace \eqref{f:disttr} and \eqref{f:disttrv}. In particular, for $\varepsilon$ small enough, we can choose $\eta_\varepsilon$ and $\bm\varphi=\bfu^-(x)\eta_{\varepsilon}$  as test functions in \eqref{f:disttr} and \eqref{f:disttrv}, respectively. Note that $\nabla\bm\varphi(y)=\nabla\eta_{\varepsilon}(y)\otimes\bfu^-(x)$,
hence by \eqref{eq:ident12} we have
$$
\MA:\nabla\bm\varphi(y)=\MA:(\bfu^-(x)\odot\nabla\eta_{\varepsilon}(y))=\bfu^-(x)\cdot(\MA\nabla\eta_{\varepsilon}(y)) = \nabla\eta_{\varepsilon}(y) \cdot (\MA \bfu^-(x))\,.
$$
Then, for \(\varepsilon > 0\) small enough, we can write
\begin{equation}\label{f:tra}
\begin{split}
&\frac{1}{\varepsilon^{N-1}}
\int_{\partial\Omega_i} 
[\Tr(\MA \bfu, \partial\Omega_i) -
\bfu^-(x) \cdot\Trv(\MA, \partial\Omega_i)]
\, \eta_{\varepsilon}(y)\, d\Haus{N-1}(y)
\\ = {} &
\frac{1}{\varepsilon^{N-1}}
\int_{\Omega_i} \nabla\eta_{\varepsilon}(y) \cdot \MA(y)\bfu(y)\, dy+\frac{1}{\varepsilon^{N-1}}
\int_{\Omega_i} \eta_{\varepsilon}(y) \, d\Div(\MA\bfu)
\\ & -\frac{1}{\varepsilon^{N-1}} 
\int_{\Omega_i} \bfu^-(x)\cdot(\MA(y) \nabla\eta_{\varepsilon}(y) )\, dy
-\frac{1}{\varepsilon^{N-1}} 
\int_{\Omega_i} \eta_{\varepsilon}(y)\bfu^-(x)\cdot d\DIV \MA(y)\,
\\ = {} &
\frac{1}{\varepsilon^{N-1}}
\int_{\Omega_i} \nabla\eta_{\varepsilon}(y) \cdot [\MA(y)\bfu(y) - \MA(y)\bfu^-(x)]\, dy
\\ & + 
\frac{1}{\varepsilon^{N-1}}
\int_{\Omega_i} \eta_{\varepsilon}(y) \, d[\Div(\MA\bfu) - \bfu^-(x)\cdot \DIV\MA](y)\, \\
=: & \,\,\,\, J_1(\varepsilon) + J_2(\varepsilon) \,.
\end{split}
\end{equation}

Using the change of variables \(z = (y-x)/\varepsilon\),
as \(\varepsilon \to 0\) the left-hand side of this equality converges to
\[
[\Trm{\MA\bfu}{\Sigma}(x) - \bfu^-(x)\cdot\Trmv{\MA}{\Sigma}] \int_{\Pi_x} \eta(z)\, d\Haus{N-1}(z)\,,
\]
where \(\Pi_x\) is the tangent plane to \(\Sigma_i\) at \(x\).
Clearly \(\eta\) can be chosen in such a way that
\(\int_{\Pi_x} \eta\, d\Haus{N-1} > 0\).

Therefore, in order to prove \eqref{f:trusBD} for \(\Tr^-\) it suffices
to show that the two integrals \(J_1(\varepsilon)\) and \(J_2(\varepsilon)\)
on the right hand side of \eqref{f:tra} converge to \(0\)
as \(\varepsilon \to 0\).

With the change of variables \( z = (y-x) / \varepsilon\) 
we have that
\[
J_1(\varepsilon) =
\int_{\Omega_i^\varepsilon} \MA(x+\eps z)[ \bfu(x+\varepsilon z) - \bfu^-(x)] \cdot \nabla\eta(z) 
\, dz,
\]
where
\[
\Omega_i^\varepsilon := \frac{\Omega_i - x}{\varepsilon}.
\]
As \(\varepsilon\to 0\), these sets locally converge to the half-space 
\(P_x := \{z\in\R^N:\ z \cdot \bm\nu(x) < 0\}\),
hence
\[
\lim_{\varepsilon\to 0}
\int_{\Omega_i^\varepsilon\cap B_1} |\bfu(x+\varepsilon z) - \bfu^-(x)|\, dz  = 
\lim_{\varepsilon\to 0}
\int_{P_x \cap B_1} |\bfu(x+\varepsilon z) - \bfu^-(x)|\, dz  = 0
\]
(see \cite[Remark 3.85]{AFP}) so that
\[
|J_1(\varepsilon)| \leq
\|\MA\|_{L^\infty(B_\eps(x);\R^{N\times N}_{\rm sym})}\, \|\nabla\eta\|_{L^\infty} 
\int_{\Omega_i^\varepsilon\cap B_1} |\bfu(x+\varepsilon z) - \bfu^-(x)|\, dz  
\to 0.
\]
As for $J_2(\varepsilon)$, using (b) we have
\begin{equation*}
\limsup_{\varepsilon\to0}|J_2(\varepsilon)| \leq \limsup_{\varepsilon\to 0} |\bfu^-(x)| \frac{|\DIV \MA|\res\Omega_i (B_\varepsilon(x))}{\varepsilon^{N-1}} + \limsup_{\varepsilon\to 0} \frac{|\Div(\MA\bfu)|\res\Omega_i (B_\varepsilon(x))}{\varepsilon^{N-1}} = 0\,,
\end{equation*}
so that $J_2(\varepsilon)\to0$ and the proof is complete. The representation formula \eqref{f:truABD} is an immediate consequence of  \eqref{f:trusBD} and \eqref{mmm}, applied with $\Sigma=J_\bfu$ to the vector field $\C=\MA\bfu\in\DM(\Omega;\R^N)$.
\end{proof}

We are now ready to state the main decomposition theorem for the pairing measure.

\begin{theorem}\label{t:pairingBD}
Let \(\MA\in\DMA(\Omega;\R^{N\times N}_{{\rm sym}})\) and \({\bfu}\in BD(\Omega)\cap L^\infty(\Omega;\R^N)\) satisfying \eqref{assum}.
Then the measure \((\MA : E\bfu)\) admits the following decomposition:
\begin{itemize}
	\item[(i)]
absolutely continuous part: 
	\((\MA : E\bfu)^a = \MA: e(\bfu)\, \Leb{N}\);
		 \item[(ii)]
jump part:
	\(\displaystyle
	(\MA : E\bfu)^j = 
	\frac{\Trpv{\MA}{J_\bfu}+\Trmv{\MA}{J_\bfu}}{2} \cdot
	 (\bfu^+-\bfu^-) \, \hh \res  J_\bfu 
	\);
\item[(iii)]
	Cantor part:
	if, in addition, 
	\begin{equation}\label{f:ipoBD}
	|E^c \bfu| (S_{\MA}) = 0,
\tag{$\rm{H_c}$}
	\end{equation}
	then
	\((\MA : E\bfu)^c = \widetilde{\MA} : E^c \bfu\).
\end{itemize}
\end{theorem}
\begin{proof} 
For the proof, we adapt the arguments from \cite[Theorem 3.2]{ChenFrid} and \cite[Theorem 3.3]{CD3}.

We first prove assertion $(i)$. Let \( x \in \Omega\) be such that the limit
\[
\lim_{r\to0} \frac{1}{\Leb{N}(B_r)} \int_{B_r(x)} d(\MA:E\bfu)(y)
\]
exists, and
\begin{equation}
\begin{split}
& \lim_{r\to0} \frac{|E^s\bfu|(B_r(x))}{\Leb{N}(B_r)} = 0 \,, \\
& \lim_{r\to0} \frac{1}{\Leb{N}(B_r)} \int_{B_r(x)} |\MA(y):e(\bfu)(y) - \MA(x):e(\bfu)(x)|\, dy = 0 \,.
\end{split}
\label{eq:pointconditions}
\end{equation}
Note that $\Leb{N}$-almost every \( x \in \Omega \) has this property. We fix \( r > 0 \) such that
\begin{equation}
|E^s\bfu|(\partial B_r(x)) = 0\,.
\label{eq:nullmeasure}
\end{equation}
Let $\bfu_\varepsilon:=\eta_\varepsilon * \bfu$ be as in \eqref{eq:mollification}. Then we have $e(\bfu_\varepsilon)=\eta_\varepsilon*e(\bfu) + \eta_\varepsilon * E^s \bfu$ and, using also the fact that $|\eta_\varepsilon * E^s \bfu| \leq \eta_\varepsilon * |E^s \bfu|$, for every $\varphi\in C_c(B_r(x))$ we can estimate
\begin{equation*}
\begin{split}
&\left|\int_{B_r(x)} \varphi(y) d(\MA:E\bfu)(y) - \int_{B_r(x)} \varphi(y)\,\MA(x):e(\bfu)(x)\,dy  \right| \\
& \,\,\,\,\,\, \leq 
\left|\int_{B_r(x)} \varphi(y) d(\MA:E\bfu)(y) - \int_{B_r(x)} \varphi(y)\,\MA(y):e(\bfu_\varepsilon)(y)\,dy  \right| \\
&  \,\,\,\,\,\, \,\,\,\,\,\,  +  \int_{B_r(x)} \left|\varphi(y) \left[ \MA(y):\eta_\varepsilon * e(\bfu)(y)  - \MA(x):e(\bfu)(x) \right]\right|\, dy \\
& \,\,\,\,\,\, \,\,\,\,\,\, + \|\varphi\|_{L^\infty(B_r(x))}\|\MA\|_{L^\infty(\Omega;\R^{N\times N}_{\rm sym})} \int_{B_r(x)} \eta_\varepsilon *|E^s\bfu(y)|\,dy \,.
\end{split}
\end{equation*}
Taking into account \eqref{eq:approxpairconv}, and the fact that 
\begin{equation*}
\eta_\varepsilon*e(\bfu)\,\Leb{N} \rightharpoonup^* e(\bfu)\,\Leb{N}\,, \quad  \eta_\varepsilon * |E^s\bfu|\,\Leb{N} \rightharpoonup^* |E^s \bfu|
\end{equation*}
(see, e.g., \cite[Theorem 2.2]{AFP}), together with \eqref{eq:nullmeasure}, passing to the limit as $\varepsilon\to0$ in the estimate above we get
\begin{equation*}
\begin{split}
&\left|\int_{B_r(x)} \varphi(y) d(\MA:E\bfu)(y) - \int_{B_r(x)} \varphi(y)\,\MA(x):e(\bfu)(x)\,dy  \right| \\
& \,\,\,\,\,\, \leq \int_{B_r(x)} \left|\varphi(y) \left[ \MA(y): e(\bfu)(y) - \MA(x):e(\bfu)(x) \right]\right|\, dy \\
& \,\,\,\,\,\, \,\,\,\,\,\, + \|\varphi\|_{L^\infty(B_r(x))}\|\MA\|_{L^\infty(\Omega;\R^{N\times N}_{\rm sym})} |E^s\bfu|(B_r(x)) \,.
\end{split}
\end{equation*}
Since $\varphi$ is arbitrary, we can choose a sequence $(\varphi_k)$, with $\|\varphi_k\|_{L^\infty(B_r(x))} \leq 1$ and $\varphi_k\to1$ pointwise on $B_r(x)$, and write an estimate as above with $\varphi_k$ in place of $\varphi$. Then, dividing both the sides by $\Leb{N}(B_r)$ and passing to the limit as $k\to+\infty$, by virtue of Lebesgue's Dominated Convergence Theorem we obtain 
\begin{equation*}
\begin{split}
&\left|\frac{1}{\Leb{N}(B_r)} \int_{B_r(x)} d(\MA: E\bfu)(y) - \MA(x):e(\bfu)(x) \right| \\
& \,\,\,\,\,\, \leq \frac{1}{\Leb{N}(B_r)} \int_{B_r(x)} \left|\MA(y): e(\bfu)(y) - \MA(x):e(\bfu)(x)\right|\, dy \\
& \,\,\,\,\,\, \,\,\,\,\,\, + \frac{\|\MA\|_{L^\infty(\Omega;\R^{N\times N}_{\rm sym})}}{\Leb{N}(B_r)} |E^s\bfu|(B_r(x)) \,.
\end{split}
\end{equation*}
Now, letting $r\to0^+$ and using \eqref{eq:pointconditions}, we finally get
\begin{equation*}
\lim_{r\to0} \frac{1}{\Leb{N}(B_r)} \int_{B_r(x)} d(\MA:E\bfu)(y) = \MA(x):e(\bfu)(x)\,,
\end{equation*}
whence $(i)$ follows. 

Let us prove $(ii)$. Since, by virtue of Theorem \ref{thm:main2}, it holds that \((\MA : E\bfu) \ll |E\bfu|\), it is clear that
\((\MA : E \bfu)^j\) is supported in \(J_\bfu\).
From \eqref{f:integrBD} and \eqref{f:truABD} we have that
\[
\begin{split}
(\MA : E\bfu)^j = {} & (\MA : E\bfu) \res J_\bfu 
= \Div(\MA \bfu)\res J_\bfu - \bfu^* \cdot \DIV\MA \res J_\bfu
\\ 
= {} &
\left[
\bfu^+ \cdot \Trpv{\MA}{J_\bfu} - \bfu^- \cdot \Trmv{\MA}{J_\bfu}
\right] \hh\res J_\bfu
\\ & -
\frac{\bfu^+ + \bfu^-}{2}
\cdot \left[
\Trpv{\MA}{J_\bfu} - \Trmv{\MA}{J_\bfu}
\right] \hh\res J_\bfu
\\ = {} &
\frac{\Trpv{\MA}{J_\bfu}+\Trmv{\MA}{J_\bfu}}{2}
\cdot (\bfu^+-\bfu^-) \, \hh \res J_\bfu,
\end{split}
\]
and the proof is complete.

As for $(iii)$, we may adapt the proof of \cite[Theorem 3.3(iii)]{CD3}. Let
\(E \bfu = \bm\sigma_\bfu \, |E\bfu|\) be the polar decomposition of \(E\bfu\).
By assumption \eqref{f:ipoBD}, the approximate limit \(\widetilde{\MA}\)
of \(\MA\) exists \(|E^c \bfu|\)-a.e.\ in \(\Omega\).
Hence, setting $\mu:=(\MA : E\bfu)$, to prove the equality in $(iii)$ we have to show that
\[
\frac{d\mu}{d|E^c\bfu|}(x) = \frac{d\mu^{c}}{d|E^c\bfu|}(x)= \widetilde{\MA}(x) : \bm\sigma_\bfu(x)
\qquad
\text{for $|E^c\bfu|$-a.e.\ $x\in \Omega$}.
\]
We recall that $|E^c \bfu|$-a.e.\ $x\in\Omega$ satisfies
\begin{itemize}
	\item[(h1)]
	\(|E^c\bfu|(B_r(x)) > 0\) for every \(r >0\);
	\item[(h2)] 
	there exists
	\(\displaystyle
	\lim_{r\to 0}\frac{\mu^c(B_r(x))}{|E^c \bfu|(B_r(x))};
	\)
	\item[(h3)]
	\(\displaystyle
	\lim_{r\to 0}\frac{|E^j \bfu|(B_r(x))}{|E \bfu|(B_r(x))}=0\,, \quad \lim_{r\to 0}\frac{|E^a \bfu|(B_r(x))}{|E \bfu|(B_r(x))}=0;
	\)
	\item[(h4)]
	\(\displaystyle
	\lim_{r\to 0}\frac1{|E^c\bfu|(B_r(x))}\int_{B_r(x)}\left|
	\widetilde{\MA}(y):\bm\sigma_\bfu(y) - \widetilde{\MA}(x):\bm\sigma_\bfu(x)
	\right|\,d|E^c \bfu|(y)=0.
	\)
\end{itemize}
We fix any of such points $x\in\Omega$, and choose $r>0$ such that
\begin{equation}
\label{gtgt8}
|E^j\bfu|\left(\partial B_r(x)\right)=0 \quad \mbox{ and } \quad |E^a\bfu|\left(\partial B_r(x)\right)=0\,.
\end{equation}
Let \(\eta \in C^\infty_c(\R^N)\) be a symmetric convolution kernel with support in the unit ball,
and denote by \(\eta_{\varepsilon}(x) := \varepsilon^{-N} \eta(\frac{x}{\varepsilon})\). It is well known (see, e.g., \cite{AFP}) that
\(\eta_ \varepsilon \ast E\bfu=
\eta_ \varepsilon \ast E^a \bfu+\eta_ \varepsilon \ast E^j\bfu + \eta_ \varepsilon \ast E^c \bfu\).
Hence for every $\phi\in C_c(\R^N)$ with support in \(B_r(x)\) it holds
\begin{equation}\label{f:diseq}
\begin{split}
&\Bigg|\frac1{|E^c \bfu|(B_r(x))}\int_{B_r(x)}\phi(y)
\MA(y) : \eta_\eps\ast E\bfu (y) \, dy
\\ & \quad 
-
\frac1{|E^c \bfu|(B_r(x))}\int_{B_r(x)}\phi(y)
\widetilde{\MA}(x) : \bm\sigma_\bfu(x)   \,d|E^c \bfu|(y)\Bigg|\\
& \leq
\Bigg|\frac1{|E^c \bfu|(B_r(x))}\int_{B_r(x)}\phi(y)
\MA(y) : \eta_\eps \ast E^c \bfu(y)\, dy
\\ & \quad -
\frac1{|E^c \bfu|(B_r(x))}\int_{B_r(x)}\phi(y)
\widetilde{\MA}(x) : \bm\sigma_\bfu(x) \,d|E^c\bfu|(y)\Bigg|
\\ & \quad 
+ \frac1{|E^c \bfu|(B_r(x))}\|\phi\|_\infty \|\MA\|_{L^\infty(B_r(x);\R^{N\times N}_{\rm sym})}
\int_{B_r(x)}\eta_ \varepsilon \ast |E^j\bfu|\,dy \\
& \quad + \frac1{|E^c \bfu|(B_r(x))}\|\phi\|_\infty \|\MA\|_{L^\infty(B_r(x);\R^{N\times N}_{\rm sym})}
\int_{B_r(x)}\eta_ \varepsilon \ast |E^a\bfu|\,dy\,,
\end{split}
\end{equation}
where in the last inequality we used that 
$\left|\eta_ \varepsilon \ast E^t\bfu\right|\leq \eta_ \varepsilon \ast |E^t\bfu|$ for $t=a,j$.
We note that by \eqref{gtgt8}
\[
\lim_{\varepsilon\to 0}\int_{B_r(x)}\eta_ \varepsilon \ast |E^j \bfu|\,dy= 
|E^j \bfu| (B_r(x))\,, \quad \lim_{\varepsilon\to 0}\int_{B_r(x)}\eta_ \varepsilon \ast |E^a \bfu|\,dy= 
|E^a \bfu| (B_r(x))\,.
\]
Furthermore,
\[
\int_{B_r(x)}\phi(y)
\MA(y) : \eta_\eps \ast E^c \bfu(y)\, dy 
= \int_{B_r(x)}
[\eta_\eps \ast (\phi \MA)](y):\bm\sigma_\bfu(y)\, d|E^c\bfu|(y).
\]
Hence by taking the limit as $\varepsilon\to 0$ in \eqref{f:diseq} we obtain
\[
\begin{split}
&\Bigg|\frac1{|E^c \bfu|(B_r(x))}\int_{B_r(x)}\phi(y)\ d\mu(y)\\&-
\frac1{|E^c \bfu|(B_r(x))}\int_{B_r(x)}\phi(y)
\widetilde{\MA}(x):\bm\sigma_\bfu(x)
\,d|E^c \bfu|(y)\Bigg|\\
&\leq 
\frac1{|E^c \bfu|(B_r(x))}\|\phi\|_\infty \int_{B_r(x)}
\left|
\widetilde{\MA}(y) : \bm\sigma_\bfu(y) - \widetilde{\MA}(x) : \bm\sigma_\bfu(x)
\right|\,d|E^c \bfu|(y)\\
&\quad +
\frac1{|E^c \bfu|(B_r(x))}\|\phi\|_\infty \|\MA\|_{L^\infty(B_r(x);\R^{N\times N}_{\rm sym})}\,|E^j \bfu|(B_r(x)) \\
& \quad + \frac1{|E^c \bfu|(B_r(x))}\|\phi\|_\infty \|\MA\|_{L^\infty(B_r(x);\R^{N\times N}_{\rm sym})}
|E^a \bfu| (B_r(x))\,.
\end{split}
\]
Taking the supremum on $\phi$, with $\|\phi\|_\infty\leq1$, on the left hand side we get 
\[
\begin{split}
&\left|\frac{\mu(B_r(x))}{|E^c \bfu|(B_r(x))}- 
\widetilde{\MA}(x) : \bm\sigma_\bfu(x)
\right|
\\
& \leq
\frac1{|E^c \bfu|(B_r(x))}\int_{B_r(x)}
\left|
\widetilde{\MA}(y):\bm\sigma_\bfu(y) -
\widetilde{\MA}(x):\bm\sigma_\bfu(x)
\right|\, d|E^c \bfu|(y)\\
& \quad +
\frac{|E^j \bfu|(B_r(x))}{|E^c \bfu|(B_r(x))} \|\MA\|_{L^\infty(B_r(x);\R^{N\times N}_{\rm sym})} +
\frac{|E^a \bfu|(B_r(x))}{|E^c \bfu|(B_r(x))} \|\MA\|_{L^\infty(B_r(x);\R^{N\times N}_{\rm sym})}\,.
\end{split}
\]
Now, by taking $r\to 0$ and using (h3)-(h4) we get $(iii)$. The proof is concluded.   
\end{proof}

\begin{remark}
[\(BV\) matrix-valued fields]\label{rem:BDpairjump}	
Let  \(\bfu\in BD(\Omega)\cap L^\infty(\Omega;\R^N)\). If \(\MA \in BV_{\rm loc}(\Omega;\mathbb{R}^{N\times N}_{\rm sym}) \cap L^\infty_{{\rm loc}}(\Omega;\mathbb{R}^{N\times N}_{\rm sym})\), then $\MA\in \DMAloc(\Omega;\mathbb{R}^{N\times N}_{\rm sym})$ and, by \eqref{f:disttrv},
\begin{equation*}
\Trpmv{\MA}{J_\bfu} = \MA^\pm_{J_{\bfu}} {\bm\nu}_{\bfu} \quad \Haus{N-1}-\mbox{a.e. in $J_\bfu$.} 
\end{equation*}
In addition, if $\MA$ and $\bfu$ satisfy \eqref{assum}, the jump part of \((\MA:E{\bfu})\) can be written as
\begin{equation*}
(\MA:E{\bfu})^j = \frac{\MA^+_{J_{\bfu}} + \MA^-_{J_{\bfu}}}{2}\, :E^j {\bfu}=\MA^*_{J_{\bfu}}:E^j{\bfu}.
\label{eq:jumppartBV}
\end{equation*}
We notice that, if \(\MA\in BV(\Omega;\R^{N\times N}_{{\rm sym}})\cap L^\infty(\Omega;\R^{N\times N}_{{\rm sym}})\) and \eqref{assum} holds (hence \eqref{f:ipoBD} is satisfied), then
	$$
	(\MA:E\bfu)=\MA^*: E\bfu,
	$$
	in the sense of measures. We point out that \eqref{assum} is not automatically satisfied for $\MA\in BV(\Omega;\R^{N\times N}_{{\rm sym}})$. Indeed, the jump part of $\DIV\MA$ is concentrated on the countably $\Haus{N-1}$-rectifiable set $J_\MA$, hence it does not charge $S_\bfu\setminus J_\bfu$ (see Remark~\ref{rem:Hstar}$(iv)$), so that \eqref{assum} reduces to $|\DIV^c\MA|(S_\bfu\setminus J_\bfu)=0$; however, $\DIV^c\MA$ vanishes on sets which are $\sigma$-finite with respect to $\Haus{N-1}$, and it is not known whether $S_\bfu\setminus J_\bfu$ is such a set. On the other hand, the slicing pairing introduced in Section~\ref{sec:pairingunbounded} satisfies $((\MA:E\bfu))_\Xi=\MA^*:E\bfu$ for every, possibly unbounded, $\bfu\in BD(\Omega)$, without assuming \eqref{assum} (see Theorem~\ref{consistenza}).
\end{remark} 

Finally, the following theorem establishes the validity of the Gauss-Green formulas for the pairing 
on sets of finite perimeter. 

\begin{theorem}\label{main1}
 Let $\MA\in\DMAloc(\R^N;\mathbb{R}^{N\times N}_{\rm sym})$ and $\bfu\in BD_{\rm loc}(\R^N)\cap L^\infty_{\rm loc}(\R^N;\R^N)$ satisfying \eqref{assum}. Let $F\subset \R^N$ be a bounded set of finite perimeter. Then the following Gauss-Green formulas hold
\begin{eqnarray}
& \displaystyle \int_{F^1} \bfu^* \cdot d\DIV\MA+ \int_{F^1}\,d(\MA:E\bfu)  =- \int_{\partial^*F} \bfu^+\cdot \Trpv{\MA}{\partial^*F}\, d\Haus{N-1}, \label{eq:gaussgreen1} \\
& \displaystyle \int_{F^1\cup \partial^*F} \bfu^* \cdot d\DIV\MA + \int_{F^1\cup\partial^*F}\,d(\MA:E\bfu)  =- \int_{\partial^*F} \bfu^-\cdot \Trmv{\MA}{\partial^*F}\, d\Haus{N-1}, \label{eq:gaussgreen2}
\end{eqnarray}
where $\bfu^\pm$ are the traces of $\bfu$ on $\partial^* F$ defined in Proposition~\ref{prop:rectifiable} and 
$\Trpmv{\MA}{\partial^*F}$
are the normal traces of \(\MA\) when \(\partial^* F\) is oriented
with respect to the interior unit normal vector. 
\end{theorem}

\begin{remark} 
Let \(\MA\in C^1(\R^N;\R^{N\times N}_{{\rm sym}})\) and $\bfu\in BD_{\rm loc}(\R^N)\cap L^\infty_{\rm loc}(\R^N;\R^N)$. 
Let $\Omega\subset\R^N$ be a bounded open set with Lipschitz boundary. 
Then \eqref{assum} 
is satisfied, and $\Trpv{\MA}{\partial \Omega}=-\MA\bm\nu_{\Omega}$.  
By \eqref{eq:gaussgreen1}  we obtain
$$
\int_{\Omega} \bfu \cdot \DIV\MA\,dx + \int_{\Omega}\MA:\,dE\bfu
= \int_{\partial \Omega} \gamma^+(\bfu)
\cdot (\MA\bm\nu_{\Omega}) \,d\Haus{N-1}.
$$
Since by \eqref{eq:ident12} we have
\begin{equation*}\label{eq:ident123}
\gamma^+(\bfu)\cdot (\MA \bm\nu_{\Omega}) =\MA : (\gamma^+(\bfu) \odot \bm\nu_{\Omega}),
\end{equation*}
we conclude that
\begin{equation}\label{eq:ident1234}
\int_{\Omega} \bfu \cdot \DIV\MA\,dx + \int_{\Omega}\MA:\,dE\bfu
= \int_{\partial \Omega} \MA : (\gamma^+(\bfu) \odot \bm\nu_{\Omega}) \,d\Haus{N-1}.
\end{equation}

Thus, formula \eqref{eq:gaussgreen1} yields the Gauss-Green formula \eqref{eq:ident12345} for $\DMA$ fields and bounded $BD$ functions, which extends \eqref{eq:ident12345} to stress fields $\MA\in\DMA$, but only for bounded $\bfu$. Since \eqref{eq:ident12345} holds for every $\bfu\in BD(\Omega)$, this result and those of Babadjian (see Theorem~\ref{Babadjian} and formula \eqref{GGbab}) and of Strang and Temam (see Remark~\ref{rem:extension}) are not comparable: none of them contains the others.

\end{remark}

\begin{proof}[Proof of Theorem \ref{main1}]  We argue as in the proof of \cite[Theorem 1.3]{DCSTensor}. Since \(F\) is bounded, up to consider a bounded extension domain compactly containing $F$, we may assume without loss of generality that
\(\MA\in\DMA(\R^N;\mathbb{R}^{N\times N}_{{\rm sym}})\) and that $\bfu\in BD(\R^N)\cap L^\infty(\R^N;\R^N)$.

Recall that $\chi_F\in BV(\R^N)$, that the reduced boundary $\partial^*F$ is countably $\Haus{N-1}$-rectifiable, and that
\begin{equation}\label{chichi}
\chi_F^*=\chi_{F^1}+\frac12\chi_{\partial^*F}.
\end{equation}

Define the vector field ${\B}:=\MA\bfu$. Then $\chi_F {\B}$ has compact support, and therefore
\[
\Div(\chi_F {\B})(\R^N)=0
\]
(see \cite[Lemma 3.1]{ComiPayne}). 

Furthermore, by Theorem~\ref{thm:main2}, the field ${\B}$ belongs to $\DM(\R^N;\R^N)$. Hence, applying the pairing theory for scalar $BV$ functions recalled in Section~\ref{sec:preliminaries} (§~\hyperref[sec:reviewpairing]{\emph{A brief review for the pairing}}), and taking $\chi_F$ and ${\B}$ in place of $u$ and $\A$ in \eqref{f:anz}, we obtain
\begin{equation}\label{GreenIIIA}
\int_{\R^N} \chi_F^* \, d\Div {\B} =
- ( {\B} , D\chi_F) (\R^N).
\end{equation}
By the definition of normal traces, together with \eqref{mmm} and \eqref{f:trusBD}, we have
\begin{equation}\label{bubu}
\Div {\B}\res \partial^* F =
\big( \bfu^+ \cdot \Trpv{\MA}{\partial^*F} - \bfu^- \cdot \Trmv{\MA}{\partial^*F} \big)
\, \Haus{N-1} \res \partial^* F.
\end{equation}

Combining \eqref{chichi} and \eqref{bubu}, we can rewrite \eqref{GreenIIIA} as
\begin{equation}\label{f:gl}
\int_{\R^N} \chi_F^* \, d\Div {\B} =
\int_{F^1} d\Div {\B}
+ \frac12 \int_{\partial ^*F}
\big( \bfu^+\cdot\Trpv{\MA}{\partial^*F}
- \bfu^-\cdot\Trmv{\MA}{\partial^*F} \big)
\, d \Haus{N-1}.
\end{equation}

On the other hand, concerning the right-hand side of \eqref{GreenIIIA}, by \cite[Remark 4.4]{CD3} and since $|D\chi_F| = \Haus{N-1}\res \partial^* F$, we obtain
\[
( {\B},D\chi_F)= \frac{\Trp{{\B}}{\partial^*F} + \Trm{{\B}}{\partial^*F}}{2}
\, \Haus{N-1}\res \partial^* F.
\]
Noting that $\bfu^\pm \in L^1_{\Haus{N-1}\res\partial^*F}(\partial^* F;\R^N)$, since $\bfu \in L^\infty(\R^N;\R^N)$ implies that its traces are essentially bounded on $\partial^*F$, and applying \eqref{f:trusBD}, we deduce that 
\begin{equation}\label{f:gr}
( \B, D\chi_F) (\R^N) = 
\int_{\partial ^*F}\frac12 \big[ \bfu^+\cdot\Trpv{\MA}{\partial^*F}
+ \bfu^-\cdot\Trmv{\MA}{\partial^*F} \big]\ d\Haus{N-1}.
\end{equation}

Finally, inserting \eqref{f:gl} and \eqref{f:gr} into \eqref{GreenIIIA}, simplifying the resulting expression, and recalling the definition of $\B$, we obtain
\begin{equation}\label{nono1}
\int_{F^1} d\Div(\MA\bfu) =
-\int_{\partial ^*F}\bfu^+\cdot\Trpv{\MA}{\partial^*F}\ d\Haus{N-1}.
\end{equation}

Using now \eqref{f:integrBD} together with the pairing $(\MA:E\bfu)$, the left-hand side of \eqref{nono1} can be rewritten as
\[
\int_{F^1} d\big(\bfu^* \cdot \DIV\MA + (\MA: E\bfu)\big)
=
-\int_{\partial ^*F}\bfu^+\cdot\Trpv{\MA}{\partial^*F}\ d\Haus{N-1}.
\]

This concludes the proof of \eqref{eq:gaussgreen1}.
In order to prove \eqref{eq:gaussgreen2} it suffices to note that, by \eqref{nono1}
\[
\begin{split}
&\int_{F^1\cup \partial ^*F} \, d\Div(\MA\bfu)
\\
=&\int_{F^1} \, d\Div(\MA\bfu) +\int_{\partial ^*F} \, [ \bfu^+\cdot\Trpv{\MA}{\partial^*F}- \bfu^-\cdot\Trmv{\MA}{\partial^*F}] \ d\Haus{N-1}
\\
=&-\int_{\partial ^*F} \bfu^+\cdot\Trpv{\MA}{\partial^*F}\ d\Haus{N-1}\\
&+\int_{\partial ^*F} \, [ \bfu^+\cdot\Trpv{\MA}{\partial^*F}- \bfu^-\cdot\Trmv{\MA}{\partial^*F}] \ d\Haus{N-1}\,,
\end{split}
\]
hence \eqref{eq:gaussgreen2} follows.  
\end{proof} 

\section{Slicing formulas for the pairing with bounded $BD$ functions} 
\label{sec:pairingbysliceBDbdd}

In this section, we aim to derive a slicing formula for the pairing defined in \eqref{eq:pairingBD}. To this end, we introduce an assumption analogous to \eqref{eq:conditionH'}, which is motivated by the representation of the distributional divergence of a symmetric matrix field as a linear combination of suitable directional derivatives.

We fix a set
$\Xi \subset \mathbb{S}^{N-1}$ with $|\Xi|=\frac{N(N+1)}2$ such that
$\{\bm\xi \otimes \bm\xi : \bm\xi \in \Xi \}$ is a basis of
$\mathbb{R}^{N\times N}_{\rm sym}$. We shall refer to such a set as a \emph{frame}. 
A standard choice is

\begin{equation*}
\Xi
=
\Big\{ {\bf e}_i : i=1,\dots,N \Big\}
\;\cup\;
\Big\{ \tfrac{{\bf e}_i+{\bf e}_j}{\sqrt2}:\ 1\le i<j\le N \Big\}.
\label{eq:orthbasis}
\end{equation*}
where $\{{\bf e}_1,\dots,{\bf e}_N\}$ is the canonical basis of $\mathbb{R}^N$.

Let $\{{\bf B}^{\bm\xi}\}_{\bm\xi\in\Xi} \subset \mathbb{R}^{N\times N}_{\rm sym}$ 
denote the dual basis, defined by
\[
{\bf B}^{\bm\xi} : (\bm\eta \otimes \bm\eta) = \delta_{\bm\xi\bm\eta}
\qquad \forall\, \bm\xi,\bm\eta \in \Xi.
\]
Let $\MA \in L^1_{\mathrm{loc}}(\Omega;\mathbb{R}^{N\times N}_{\rm sym})$. Then, for a.e. $x\in\Omega$, we have
\[
\MA(x) = \sum_{\bm\xi \in \Xi} {a}^{\bm\xi}(x)\,(\bm\xi \otimes \bm\xi)\,,
\]
where the coefficients ${a}^{\bm\xi}$ are given by
\begin{equation}\label{eq:coeff}
{a}^{\bm\xi}(x) := \MA(x) : {\bf B}^{\bm\xi},
\qquad \forall\, \bm\xi \in \Xi.
\end{equation}

\begin{lemma} Let $\Xi$ be a frame and $\MA \in L^1_{\mathrm{loc}}(\Omega;\mathbb{R}^{N\times N}_{\rm sym})$. Then, the divergence of $\MA$ admits the decomposition 
\begin{equation}
\DIV \MA
=
\sum_{\bm\xi \in \Xi} (D_{\bm\xi} {a}^{\bm\xi})\, \bm\xi
\qquad \text{in } \mathcal{D}'(\Omega;\R^N),
\label{eq:divslic_dist}
\end{equation}
where, for every $\bm\xi\in\Xi$, $D_{\bm\xi} {a}^{\bm\xi}$ denotes the distributional derivative of ${a}^{\bm\xi}$
in the direction $\bm\xi$. 
\end{lemma}
\begin{proof}
For every ${\bm\varphi}\in C^\infty_c(\Omega;\R^N)$, since the matrix $\MA$ is symmetric, we have $\MA:\nabla{\bm \varphi} = \MA:e({\bm \varphi})$ and then, from the distributional definition of divergence, we have
\begin{equation}
\begin{split}
\langle \DIV \MA, {\bm\varphi}\rangle& = - \int_{\Omega} \MA:e({\bm \varphi}) \, dx \\ & = -\sum_{{\bm\xi}\in\Xi}\int_{\Omega}{a}^{\bm\xi}(x)\,({\bm\xi}\otimes{\bm\xi}): e({\bm \varphi})(x) \,dx \\
& = - \sum_{{\bm\xi}\in\Xi} \int_{\Omega} {a}^{\bm\xi}(x)(e({\bm \varphi})(x) {\bm\xi}\cdot \bm\xi)\,dx \\
& = - \sum_{{\bm\xi}\in\Xi}\int_{\Omega}{a}^{\bm\xi}(x)D_{\bm\xi}{\hat{\varphi}}^{\bm\xi}(x)
\,dx=\langle\sum_{\bm\xi \in \Xi} (D_{\bm\xi} {a}^{\bm\xi})\, \bm\xi , {\bm\varphi}\rangle
\,,
\end{split}
\label{eq:distributionaldiv}
\end{equation}
since
$D_{\bm\xi}{\hat{\varphi}}^{\bm\xi}=\bm\xi \cdot (D \bm\varphi(x)\bm\xi) = \bm\xi \cdot (D \bm\varphi(x)^T\bm\xi) = e({\bm \varphi}) {\bm\xi}\cdot \bm\xi$.
\end{proof}

However, in order to give a meaning to a slicing formula for the pairing \eqref{eq:pairingBD}, we shall require a stronger regularity on the directional derivatives $D_{\bm\xi} {a}^{\bm\xi}$ of the coefficients, namely that they are Radon measures. This leads to the following assumption:
\begin{equation}
D_{\bm\xi} {a}^{\bm\xi} \in \mathcal M(\Omega) \mbox{ for every } \bm\xi\in\Xi\,.
\tag{H$^{\prime\prime}$}
\label{eq:condH2}
\end{equation}
By Proposition \ref{prop:slicingdirectd}, this is equivalent to the following assumptions: 
for all $\bm\xi\in \Xi$  
\begin{equation}
\begin{cases}
   {a}_y^{\bm\xi}
\in BV(\Omega^{{\bm\xi}}_y)\,\, \mbox{for $\Leb{N-1}$-a.e.  $y \in \bm{\xi}^\perp$}\,, \\
\\
\displaystyle \int_{\bm\xi^\perp} \big| D {a}_y^{\bm\xi} \big|(\Omega_y^{\bm\xi}) \, d\Haus{N-1}(y) < +\infty\,,
\end{cases}
\tag{$\rm \tilde{H}^{\prime\prime}$}
\label{eq:condH2tilde}
\end{equation}
where, with a slight abuse of notation, we set
\begin{equation*}\label{eq:sectionsaxi}
 {a}_y^{\bm\xi}(s):= {a}^{\bm\xi}(y+s\xi)\,,\,\,\,s\in \Omega^{{\bm\xi}}_y.
\end{equation*}
This implies that $   {a}_y^{\bm\xi}
\in L^\infty_{\rm{loc}}(\Omega^{\bm\xi}_y)$.

We are led to introduce the class (depending on $\Xi$)
\begin{equation}
\label{eq:newspace}
\BV^\infty_{\Xi}(\Omega;\R^{N\times N}_{\rm sym}):=\left\{\MA\in L^\infty(\Omega;\R^{N\times N}_{\rm sym}):\ \ 
D_{\bm\xi} {a}^{\bm\xi} \in \mathcal M(\Omega) \mbox{ for every } \bm\xi\in\Xi
\right\}.
\end{equation}
We notice that, if $\MA\in BV(\Omega;\R^{N\times N}_{\rm sym})\cap L^\infty(\Omega;\R^{N\times N}_{\rm sym})$, then the coefficients defined by \eqref{eq:coeff} belong to $BV(\Omega)\cap L^\infty(\Omega)$, and $D_{\bm\xi} {a}^{\bm\xi} \in \mathcal M(\Omega)$ for every $\bm\xi\in\mathbb S^{N-1}$ (see \cite[Theorem 3.107]{AFP}). Therefore, taking into account also \eqref{eq:divslic_dist} we deduce that
$$BV(\Omega;\R^{N\times N}_{\rm sym})\cap L^\infty(\Omega;\R^{N\times N}_{\rm sym})\subset 
\BV^\infty_{\Xi}(\Omega;\R^{N\times N}_{\rm sym})\subset\DMA(\Omega;\R^{N\times N}_{\rm sym}).$$

\begin{remark}\label{rem:strict}
The first inclusion is strict. To see this, consider $\Omega:=(-1,1)\times(-1,1)\subset \R^2$
the function $h:(-1,1)\to \R$
$$
h(t):=\begin{cases}
\sin \left(\frac1{t}\right)\,, \ &t\not=0\\
0\,, \ &t=0,
\end{cases}
$$
and the stress tensor field
\begin{equation*}
\MA(x) = \begin{pmatrix}    h(x_2) & 0 \\  0 &  H\left(x_2 - \frac{1}{2}\right)  \end{pmatrix}\,,
\end{equation*}
where $H$ denotes the Heaviside function.
Then $\MA\in L^\infty(\Omega;\mathbb{R}^{2\times2}_{\rm sym})$, but $\MA \not \in BV(\Omega;\mathbb{R}^{2\times2}_{\rm sym})$. Indeed, 
\[
D_{{\bf e}_2}A_{11}=-\frac{\cos\left(\frac{1}{x_2}\right)}{x_2^2}\notin L^1_{\mathrm{loc}}((-1,1)), \quad\mbox{ for }x_2\not= 0
\]
so this distributional derivative cannot be represented by a finite Radon measure. Let $\Xi:=\{{\bf e}_1, {\bf e}_2, \frac{{\bf e}_1+{\bf e}_2}{\sqrt{2}}\}$. A simple computation shows that
\[
a^{{\bf e}_1}=\sin(1/x_2),\qquad
a^{{\bf e}_2}=H(x_2-1/2),\qquad
a^{\frac{{\bf e}_1+{\bf e}_2}{\sqrt{2}}}=0,
\]
so that
\[
D_{{\bf e}_1}a^{{\bf e}_1}=0,\qquad
D_{{\bf e}_2}a^{{\bf e}_2}=\delta_{\{x_2=1/2\}},\qquad
D_{\frac{{\bf e}_1+{\bf e}_2}{\sqrt{2}}}a^{\frac{{\bf e}_1+{\bf e}_2}{\sqrt{2}}}=0.
\]
Hence, $\MA \in \BV^\infty_\Xi(\Omega; \mathbb{R}^{2 \times 2}_{\mathrm{sym}})$. Finally, note that 
\[
\DIV\MA= (0, \delta_{\{x_2=1/2\}})\,.
\] 
\end{remark}

When $\bfu^*$ is integrated against $\DIV\MA$ by slicing, each directional derivative $D_{\bm\xi}a^{\bm\xi}$ appears separately, and assumption \eqref{assum} is not enough to give a meaning to $\bfu^*$ with respect to these measures. We shall therefore use the following frame-wise compatibility condition between $\MA$ and $\bfu\in BD(\Omega)$:
\begin{equation}\label{assumXi}\tag{${\rm H}_*^\Xi$}
|D_{\bm\xi}a^{\bm\xi}|(S_\bfu\setminus J_\bfu)=0\qquad\text{for every }\bm\xi\in\Xi .
\end{equation}
Under \eqref{assumXi}, $\bfu^*$ is defined $|D_{\bm\xi}a^{\bm\xi}|$-a.e.\ in $\Omega$ for every $\bm\xi\in\Xi$.

\begin{remark}\label{rem:HXi}
$(i)$ Assumption \eqref{assumXi} implies \eqref{assum}. Indeed, by \eqref{eq:divslic_dist} we have $\DIV\MA=\sum_{\bm\xi\in\Xi}(D_{\bm\xi}a^{\bm\xi})\,\bm\xi$ as measures, so that, since $|\bm\xi|=1$,
\[
|\DIV\MA|(S_\bfu\setminus J_\bfu)\le\sum_{\bm\xi\in\Xi}|D_{\bm\xi}a^{\bm\xi}|(S_\bfu\setminus J_\bfu)=0 .
\]
$(ii)$ The measures $|D_{\bm\xi}a^{\bm\xi}|$ are not controlled by $|\DIV\MA|$, since the vectors $\bm\xi\in\Xi$ are linearly dependent and cancellations may occur. For instance, let $N=2$, let $\Xi$ be the standard frame, ${\bf n}:=\frac{1}{\sqrt5}(1,2)$, ${\bf t}:=\frac{1}{\sqrt5}(2,-1)$, and $\MA(x):=H(x\cdot{\bf n})\,{\bf t}\otimes{\bf t}$. Then $\DIV\MA=({\bf t}\otimes{\bf t}){\bf n}\,\Haus{1}\res\{x\cdot{\bf n}=0\}=0$, while the coefficients of $\MA$ are $\big(\frac65,\frac35,-\frac45\big)H(x\cdot{\bf n})$, so that $D_{{\bf e}_1}a^{{\bf e}_1}=\frac{6}{5\sqrt5}\,\Haus{1}\res\{x\cdot{\bf n}=0\}\neq0$.

$(iii)$ For every $\bm\xi\in\Xi$, the measure $|D_{\bm\xi}a^{\bm\xi}|=\Leb{N-1}\res\Omega^{\bm\xi}\otimes_{\mathcal M}|Da^{\bm\xi}_y|$ (see \eqref{eq:slicdirectder} and \cite[Corollary 2.29]{AFP}) vanishes on $\Haus{N-1}$-negligible sets, since the orthogonal projection of such a set onto $\bm\xi^\perp$ is $\Leb{N-1}$-negligible. Hence \eqref{assumXi} holds whenever $\Haus{N-1}(S_\bfu\setminus J_\bfu)=0$, for instance if $\bfu\in BV(\Omega;\R^N)$ or $\bfu\in SBD^p(\Omega)$ for some $p>1$ (see the Remark after \eqref{assum}), and it also holds if $D_{\bm\xi}a^{\bm\xi}\ll\Leb{N}$ for every $\bm\xi\in\Xi$.
\end{remark}

The following lemma identifies the slices of $\bfu^*\cdot\bm\xi$ with the precise representatives of the slices $\hat u^{\bm\xi}_y$ on the support of $|Da^{\bm\xi}_y|$.

\begin{lemma}\label{t:identification}
Let $\MA\in\BV^\infty_{\Xi}(\Omega;\R^{N\times N}_{\rm sym})$ and $\bfu\in BD(\Omega)$ satisfy \eqref{assumXi}, and let $\bm\xi\in\Xi$. Then, for $\Leb{N-1}$-a.e.\ $y\in\Omega^{\bm\xi}$,
\begin{equation}\label{f:identification}
(\bfu^*\cdot\bm\xi)(y+s\bm\xi)=(\hat u^{\bm\xi}_y)^*(s)\qquad\text{for }|Da^{\bm\xi}_y|\text{-a.e.\ }s\in\Omega^{\bm\xi}_y .
\end{equation}
\end{lemma}
\begin{proof}
Let $Z:=S_\bfu\setminus J_\bfu$. By \eqref{assumXi} and Remark~\ref{rem:HXi}$(iii)$, $\int_{\bm\xi^\perp}|Da^{\bm\xi}_y|(Z^{\bm\xi}_y)\,d\Leb{N-1}(y)=|D_{\bm\xi}a^{\bm\xi}|(Z)=0$, so that $|Da^{\bm\xi}_y|(Z^{\bm\xi}_y)=0$ for $\Leb{N-1}$-a.e.\ $y$. The conclusion follows from Lemma~\ref{t:BDpointwise}$(iii)$.
\end{proof}

\begin{proposition}\label{L1DivA2233}
Let $\MA\in\BV^\infty_{\Xi}(\Omega;\R^{N\times N}_{\rm sym})$. For every ${\bm \varphi}\in C_c(\Omega;\R^N)$ we have that 
\begin{equation}
\int_{\Omega} {\bm \varphi} \cdot d\DIV\MA
= \sum_{{\bm\xi}\in\Xi}\int_{{\bm\xi}^\perp}\left(
\int_{\Omega_{y }^{{\bm\xi}}} {\hat{\varphi}}_y^{{\bm\xi}}(s)\,dD{a}_y^{\bm\xi}(s)\right)\,d\Leb{N-1}(y)\,,
\label{eq:slicingdiverge}
\end{equation}
\end{proposition}
\begin{proof}
For every $\bm\varphi\in C_c(\Omega;\mathbb{R}^N)$, by assumption \eqref{eq:condH2tilde} we may apply \eqref{eq:slicdirectder2} with $\varphi=\bm\varphi\cdot\bm\xi$ and $v=a^{\bm\xi}$ to obtain \begin{equation*}
\int_\Omega \bm\varphi \cdot \bm\xi d D_{\bm\xi} a^{\bm\xi} = \int_{{\bm\xi}^\perp}\left(
\int_{\Omega_{y }^{{\bm\xi}}} {\hat{\varphi}}_y^{{\bm\xi}}(s)\,dD{a}_y^{\bm\xi}(s)\right)\,d\Leb{N-1}(y)\,.
\end{equation*}
The claim follows by summing over all $\bm\xi\in\Xi$.
\end{proof}
Now, we establish slicing formulas for the traces of $\DMA$-fields.
\begin{proposition}
Let $\MA\in \BV^\infty_{\Xi}(\Omega;\R^{N\times N}_{\rm sym})$ and 
\(\Sigma\subset\Omega\) be an oriented countably \(\Haus{N-1}\)-rectifiable set such that $\Haus{N-1}(\Sigma\cap K)<+\infty$ for every compact set $K\subset\Omega$. Then, for every test function  ${\bm \varphi}\in C_c(\Omega;\R^N)$, we have
\begin{equation}\label{eq:tracesl111pm}
\pscal{{\bf Tr^\pm}(\MA,\Sigma)}{{\bm \varphi}} 
= 
\sum_{{\bm\xi}\in\Xi}\int_{{\bm\xi}^\perp }
\left( 
\sum_{s\in\Sigma^{{\bm\xi}}_y} 
({a}_y^{\bm\xi})^\pm(s)\, \nu_{\Sigma^{{\bm\xi}}_y}(s)\,{\hat{\varphi}}_y^{{\bm\xi}}(s)
\right)
\,d\Leb{N-1}(y)\,;
\end{equation}
i.e.,
\begin{equation}
{\bf Tr^\pm}(\MA,\Sigma)\, \Haus{N-1}\res \Sigma = \sum_{\bm\xi\in \Xi}\Leb{N-1}\res\Omega^{\bm\xi}\otimes_{\mathcal{M}} ({a}_y^{\bm\xi})^\pm(s)\, \nu_{\Sigma^{{\bm\xi}}_y}(s)\bm\xi\,\Haus{0}\res\Sigma^{{\bm\xi}}_y,
\label{eq:disinttracesh}
\end{equation}
as measures on $\Sigma$, where $({a}_y^{\bm\xi})^\pm(s)$ are the one-sided limits of ${a}_y^{\bm\xi}$ at $s$ with respect to the orientation $\nu_{\Sigma^{{\bm\xi}}_y}(s)$ (i.e., $({a}_y^{\bm\xi})^+(s)$ is the limit from the side towards which $\nu_{\Sigma^{{\bm\xi}}_y}(s)$ points), and $\nu_{\Sigma^{{\bm\xi}}_y}\in \{-1, 1\}$ is the normal of $\Sigma^{{\bm\xi}}_y$ oriented as  
\begin{equation}\label{eq:tracesl111pm22}
\nu_{\Sigma^{\bm\xi}_y}(s) = \operatorname{sign}\left( \bm\nu_\Sigma(y + s\bm\xi) \cdot \bm\xi \right).
\end{equation}
Finally, for every oriented countably $\Haus{N-1}$-rectifiable set $\Sigma\subset\Omega$, not necessarily satisfying the previous finiteness condition, and for every ${\bf v}\in L^1_{\Haus{N-1}\res\Sigma}(\Omega;\R^N)$
\begin{equation}
\begin{split}
\int_{\Sigma}{\bf Tr^\pm}(\MA,\Sigma)\cdot{\bf v}\,d\Haus{N-1}
\label{eq:tracesl111starmmmm5999}
&=\sum_{{\bm\xi}\in\Xi}\int_{{\bm\xi}^\perp}
\left( 
\sum_{s\in\Sigma^{{\bm\xi}}_y} 
({a}_y^{\bm\xi})^\pm(s)\, \nu_{\Sigma^{{\bm\xi}}_y}(s)\,{\hat{v}}_y^{{\bm\xi}}(s)
\right)
\,d\Leb{N-1}(y)\,.
\end{split}
\end{equation} 
\end{proposition}
\proof
We first prove that for every open set $\Omega'\Subset\Omega$ with
\(C^1\) boundary, and for every test function  ${\bm \varphi}\in C^1_c(\Omega;\R^N)$, 
\begin{equation}\label{eq:tracesl111}
\int_{\partial \Omega'} \bm\varphi \cdot {\bf Tr}(\MA,\partial\Omega')\, d \Haus{N-1} 
 =-\sum_{{\bm\xi}\in\Xi}\int_{{\bm\xi}^\perp}
\left( 
\sum_{s\in\partial(\Omega')^{{\bm\xi}}_y} 
{({a}_y^{\bm\xi}})^+(s)\, \nu_{(\Omega')^{{\bm\xi}}_y}(s)\,{\hat{\varphi}}_y^{{\bm\xi}}(s)
\right)
\,d\Leb{N-1}(y).
\end{equation}
In order to prove \eqref{eq:tracesl111}, we recall that by \eqref{f:disttrv} and \eqref{eq:identification}
\begin{equation*} 
\int_{\partial \Omega'} \bm\varphi \cdot {\bf Tr}(\MA,\partial\Omega')\, d \Haus{N-1} 
= \int_{\Omega'} \MA:\nabla{\bm \varphi} \, dx + \int_{\Omega'} {\bm \varphi} \cdot d\DIV\MA,
\qquad
\forall {\bm \varphi}\in C^1_c(\Omega;\R^N).
\end{equation*}
Then, by \eqref{eq:slicingdiverge} and an analogous computation as in \eqref{eq:distributionaldiv} we get
\begin{equation}
\begin{split}
\int_{\partial \Omega'} & \bm\varphi \cdot {\bf Tr}(\MA,\partial\Omega')\, d \Haus{N-1} \\
& =\sum_{{\bm\xi}\in\Xi} \int_{{\bm\xi}^\perp}\left(
\int_{(\Omega')_{y }^{{\bm\xi}}} {\hat{\varphi}}_y^{{\bm\xi}}(s)\,dD{a}_y^{\bm\xi}(s) + \int_{(\Omega')_{y }^{{\bm\xi}}} {a}^{\bm\xi}_y(s)({\hat{\varphi}}_y^{{\bm\xi}})'(s)\, ds\right)\,d\Leb{N-1}(y) \,.\end{split}
\label{eq:traceslicingg}
\end{equation}
Now, for every fixed ${\bm\xi}\in\Xi$ and for $\Leb{N-1}$-a.e. $y\in {\bm\xi}^\perp$, integrating by parts we have
\begin{equation*}
\int_{(\Omega')_{y }^{{\bm\xi}}} {a}^{\bm\xi}_y(s)({\hat{\varphi}}_y^{{\bm\xi}})'(s)\, ds = -\int_{(\Omega')_{y }^{{\bm\xi}}} {\hat{\varphi}}_y^{{\bm\xi}}(s)\, d\,D{a}^{\bm\xi}_y(s) - 
\sum_{s\in\partial(\Omega')^{{\bm\xi}}_y} 
{({a}^{\bm\xi}_y)^+}(s)\, \nu_{(\Omega')^{{\bm\xi}}_y}(s)\,{\hat{\varphi}}_y^{{\bm\xi}}(s)
\,,
\end{equation*}
which inserted in \eqref{eq:traceslicingg}
gives \eqref{eq:tracesl111}. Arguing as in the proof of \eqref{eq:stimacont}, we deduce the analogous identity for every $\bm\varphi \in C_c(\Omega;\R^N)$.  From this and a standard localization argument, see \eqref{eq:tracesvc}, we immediately get \eqref{eq:tracesl111pm} which, in turn, implies \eqref{eq:disinttracesh}. In particular, if $\Haus{N-1}\res\Sigma$ is locally finite, \eqref{eq:tracesl111starmmmm5999} follows from \eqref{eq:disinttracesh} and \eqref{eq:disintegrationf}, since, by the coarea formula \cite[Theorem 2.93]{AFP}, the total variation of the right-hand side of \eqref{eq:disinttracesh} is bounded by $C_\Xi\|\MA\|_{L^\infty(\Omega;\R^{N\times N})}\Haus{N-1}\res\Sigma$. In the general case, let $\Sigma_k:=\{x\in\Sigma:\ |{\bf v}(x)|>1/k\}$, so that $\Haus{N-1}(\Sigma_k)\le k\|{\bf v}\|_{L^1_{\Haus{N-1}\res\Sigma}}<+\infty$. Since the normal traces of $\MA$ on the Borel set $\Sigma_k\subset\Sigma$, with the induced orientation, are the restrictions of those on $\Sigma$ (see \eqref{eq:tracesvc}), \eqref{eq:tracesl111starmmmm5999} holds with $\Sigma_k$ and ${\bf v}\chi_{\Sigma_k}$ in place of $\Sigma$ and ${\bf v}$. Letting $k\to\infty$, the conclusion follows by dominated convergence on both sides: on the left-hand side we use that $|{\bf Tr}^\pm(\MA,\Sigma)|\le\|\MA\|_{L^\infty(\Omega;\R^{N\times N})}$, and on the right-hand side that $|(a^{\bm\xi}_y)^\pm|\le\|a^{\bm\xi}\|_{L^\infty(\Omega)}$ and, by the coarea formula, $\int_{\bm\xi^\perp}\sum_{s\in\Sigma^{\bm\xi}_y}|\hat v^{\bm\xi}_y(s)|\,d\Leb{N-1}(y)\le\int_\Sigma|{\bf v}|\,d\Haus{N-1}$; finally, the points where ${\bf v}=0$ give no contribution to either side.
\endproof

The following theorem provides a slicing representation for the pairing \eqref{eq:pairingBD}.
\begin{theorem}\label{thm:slicingBDbdd} Let $\MA\in\BV^\infty_{\Xi}(\Omega;\R^{N\times N}_{\rm sym})$ and ${\bfu}\in BD(\Omega)\cap L^\infty(\Omega;\R^N)$ be such that
assumptions \eqref{assum} and \eqref{f:ipoBD} hold. 
Then, for every $\varphi\in C_c(\Omega)$, we have
\begin{equation}
\int_{\Omega}\varphi d(\MA:E\bfu)
=\sum_{{\bm\xi}\in\Xi}\int_{{\bm\xi}^\perp}
\left(\int_
{\Omega^{{\bm\xi}}_y} 
\varphi_y^{{\bm\xi}}(s)\,
d({a}_y^{{\bm\xi}},D{\hat{u}}_{y }^{{\bm\xi}})\,(s)
\right)\,d\Leb{N-1}(y).
\label{eq:BDslicingrepr}
\end{equation}
\end{theorem}
\proof
We subdivide the proof into steps. \\
\\
\emph{Step~1 (Lebesgue part).} From Theorem~\ref{t:pairingBD}(i) and Fubini's Theorem, 
\begin{equation*}
\begin{split}
\int_{\Omega}\varphi d(\MA:E\bfu)^a &=\int_{\Omega}\varphi\MA: e(\bfu) \, d\Leb{N}(x) \\
& =\sum_{{\bm\xi}\in\Xi}\int_{\Omega}\varphi  {a}^{\bm\xi}(x)(e(\bfu)(x){\bm\xi}\cdot \bm\xi)\,d\Leb{N}(x) \\
& =\sum_{{\bm\xi}\in\Xi}\int_{{\bm\xi}^\perp }
\left( 
\int_{\Omega^{{\bm\xi}}_y} 
{{a}_y^{{\bm\xi}}}(s)({\hat{u}}_y^{{\bm\xi}})'(s){\varphi}_y^{{\bm\xi}}(s)ds
\right)
\,d\Leb{N-1}(y)
\\&=\sum_{{\bm\xi}\in\Xi}\int_{{\bm\xi}^\perp}
\left(\int_
{\Omega^{{\bm\xi}}_y} 
\varphi_y^{{\bm\xi}}(s)
({a}_y^{{\bm\xi}},D{\hat{u}}_y^{{\bm\xi}})^a\,ds
\right)\,d\Leb{N-1}(y)
\,,
\end{split}
\end{equation*}
where we used Theorem~\ref{t:pairingCD3}\,(i) for $N=1$, since ${\hat{u}}_y^{{\bm\xi}}$ is bounded and so it belongs to $L^1(\Omega^{{\bm\xi}}_y,|D{a}_y^{{\bm\xi}}|)$
and this concludes the proof of \emph{Step~1}.  \\

\emph{Step~2 (Jump part).} From Theorem \ref{t:pairingBD}\,(ii), \eqref{eq:tracesl111starmmmm5999} with ${\bf v}:=\varphi (\bfu^+-\bfu^-)$, 
and Theorem~\ref{t:pairingCD3}\,(ii), together with Theorem~\ref{structure}$(iii)$ and Lemma~\ref{t:BDpointwise}$(ii)$, which identify, for $\Leb{N-1}$-a.e.\ $y$, the slices of $(\bfu^+-\bfu^-)\cdot\bm\xi$ on $(J_\bfu)^{\bm\xi}_y$ with the jumps of $\hat u^{\bm\xi}_y$ (the points of $J_\bfu\setminus J^{\bm\xi}_\bfu$ give no contribution, since there $(\bfu^+-\bfu^-)\cdot\bm\xi=0$, and for $\Leb{N-1}$-a.e.\ $y$ the line $y+\R\bm\xi$ does not meet the points of $J_\bfu$ with ${\bm\nu}_\bfu\cdot\bm\xi=0$), we get
\begin{equation*}
\begin{split}
\int_{\Omega}\varphi d(\MA:E\bfu)^j &=	\int_{J_{\bfu}}\varphi\, \frac{\Trpv{\MA}{J_\bfu}+\Trmv{\MA}{J_\bfu}}{2}\cdot
	 (\bfu^+-\bfu^-) \, d\hh \\
&=\sum_{{\bm\xi}\in\Xi}\int_{{\bm\xi}^\perp}\left(\sum_
{s\in J_{{\hat{u}}_y^{{\bm\xi}}}} 
\varphi_y^{{\bm\xi}}(s)
({a}_y^{{\bm\xi}})^*(s)\, \nu_{{\hat{u}}_y^{{\bm\xi}}}(s)\,\left(({\hat{u}}_y^{{\bm\xi}})^+(s)-({\hat{u}}_y^{{\bm\xi}})^-(s)\right)
\right)\,d\Leb{N-1}(y)
\\&=\sum_{{\bm\xi}\in\Xi}\int_{{\bm\xi}^\perp}
\left(\int_
{J_{{\hat{u}}_y^{{\bm\xi}}}} 
\varphi_y^{{\bm\xi}}(s)\,
d({a}_y^{\bm\xi},D{{\hat{u}}_y^{{\bm\xi}}})^j\,(s)
\right)\,d\Leb{N-1}(y)
\,,
\end{split}
\end{equation*}
where in the last equality we used again Theorem~\ref{t:pairingCD3}\,(ii) for $N=1$.

\emph{Step~3 (Cantor part).} We recall that, under assumption \eqref{f:ipoBD}, by Theorem \ref{t:pairingBD}(iii) the representation \((\MA : E\bfu)^c = \widetilde{\MA} : E^c \bfu\) holds.
From this and \eqref{eq:disintEc}, we have 
\begin{equation*}
\begin{split}
\int_{\Omega}\varphi d(\MA:E\bfu)^c &=\int_{\Omega}\varphi \widetilde{\MA} : dE^c \bfu
=\sum_{{\bm\xi}\in\Xi}\int_{\Omega}\varphi  \widetilde{{a}^{{\bm\xi}}}\,d(E^c \bfu{\bm\xi}\cdot \bm\xi)\,  
 \\
& 
=\sum_{{\bm\xi}\in\Xi}\int_{{\bm\xi}^\perp}
\left(\int_
{\Omega^{{\bm\xi}}_y} 
\varphi_y^{{\bm\xi}}(s)
(\widetilde{{a}_y^{{\bm\xi}}})(s)\,dD^c{\hat{u}}_{y }^{{\bm\xi}}(s)
\right)\,d\Leb{N-1}(y)\,.
\end{split}
\end{equation*}
In the last equality we used \eqref{eq:disintEc} and the fact that, for $\Leb{N-1}$-a.e.\ $y$, the slice $(\widetilde{a^{\bm\xi}})^{\bm\xi}_y$ of the Lebesgue representative of $a^{\bm\xi}$ coincides with $\widetilde{a^{\bm\xi}_y}=(a^{\bm\xi}_y)^*$ $|D^c\hat u^{\bm\xi}_y|$-a.e. This follows from Lemma~\ref{t:directionalpointwise}, applied with $v=a^{\bm\xi}$ (recall that $D_{\bm\xi}a^{\bm\xi}\in\mathcal M(\Omega)$) and $\lambda=|E^c\bfu\,\bm\xi\cdot\bm\xi|=\Leb{N-1}\res\Omega^{\bm\xi}\otimes_{\mathcal M}|D^c\hat u^{\bm\xi}_y|$ (see \eqref{eq:disintEc}): indeed, the measures $D^c\hat u^{\bm\xi}_y$ have no atoms, and $\lambda(S_{a^{\bm\xi}})\le|E^c\bfu|(S_\MA)=0$ by \eqref{f:ipoBD}, since $S_{a^{\bm\xi}}\subset S_\MA$.
Now, by applying Theorem \ref{t:pairingCD3}(iii) in dimension one, we obtain
\begin{equation*}
\begin{split}
\sum_{{\bm\xi}\in\Xi}&\int_{{\bm\xi}^\perp}
\left(\int_
{\Omega^{{\bm\xi}}_y} 
\varphi_y^{{\bm\xi}}(s)
(\widetilde{{a}_y^{{\bm\xi}}})(s)\,dD^c{\hat{u}}_{y }^{{\bm\xi}}(s)
\right)\,d\Leb{N-1}(y) \\
&=\sum_{{\bm\xi}\in\Xi}\int_{{\bm\xi}^\perp}
\left(\int_
{\Omega^{{\bm\xi}}_y} 
\varphi_y^{{\bm\xi}}(s)\,
d({a}_y^{{\bm\xi}},D{\hat{u}}_{y }^{{\bm\xi}})^c\,(s)
\right)\,d\Leb{N-1}(y)
\,,
\end{split}
\end{equation*}
and this concludes the proof of \eqref{eq:BDslicingrepr}.  
\endproof

\section{Pairing for (possibly) unbounded $BD$ functions} \label{sec:pairingunbounded}

A natural extension of the pairing defined in \eqref{eq:pairingBD} would be to consider possibly unbounded functions $\bfu \in BD(\Omega)$, requiring that $\bfu^* \in L^1_{|\DIV \MA|}(\Omega;\R^N)$, still under assumption \eqref{assum}. 
In this case, the functional in \eqref{eq:pairingBD} would remain well defined, but it is not clear whether it is a Radon measure for a general $\MA\in\DMA(\Omega;\R^{N\times N}_{\rm sym})$. We shall see in Theorem~\ref{t:leibnizconsistency} that, if $\MA\in\BV^\infty_{\Xi}(\Omega;\R^{N\times N}_{\rm sym})$, and the frame-wise compatibility condition \eqref{assumXi} and the frame-wise integrability condition \eqref{f:integrabloc} hold, it is a Radon measure, which coincides with the slicing pairing introduced below. 
We recall that the strategy developed in \cite[Proposition 4.4]{CDCM} to remove the assumption $\bfu \in L^\infty(\Omega;\R^N)$ - namely, passing to truncations - does not apply in the present setting, since $BD$ is not closed under truncations, even smooth ones.

\subsection{A new pairing by slicing}
In this section, we introduce an alternative definition of pairing, which is given by slicing.

First, we study the measurability of the one-dimensional pairing.
Let us fix a frame $\Xi$, a matrix \(\MA\in\BV^\infty_{\Xi}(\Omega;\R^{N\times N}_{\rm sym})\), and a function \(\bfu\in BD(\Omega)\).
For every $\bm\xi\in \Xi$ and for $\Haus{N-1}$-a.e.\ $y \in \Omega^{\bm\xi}$, let $(a^{\bm\xi}_y, D\hat{u}_y^{\bm\xi})$ denote the one-dimensional pairing on $\Omega^{{\bm\xi}}_y$ between the locally bounded $BV$ function ${a}_y^{{\bm\xi}}$ and the measure derivative $D{\hat{u}}_{y }^{{\bm\xi}}$ of the $BV$ function ${\hat u}_y^{{\bm\xi}}$. By Theorem \ref{t:pairingCD3}, this pairing admits the representation 
\begin{equation}\label{1dim}
({a}_y^{{\bm\xi}},D{\hat{u}}_{y }^{{\bm\xi}})=({a}_y^{{\bm\xi}})^*D{\hat{u}}_{y }^{{\bm\xi}}
\end{equation}
in the sense of measures on $\Omega^{{\bm\xi}}_y$. However, in what follows, we prefer to keep the pairing notation.
\begin{lemma}
\label{lem:measurability_slicing}
For every ${\bm\xi}\in \Xi$ and for every test function $\varphi \in C_c(\Omega)$, the function 
\begin{equation*}\label{1dim2}
y \mapsto \int_{\Omega^{\bm{\xi}}_y} \varphi^{\bm{\xi}}_y(s) \, d(a^{\bm\xi}_y, D\hat{u}^{\bm{\xi}}_y)(s)
\end{equation*}
is $\Haus{N-1}$-measurable on $\Omega^{\bm\xi}$. 
\end{lemma}
\begin{proof}
By modifying the function $a^{\bm\xi}$ on a set of Lebesgue measure zero, we may assume that $a^{\bm\xi}$ is a Borel function on $\Omega$, that $|a^{\bm\xi}|\le\|a^{\bm\xi}\|_{L^\infty(\Omega)}$ everywhere in $\Omega$, and that $a^{\bm\xi}_y \in BV_{\rm loc}(\Omega_y^{\bm\xi})$ for every $y \in \Omega^{\bm\xi}$: indeed, it is enough to truncate a Borel representative at the levels $\pm\|a^{\bm\xi}\|_{L^\infty(\Omega)}$, and then to set it equal to $0$ on the lines $y+\R\bm\xi$ with $y$ in a Borel $\Leb{N-1}$-negligible set containing all $y$ such that the slice is not in $BV_{\rm loc}(\Omega^{\bm\xi}_y)$ (see \eqref{eq:condH2tilde}); this does not affect the slicing pairing, since for $\Leb{N-1}$-a.e.\ $y$ the slices of two representatives coincide $\Leb{1}$-a.e. For every $x \in \Omega$ we define
\begin{equation*}
v_+^{\bm\xi}(x) := \limsup_{\rho \to 0^+} \frac{1}{\rho} \int_0^\rho a^{\bm\xi}(x+\tau{\bm\xi}) \, d\tau \quad \text{and} \quad v_-^{\bm\xi}(x) := \limsup_{\rho \to 0^+} \frac{1}{\rho} \int_{-\rho}^0 a^{\bm\xi}(x+\tau{\bm\xi}) \, d\tau.
\end{equation*}
By Fubini's theorem $v_+^{\bm\xi}$ and $v_-^{\bm\xi}$ are Borel functions on $\Omega$ (since $a^{\bm\xi}$ is bounded, the averages depend continuously on $\rho$, so that the limit superior can be computed along rational values of $\rho$); moreover, $|v^{\bm\xi}_\pm|\le\|a^{\bm\xi}\|_{L^\infty(\Omega)}$. 
Therefore $F := \{ x \in \Omega : v_+^{\bm\xi}(x) \neq v_-^{\bm\xi}(x) \}$ is a Borel set. For every $y \in \Omega^{\bm\xi}$ we have for every $s\in \Omega^{\bm\xi}_y$
\begin{equation*}
(v_+^{\bm\xi})^{\bm\xi}_y(s)= \limsup_{\rho \to 0^+} \frac{1}{\rho} \int_s^{s+\rho} a_y^{\bm\xi}(\tau) \, d\tau 
\quad \text{and} 
\quad (v_-^{\bm\xi})^{\bm\xi}_y(s) = \limsup_{\rho \to 0^+} \frac{1}{\rho} \int_{s-\rho}^{s} a_y^{\bm\xi}(\tau) \, d\tau 
\end{equation*}
and
$$
(v_+^{\bm\xi})^{\bm\xi}_y(s)=(v_-^{\bm\xi})^{\bm\xi}_y(s)=a^{\bm\xi}_y(s) \quad \mbox{ $\Leb{1}$-a.e. in $\Omega_y^{\bm\xi}$ }
$$
thanks to Lebesgue's differentiation theorem. By elementary properties of $BV$ functions of one variable, for every \( y \in \Omega^{\bm\xi} \) this implies
\[
\lim_{\rho \to 0+} \frac{1}{2\rho} \int_{s-\rho}^{s+\rho} 
\left| a^{\bm\xi}_y(\tau) - (v^{\bm\xi}_+)^{\bm\xi}_y(s) \right| \, d\tau = 0\,,\quad \mbox{ for every } s \in \Omega^{\bm\xi}_y \setminus F^{\bm\xi}_y\,,
\]
\[
(a^{\bm\xi}_y)^+(s) = (v^{\bm\xi}_+)^{\bm\xi}_y(s) 
\quad \text{and} \quad 
(a^{\bm\xi}_y)^-(s) = (v^{\bm\xi}_-)^{\bm\xi}_y(s)\,,\quad \mbox{ for every } s \in \Omega^{\bm\xi}_y \cap F^{\bm\xi}_y\,.
\]
This implies that 
\begin{equation*}
(a_y^{\bm\xi})^*(s) = 
\begin{cases}
(v^{\bm\xi}_+)^{\bm\xi}_y(s)\,, & \mbox{ for every } s \in \Omega^{\bm\xi}_y \setminus F^{\bm\xi}_y\,, \\
\frac{(v^{\bm\xi}_+)^{\bm\xi}_y(s) +(v^{\bm\xi}_-)^{\bm\xi}_y(s)}2\,, & \mbox{ for every } s \in \Omega^{\bm\xi}_y \cap F^{\bm\xi}_y\,,
\end{cases}
\end{equation*}
and
\( J_{a^{\bm\xi}_y} = F^{\bm\xi}_y \) for every \( y \in \Omega^{\bm\xi} \).
 Therefore,
by \eqref{1dim}, we have
\begin{equation*}
\int_{\Omega^{\bm{\xi}}_y\setminus J_{a^{\bm\xi}_y} }\varphi^{\bm{\xi}}_y(s) \, d({a}_y^{{\bm\xi}},D{\hat{u}}_{y }^{{\bm\xi}})(s)
=\int_{\Omega^{\bm{\xi}}_y\setminus F_y^{\bm\xi} }\varphi^{\bm{\xi}}_y(s) \, (v_+^{\bm\xi})_y^{{\bm\xi}}\,dD{\hat{u}}_{y }^{{\bm\xi}}(s)
\end{equation*}
\begin{equation*}\label{1dim222}
\int_{\Omega^{\bm{\xi}}_y\cap J_{a^{\bm\xi}_y} }\varphi^{\bm{\xi}}_y(s) \, d({a}_y^{{\bm\xi}},D{\hat{u}}_{y }^{{\bm\xi}})(s)
=\int_{\Omega^{\bm{\xi}}_y\cap F_y^{\bm\xi} }\varphi^{\bm{\xi}}_y(s) \, \frac{(v^{\bm\xi}_+)^{\bm\xi}_y(s) +(v^{\bm\xi}_-)^{\bm\xi}_y(s)}2 \,dD{\hat{u}}_{y }^{{\bm\xi}}(s)
\end{equation*}
for $\Haus{N-1}$-a.e. $y \in \Omega^{\bm\xi}$, namely for those $y$ such that $\hat u^{\bm\xi}_y\in BV_{\rm loc}(\Omega^{\bm\xi}_y)$, for which \eqref{1dim} holds. Hence, the map
\begin{equation*}
y \mapsto 
\int_{\Omega^{\bm{\xi}}_y} \varphi^{\bm{\xi}}_y(s)\,\frac{(v^{\bm\xi}_+)^{\bm\xi}_y(s) +(v^{\bm\xi}_-)^{\bm\xi}_y(s)}2 \, dD\hat{u}^{\bm{\xi}}_y(s)
\end{equation*}
 is $\Haus{N-1}$-measurable on $\Omega^{\bm\xi}$, by using \eqref{measurability} with the Borel function $g(x):=\varphi(x) \frac{v^{\bm\xi}_+(x) +v^{\bm\xi}_-(x)}2$. This concludes the proof.
\end{proof}

Given a frame $\Xi$, a matrix \(\MA\in\BV^\infty_{\Xi}(\Omega;\R^{N\times N}_{\rm sym})\) and a function \(\bfu\in BD(\Omega)\),
we define the \emph{slicing pairing} as the linear functional
\(((\MA:E\bfu))_\Xi \colon  C^\infty_c(\Omega) \to \R\) by
\begin{equation}\label{eq:pairingBDnew}
\pscal{((\MA:E\bfu))_\Xi}{\varphi} 
:=
\sum_{{\bm\xi}\in\Xi}\int_{{\bm\xi}^\perp}
\left(\int_
{\Omega^{{\bm\xi}}_y} 
\varphi_y^{{\bm\xi}}(s)\,
d({a}_y^{{\bm\xi}},D{\hat{u}}_{y }^{{\bm\xi}})(s)
\right)\,d\Leb{N-1}(y).
\end{equation}

\begin{theorem}\label{consistenza}
Let \(\MA\in\BV^\infty_{\Xi}(\Omega;\R^{N\times N}_{\rm sym})$ and \(\bfu\in BD(\Omega)\).
Then 
\begin{itemize}
\item[$(i)$] the slicing pairing $((\MA:E\bfu))_\Xi$ is a Radon measure in $\Omega$ and
$$
((\MA:E\bfu))_\Xi=\sum_{{\bm\xi}\in\Xi}\Leb{N-1}\res\Omega^{\bm\xi}\otimes_{\mathcal{M}}({a}_y^{{\bm\xi}},D{\hat{u}}_{y }^{{\bm\xi}})\res\Omega^{{\bm\xi}}_y
$$
in the sense of measures;
\item[$(ii)$]  the measure $((\MA:E\bfu))_\Xi$ is absolutely continuous with respect to \(|E\bfu|\), and it holds that
\begin{equation}
|((\MA:E\bfu))_\Xi|(B) \leq C_\Xi\|\MA\|_{L^\infty(U;\R^{N\times N})} |E\bfu|(B) \quad \mbox{for every Borel set $B \subset U \Subset \Omega$,}
\label{eq:abscontBD11}
\end{equation}
where $C_\Xi:=\left(\displaystyle\max_{\xi \in \Xi}\|{\bf B}^{\bm\xi}\|\right) |\Xi|$; 
\item[$(iii)$] (consistency)  if \(\bfu\in BD(\Omega)\cap L^\infty(\Omega;\R^N)\) such that \eqref{f:ipoBD} and \eqref{assum} hold, then
 $((\MA:E\bfu))_\Xi$ coincides with the pairing defined in \eqref{eq:pairingBD}; i.e.,
 $$
 ((\MA:E\bfu))_\Xi=(\MA:E\bfu),
 $$
 in the sense of measures (the same conclusion holds without \eqref{f:ipoBD} if \eqref{assumXi} is satisfied, see Theorem~\ref{t:leibnizconsistency}).
\end{itemize}
Finally, if \(\bfu\in BD(\Omega)\) and  $\MA\in BV(\Omega;\mathbb{R}^{N\times N}_{\rm sym})\cap L^\infty(\Omega;\mathbb{R}^{N\times N}_{\rm sym})$, then
 $$
 ((\MA:E\bfu))_\Xi=\MA^*:E\bfu,
 $$
 in the sense of measures.
\end{theorem}
\begin{proof}
As for $(i)$, we note that for every ${\bm\xi}\in\Xi$ and for $\Haus{N-1}$-a.e. $y\in \bm\xi^\perp$, by \eqref{eq:condH2} we have ${a}_y^{{\bm\xi}}\in BV_{\rm{loc}}(\Omega^{{\bm\xi}}_y)\cap L^\infty_{\rm{loc}}(\Omega^{{\bm\xi}}_y)$ and, by Theorem~\ref{structure}$(ii)$, ${\hat{u}}_{y }^{{\bm\xi}}\in BV_{\rm{loc}}(\Omega^{{\bm\xi}}_y)$. Therefore, for such $y$, the pairing $({a}_y^{{\bm\xi}},D {\hat{u}}_{y }^{{\bm\xi}})$  is well defined and is a Radon measure in $\Omega^{{\bm\xi}}_y$. Moreover, 
for every Borel set $B \subset U \Subset \Omega$ we have
$$
|({a}_y^{{\bm\xi}},D {\hat{u}}_{y }^{{\bm\xi}})|(B^{{\bm\xi}}_y)
\leq \|{a}_y^{{\bm\xi}}\|_{L^\infty(U^{{\bm\xi}}_y)}|D {\hat{u}}_{y }^{{\bm\xi}}|(B^{{\bm\xi}}_y)
\leq \max_{\xi \in \Xi}\|{\bf B}^{\bm\xi}\|\|\MA\|_{L^\infty(U;\mathbb{R}^{N\times N}_{\rm sym})}|D {\hat{u}}_{y }^{{\bm\xi}}|(B^{{\bm\xi}}_y)\,.
$$
Let $\varphi \in C_c^\infty(\Omega)$. For every ${\bm\xi}\in\Xi$ and for $\Haus{N-1}$-a.e. $y\in \bm\xi^\perp$, we estimate
\[
\left|
\int_{\Omega_y^{\bm{\xi}}}
\varphi_y^{\bm{\xi}}(s)\,
d({a}_y^{{\bm\xi}}, D{\hat{u}}_y^{{\bm\xi}})(s)
\right|
\le
\|\varphi\|_{L^\infty(\Omega)}\,
|({a}_y^{{\bm\xi}}, D{\hat{u}}_y^{{\bm\xi}})|(\Omega_y^{\bm{\xi}}).
\]
Hence the map
\[
y \mapsto
\int_{\Omega_y^{\bm{\xi}}}
\varphi_y^{\bm{\xi}}(s)\,
d({a}_y^{{\bm\xi}}, D{\hat{u}}_y^{{\bm\xi}})(s)
\]
(which is $\Haus{N-1}$-measurable on $\Omega^{\bm\xi}$ by Lemma \ref{lem:measurability_slicing}) belongs to $L^1(\Omega^{\bm\xi})$, thanks to the above estimate and to \eqref{eq:slicefq}.
Therefore, the linear functional on $C_c^\infty(\Omega)$ defined by
\eqref{eq:pairingBDnew} is well defined and satisfies the estimate
\[
|\langle ((\MA:E\bfu))_\Xi,\varphi\rangle|
\leq \|\varphi\|_{L^\infty(\Omega)}\left(\max_{\xi \in \Xi}\|{\bf B}^{\bm\xi}\|\right)\|\MA\|_{L^\infty(U;\R^{N\times N}_{\rm sym})}
\sum_{{\bm\xi}\in\Xi}
\int_{{\bm\xi}^\perp}
|D{\hat{u}}_y^{{\bm\xi}}|(\Omega_y^{\bm{\xi}})\,d\Leb{N-1}(y).
\]
Using \eqref{eq:slicefq}, we obtain
\[
|\langle ((\MA:E\bfu))_\Xi,\varphi\rangle|
\le
C_\Xi \,\|\varphi\|_{L^\infty(\Omega)}\|\MA\|_{L^\infty(\Omega;\R^{N\times N}_{\rm sym})}\,|E\bfu|(\Omega).
\]
In particular, the functional is bounded with respect to the sup norm on $C_c^\infty(\Omega)$ and thus extends uniquely to a bounded linear functional on $C_c(\Omega)$.
Therefore, it can be represented by a Radon measure, still denoted by $((\MA:E\bfu))_\Xi$. By construction, its action on test functions is exactly given by \eqref{eq:pairingBDnew}, which directly translates to the representation formula in $(i)$.

For what concerns $(ii)$, for every Borel set $B \subset U \Subset \Omega$ we have
$$
|((\MA:E\bfu))_\Xi|(B)\leq\sum_{{\bm\xi}\in\Xi}\int_{{\bm\xi}^\perp}
\left(\int_
{B^{{\bm\xi}}_y} 
d|({a}_y^{\bm\xi},D{\hat{u}}_{y }^{{\bm\xi}})|(s)
\right)\,d\Leb{N-1}(y)
$$
$$
\leq
\sum_{{\bm\xi}\in\Xi}\int_{{\bm\xi}^\perp}
\|{a}_y^{\bm\xi}\|_{L^\infty(U^{{\bm\xi}}_y)}|D{\hat{u}}_{y }^{{\bm\xi}}|(B^{{\bm\xi}}_y)
\,d\Leb{N-1}(y)\leq C_\Xi \|\MA\|_{L^\infty(U;\R^{N\times N}_{\rm sym})} |E\bfu|(B).
$$
Assertion $(iii)$ is an immediate consequence of the analysis developed in Section \ref{sec:pairingbysliceBDbdd}\,. 
In particular, it is a consequence of the boundedness of $\bfu$, of the assumptions \eqref{f:ipoBD} and \eqref{assum}, of Theorem \ref{thm:slicingBDbdd} and definition \eqref{eq:pairingBDnew}.
Finally, if $\MA\in BV(\Omega;\mathbb{R}^{N\times N}_{\rm sym})\cap L^\infty(\Omega;\mathbb{R}^{N\times N}_{\rm sym})$, then by Theorem \ref{structure} for every $\bfu\in BD(\Omega)$ and for every $\varphi\in C_c(\Omega)$ we get 
\begin{equation*}\label{intrinsic}
\begin{split}
	\int_{\Omega}\varphi\, \MA^*:\,dE{\bfu}&=
	\sum_{{\bm\xi}\in\Xi}\int_{\Omega}\varphi\, (\MA^*:{\bf B}^{\bm\xi}){\bm\xi}\otimes {\bm\xi}:\,dE{\bfu}
 =\sum_{{\bm\xi}\in\Xi}\int_{\Omega}\varphi(x)  ({a}^{{\bm\xi}})^*(x)\,d(E\bfu{\bm\xi}\cdot\bm\xi)(x)
\\&=\sum_{{\bm\xi}\in\Xi}\int_{{\bm\xi}^\perp}
\left(\int_
{\Omega^{{\bm\xi}}_y} 
\varphi_y^{{\bm\xi}}(s)\,
({a}_y^{{\bm\xi}})^*(s)\,dD{\hat{u}}_{y }^{{\bm\xi}}(s)
\right)\,d\Leb{N-1}(y)=((\MA:E\bfu))_\Xi
\,,
\end{split}
\end{equation*}	
where  we used Theorem~\ref{t:pairingCD3}\, for $N=1$, together with the following facts. Since $\MA\in BV$, we have $\Haus{N-1}(S_\MA\setminus J_\MA)=0$ \cite[Theorem 3.78]{AFP}, and $|E\bfu|$ vanishes on $\Haus{N-1}$-negligible sets by the decomposition of $E\bfu$ recalled in Section~\ref{sec:preliminaries}; hence $\MA^*$ is defined $|E\bfu|$-a.e., and $\MA^*:{\bf B}^{\bm\xi}=(a^{\bm\xi})^*$ wherever $\MA^*$ is defined, by the linearity of approximate limits and one-sided approximate limits. In the third equality we used \eqref{eq:slicefq} and \eqref{eq:disintegrationf} with the bounded Borel function $\varphi\,(a^{\bm\xi})^*$, and the fact that, since $a^{\bm\xi}\in BV(\Omega)$, for $\Leb{N-1}$-a.e.\ $y$ the slice of $(a^{\bm\xi})^*$ coincides with $(a^{\bm\xi}_y)^*$ at every point of $\Omega^{\bm\xi}_y$ (see \cite[Theorem 3.108]{AFP}).
\end{proof}

\begin{remark}\label{rem:W11LD}
If \(\MA\in W^{1,1}(\Omega;\mathbb \R^{N \times N}_{\rm sym})\cap L^\infty(\Omega;\mathbb \R^{N \times N}_{\rm sym})\) and \(\bfu\in BD(\Omega)\), then
$$
((\MA:E\bfu))_\Xi={\widetilde\MA}:E\bfu,
$$
in the sense of measures in $\Omega$,
since for every $\bm\xi\in\Xi$ and for $\Leb{N-1}$-a.e. $y\in\bm\xi^\perp$ we have  $({a}_y^{\bm\xi},D{\hat{u}}_{y }^{{\bm\xi}})={\widetilde{a}_y^{\bm\xi}}\,dD{\hat{u}}_{y }^{{\bm\xi}}$ in the sense of measures in $\Omega^{{\bm\xi}}_y$.\\
Moreover,
if \(\MA\in\BV^\infty_{\Xi}(\Omega;\R^{N\times N}_{\rm sym})\) and
 \(\bfu\in LD(\Omega)\), it holds that
 $$
((\MA:E\bfu))_\Xi=\MA:e(\bfu)\Leb{N},
$$
in the sense of measures,
since for every $\bm\xi\in\Xi$ and for $\Leb{N-1}$-a.e. $y\in\bm\xi^\perp$ we have $({a}_y^{{\bm\xi}},D{\hat{u}}_{y }^{{\bm\xi}})={{a}_y^{{\bm\xi}}}({\hat{u}}_{y }^{{\bm\xi}})^\prime\,\Leb{1}$ in the sense of measures in $\Omega^{{\bm\xi}}_y$  (see \eqref{mmmm}) 
(see \cite[Remark 7.6]{Temam}).
\end{remark}
We now show that the slicing pairing can be written as the integral, with respect to $E\bfu$, of a suitable Borel representative of $\MA$ adapted to the frame $\Xi$. Its construction is based on the one-sided limits along $\bm\xi$ introduced in the proof of Lemma~\ref{lem:measurability_slicing}.

Let $\MA\in\BV^\infty_{\Xi}(\Omega;\R^{N\times N}_{\rm sym})$. For every $\bm\xi\in\Xi$ we fix a Borel representative of the coefficient $a^{\bm\xi}$ defined in \eqref{eq:coeff} such that
\begin{equation}\label{f:goodrep}
|a^{\bm\xi}(x)|\le\|a^{\bm\xi}\|_{L^\infty(B_r(x))}\qquad\text{for every $x\in\Omega$ and every ball $B_r(x)\subset\Omega$.}
\end{equation}
Such a representative exists: if $\ell^{\bm\xi}(x):=\lim_{r\to0^+}\|a^{\bm\xi}\|_{L^\infty(B_r(x))}$, then $\ell^{\bm\xi}$ is upper semicontinuous, $|a^{\bm\xi}|\le\ell^{\bm\xi}$ at every Lebesgue point of $a^{\bm\xi}$, and $a^{\bm\xi}\chi_{\{|a^{\bm\xi}|\le\ell^{\bm\xi}\}}$ satisfies \eqref{f:goodrep}. In particular, $|a^{\bm\xi}|\le\|a^{\bm\xi}\|_{L^\infty(\Omega)}$ everywhere in $\Omega$; we extend $a^{\bm\xi}$ by $0$ outside $\Omega$. For $x\in\Omega$ we set, as in the proof of Lemma~\ref{lem:measurability_slicing},
\[
v_\pm^{\bm\xi}(x):=\limsup_{\rho\to0^+}\frac1\rho\int_0^\rho a^{\bm\xi}(x\pm\tau\bm\xi)\,d\tau ,
\qquad
\bar a^{\bm\xi}:=\frac{v_+^{\bm\xi}+v_-^{\bm\xi}}2 ,
\]
and
\begin{equation}\label{f:AXi}
\MA^\Xi:=\sum_{\bm\xi\in\Xi}\bar a^{\bm\xi}\,\bm\xi\otimes\bm\xi .
\end{equation}
Thus $\bar a^{\bm\xi}$ is the average of the one-sided limits of the coefficient $a^{\bm\xi}$ in its own direction $\bm\xi$.

\begin{lemma}\label{t:AXi}
Let $\MA\in\BV^\infty_{\Xi}(\Omega;\R^{N\times N}_{\rm sym})$ and let $\MA^\Xi$ be defined by \eqref{f:AXi}. Then:
\begin{itemize}
\item[$(i)$] $\MA^\Xi$ is a Borel function and $|\MA^\Xi(x)|\le C_\Xi\|\MA\|_{L^\infty(U;\R^{N\times N})}$ for every open set $U\subset\Omega$ and every $x\in U$, where $C_\Xi$ is the constant in \eqref{eq:abscontBD11};
\item[$(ii)$] for every $\bm\xi\in\Xi$ and for $\Leb{N-1}$-a.e.\ $y\in\Omega^{\bm\xi}$ we have $(\bar a^{\bm\xi})^{\bm\xi}_y=(a^{\bm\xi}_y)^*$ on $\Omega^{\bm\xi}_y$; consequently, $\MA^\Xi=\MA$ $\Leb{N}$-a.e.\ in $\Omega$ and, for every $\bfu\in BD(\Omega)$ and for $\Leb{N-1}$-a.e.\ $y\in\Omega^{\bm\xi}$,
\[
(a^{\bm\xi}_y,D\hat u^{\bm\xi}_y)=(\bar a^{\bm\xi})^{\bm\xi}_y\,D\hat u^{\bm\xi}_y\qquad\text{in the sense of measures in $\Omega^{\bm\xi}_y$};
\]
\item[$(iii)$] for every $\bfu\in BD(\Omega)$,
\begin{equation}\label{f:representationXi}
((\MA:E\bfu))_\Xi=\MA^\Xi:E\bfu
\end{equation}
in the sense of measures in $\Omega$.
\end{itemize}
\end{lemma}
\begin{proof}
$(i)$ Since $a^{\bm\xi}$ is a bounded Borel function on $\R^N$, for fixed $\rho>0$ the function $x\mapsto\frac1\rho\int_0^\rho a^{\bm\xi}(x\pm\tau\bm\xi)\,d\tau$ is Borel, by Fubini's theorem, and it depends continuously on $\rho$. Hence the limit superior can be computed along rational values of $\rho$, and $v^{\bm\xi}_\pm$, $\bar a^{\bm\xi}$ and $\MA^\Xi$ are Borel functions. Let $U\subset\Omega$ be open, $x\in U$, and $B_r(x)\subset U$. If $0<\tau<\rho<r$, then $B_{r-\tau}(x\pm\tau\bm\xi)\subset B_r(x)$, so that, by \eqref{f:goodrep}, $|a^{\bm\xi}(x\pm\tau\bm\xi)|\le\|a^{\bm\xi}\|_{L^\infty(U)}\le\max_{\bm\xi\in\Xi}\|{\bf B}^{\bm\xi}\|\,\|\MA\|_{L^\infty(U;\R^{N\times N})}$, where the last inequality is the one used in the proof of Theorem~\ref{consistenza}. Hence $|\bar a^{\bm\xi}(x)|\le\max_{\bm\xi\in\Xi}\|{\bf B}^{\bm\xi}\|\,\|\MA\|_{L^\infty(U;\R^{N\times N})}$ and, since $|\bm\xi\otimes\bm\xi|=1$,
\[
|\MA^\Xi(x)|\le\sum_{\bm\xi\in\Xi}|\bar a^{\bm\xi}(x)|\le C_\Xi\|\MA\|_{L^\infty(U;\R^{N\times N})}.
\]

$(ii)$ By \eqref{eq:condH2tilde}, for $\Leb{N-1}$-a.e.\ $y\in\Omega^{\bm\xi}$ the slice $a^{\bm\xi}_y$ belongs to $BV_{\rm loc}(\Omega^{\bm\xi}_y)$. For such $y$, the one-sided averages of $a^{\bm\xi}_y$ at every point $s\in\Omega^{\bm\xi}_y$ converge to the right and left limits $a^{\bm\xi}_y(s{+})$ and $a^{\bm\xi}_y(s{-})$ of the good representatives of $a^{\bm\xi}_y$ \cite[Theorem 3.28]{AFP}, as in the proof of Lemma~\ref{lem:measurability_slicing}; hence $(v^{\bm\xi}_\pm)^{\bm\xi}_y(s)=a^{\bm\xi}_y(s{\pm})$ and $(\bar a^{\bm\xi})^{\bm\xi}_y(s)=(a^{\bm\xi}_y)^*(s)$ for every $s\in\Omega^{\bm\xi}_y$. Since $(a^{\bm\xi}_y)^*=a^{\bm\xi}_y$ $\Leb{1}$-a.e.\ in $\Omega^{\bm\xi}_y$, Fubini's theorem gives $\bar a^{\bm\xi}=a^{\bm\xi}$ $\Leb{N}$-a.e.\ for every $\bm\xi\in\Xi$, hence $\MA^\Xi=\MA$ $\Leb{N}$-a.e. The last assertion follows from \eqref{1dim}.

$(iii)$ Let $\varphi\in C_c(\Omega)$. The function $\varphi\,\bar a^{\bm\xi}$ is a bounded Borel function with compact support in $\Omega$. By Theorem~\ref{consistenza}$(i)$, assertion $(ii)$, \eqref{eq:slicefq} and \eqref{eq:disintegrationf}, we get
\[
\begin{split}
\int_\Omega\varphi\,d((\MA:E\bfu))_\Xi
&=\sum_{\bm\xi\in\Xi}\int_{\bm\xi^\perp}\left(\int_{\Omega^{\bm\xi}_y}\varphi^{\bm\xi}_y\,(\bar a^{\bm\xi})^{\bm\xi}_y\,dD\hat u^{\bm\xi}_y\right)d\Leb{N-1}(y)\\
&=\sum_{\bm\xi\in\Xi}\int_\Omega\varphi\,\bar a^{\bm\xi}\,d(E\bfu\,\bm\xi\cdot\bm\xi)
=\int_\Omega\varphi\,\MA^\Xi:dE\bfu ,
\end{split}
\]
since $(\bm\xi\otimes\bm\xi):E\bfu=E\bfu\,\bm\xi\cdot\bm\xi$.
\end{proof}

We point out that $((\MA:E\bfu))_\Xi=\MA^\Xi:E\bfu$ does not depend on the choice of the representatives of the coefficients $a^{\bm\xi}$ fixed in \eqref{f:goodrep}, since for $\Leb{N-1}$-a.e.\ $y$ the slices of two representatives coincide $\Leb{1}$-a.e.\ in $\Omega^{\bm\xi}_y$; on the contrary, the function $\MA^\Xi$ itself depends on this choice (see Remark~\ref{t:remdevcomparison}$(ii)$).

The representation \eqref{f:representationXi} allows us to identify the absolutely continuous part and the jump part of the slicing pairing, and to show that they do not depend on the frame.

\begin{proposition}\label{t:framecomparison}
Let $\MA\in\BV^\infty_{\Xi}(\Omega;\R^{N\times N}_{\rm sym})$ and $\bfu\in BD(\Omega)$. Then:
\begin{itemize}
\item[$(i)$] for every oriented countably $\Haus{N-1}$-rectifiable set $\Sigma\subset\Omega$ with normal ${\bm\nu}_\Sigma$,
\[
\MA^\Xi{\bm\nu}_\Sigma=\frac12\big(\Trpv{\MA}{\Sigma}+\Trmv{\MA}{\Sigma}\big)\qquad\Haus{N-1}\text{-a.e.\ on }\Sigma ;
\]
\item[$(ii)$] the absolutely continuous part of $((\MA:E\bfu))_\Xi$ with respect to $\Leb{N}$ and its restriction to $J_\bfu$ (oriented by ${\bm\nu}_\bfu$) are given by
\[
((\MA:E\bfu))_\Xi^a=\MA:e(\bfu)\,\Leb{N},
\]
\[
((\MA:E\bfu))_\Xi\res J_\bfu=\frac{\Trpv{\MA}{J_\bfu}+\Trmv{\MA}{J_\bfu}}{2}\cdot(\bfu^+-\bfu^-)\,\Haus{N-1}\res J_\bfu ;
\]
\item[$(iii)$] if $\Xi'$ is another frame and $\MA\in\BV^\infty_{\Xi}(\Omega;\R^{N\times N}_{\rm sym})\cap\BV^\infty_{\Xi'}(\Omega;\R^{N\times N}_{\rm sym})$, then
\[
((\MA:E\bfu))_\Xi-((\MA:E\bfu))_{\Xi'}=(\MA^\Xi-\MA^{\Xi'}):E^c\bfu .
\]
\end{itemize}
\end{proposition}
\begin{proof}
$(i)$ Since $\Sigma$, being countably $\Haus{N-1}$-rectifiable, is $\sigma$-finite with respect to $\Haus{N-1}$, and the normal traces of $\MA$ on a Borel subset of $\Sigma$ are the restrictions of those on $\Sigma$, we may assume that $\Haus{N-1}(\Sigma)<+\infty$, so that \eqref{eq:disinttracesh} is available. Let $\bm\xi\in\Xi$. By Lemma~\ref{t:AXi}$(ii)$, for $\Leb{N-1}$-a.e.\ $y\in\Omega^{\bm\xi}$ and for every $s\in\Sigma^{\bm\xi}_y$ we have $\frac12\big((a^{\bm\xi}_y)^+(s)+(a^{\bm\xi}_y)^-(s)\big)=(a^{\bm\xi}_y)^*(s)=\bar a^{\bm\xi}(y+s\bm\xi)$, whatever the orientation of $s$. Hence, adding the two identities in \eqref{eq:disinttracesh} and using \eqref{eq:tracesl111pm22}, we get
\[
\begin{split}
&\frac12\big(\Trpv{\MA}{\Sigma}+\Trmv{\MA}{\Sigma}\big)\Haus{N-1}\res\Sigma\\
&\qquad=\sum_{\bm\xi\in\Xi}\Leb{N-1}\res\Omega^{\bm\xi}\otimes_{\mathcal M}\bar a^{\bm\xi}(y+s\bm\xi)\,{\rm sign}\big({\bm\nu}_\Sigma(y+s\bm\xi)\cdot\bm\xi\big)\,\bm\xi\,\Haus{0}\res\Sigma^{\bm\xi}_y\\
&\qquad=\sum_{\bm\xi\in\Xi}\bar a^{\bm\xi}\,({\bm\nu}_\Sigma\cdot\bm\xi)\,\bm\xi\,\Haus{N-1}\res\Sigma ,
\end{split}
\]
where the last equality follows from the coarea formula on countably $\Haus{N-1}$-rectifiable sets \cite[Theorem 2.93]{AFP}, which gives $\int_{\bm\xi^\perp}\sum_{s\in\Sigma^{\bm\xi}_y}g(y+s\bm\xi)\,d\Leb{N-1}(y)=\int_\Sigma g\,|{\bm\nu}_\Sigma\cdot\bm\xi|\,d\Haus{N-1}$ for every nonnegative Borel function $g$ on $\Omega$. Since $\sum_{\bm\xi\in\Xi}\bar a^{\bm\xi}\,({\bm\nu}_\Sigma\cdot\bm\xi)\,\bm\xi=\sum_{\bm\xi\in\Xi}\bar a^{\bm\xi}\,(\bm\xi\otimes\bm\xi)\,{\bm\nu}_\Sigma=\MA^\Xi{\bm\nu}_\Sigma$, assertion $(i)$ follows.

$(ii)$ By Lemma~\ref{t:AXi}$(iii)$ and the decomposition
\[
E\bfu=e(\bfu)\,\Leb{N}+(\bfu^+-\bfu^-)\odot{\bm\nu}_\bfu\,\Haus{N-1}\res J_\bfu+E^c\bfu ,
\]
we have
\[
\begin{split}
((\MA:E\bfu))_\Xi&=\MA^\Xi:e(\bfu)\,\Leb{N}+\MA^\Xi:\big((\bfu^+-\bfu^-)\odot{\bm\nu}_\bfu\big)\Haus{N-1}\res J_\bfu+\MA^\Xi:E^c\bfu .
\end{split}
\]
The first term equals $\MA:e(\bfu)\,\Leb{N}$, since $\MA^\Xi=\MA$ $\Leb{N}$-a.e.; the last one is singular with respect to $\Leb{N}$ and vanishes on $J_\bfu$, since $E^c\bfu$ vanishes on sets which are $\sigma$-finite with respect to $\Haus{N-1}$. By \eqref{eq:ident12} and $(i)$ with $\Sigma=J_\bfu$, $\MA^\Xi:\big((\bfu^+-\bfu^-)\odot{\bm\nu}_\bfu\big)=(\bfu^+-\bfu^-)\cdot\MA^\Xi{\bm\nu}_\bfu=\frac12\big(\Trpv{\MA}{J_\bfu}+\Trmv{\MA}{J_\bfu}\big)\cdot(\bfu^+-\bfu^-)$ $\Haus{N-1}$-a.e.\ on $J_\bfu$.

$(iii)$ By Lemma~\ref{t:AXi}$(iii)$, the difference is $(\MA^\Xi-\MA^{\Xi'}):E\bfu$. By $(ii)$, applied with both frames, the absolutely continuous parts and the restrictions to $J_\bfu$ of the two measures coincide, since the normal traces of $\MA$ do not depend on the frame. Hence the difference reduces to $(\MA^\Xi-\MA^{\Xi'}):E^c\bfu$.
\end{proof}

\begin{remark}[Invariance of the slicing pairing for $BV$ matrix fields]\label{invariance}
We emphasize that the previous definition involves the particular choice of the set $\Xi$, and that the class $\BV^\infty_{\Xi}(\Omega;\R^{N\times N}_{\rm sym})$ itself depends on $\Xi$.
 We recall that $\Xi$ is a finite set of directions required to reconstruct a symmetric matrix from rank-one symmetric matrices.
If $\MA$ belongs to $\BV^\infty_{\Xi}(\Omega;\R^{N\times N}_{\rm sym})\cap\BV^\infty_{\Xi'}(\Omega;\R^{N\times N}_{\rm sym})$ for two frames $\Xi$ and $\Xi'$, then, by Proposition~\ref{t:framecomparison}, the measures $((\MA:E\bfu))_\Xi$ and $((\MA:E\bfu))_{\Xi'}$ have the same absolutely continuous part and the same restriction to $J_\bfu$, so that they can differ only on the Cantor part of $E\bfu$; we do not know whether this may actually happen.
 The invariance is fully recovered under suitable regularity assumptions. For instance, 
 if 
 $\MA\in BV(\Omega;\mathbb{R}^{N\times N}_{\rm sym})\cap L^\infty(\Omega;\mathbb{R}^{N\times N}_{\rm sym})$, then the definition \eqref{eq:pairingBDnew} is independent of $\Xi$, since the slicing formula reconstructs the integral $\int_\Omega \varphi \MA^* : dE\bfu$, which is intrinsically coordinate-independent (see the last assertion of Theorem \ref{consistenza}). In this case, we use the notation $((\MA:E\bfu))$ for the pairing. We note, however, that the conditions $\MA\in\BV^\infty_{\Xi}(\Omega;\R^{N\times N}_{\rm sym})$ and $\MA\in\BV^\infty_{\Xi'}(\Omega;\R^{N\times N}_{\rm sym})$ for two different frames do not, in general, imply that $\MA\in BV(\Omega;\R^{N\times N}_{\rm sym})$: for instance, the field $\MA$ of Remark~\ref{rem:strict} belongs to $\BV^\infty_{\Xi}\cap\BV^\infty_{\Xi'}$, with $\Xi=\{{\bf e}_1,{\bf e}_2,\frac{{\bf e}_1+{\bf e}_2}{\sqrt2}\}$ and $\Xi'=\{{\bf e}_1,{\bf e}_2,\frac{{\bf e}_1-{\bf e}_2}{\sqrt2}\}$, since its coefficients with respect to both frames are $h(x_2)$, $H(x_2-\frac12)$ and $0$, but $\MA\notin BV$. We also point out that, by Proposition~\ref{t:framecomparison}$(ii)$, the interfacial work on the jump set of $\bfu$, for instance along a crack, is given by $\frac12\big(\Trpv{\MA}{J_\bfu}+\Trmv{\MA}{J_\bfu}\big)\cdot(\bfu^+-\bfu^-)\,\Haus{N-1}\res J_\bfu$ for every frame, in accordance with Theorem~\ref{t:pairingBD}$(ii)$.
\end{remark}

\subsection{The Gauss-Green formula for the slicing pairing}

We begin by establishing a preliminary result that will be instrumental in proving the Gauss-Green formula for possibly unbounded $BD$ functions. Specifically, we extend \eqref{eq:slicingdiverge} to the integrand $\bfu^*$, with $\bfu\in BD(\Omega)$ possibly unbounded.

Let $\MA\in\BV^\infty_{\Xi}(\Omega;\R^{N\times N}_{\rm sym})$ and $\bfu\in BD(\Omega)$ satisfy \eqref{assumXi} and $\bfu^*\cdot\bm\xi\in L^1_{\rm loc}(\Omega,|D_{\bm\xi}a^{\bm\xi}|)$ for every $\bm\xi\in\Xi$. Since, by \eqref{assumXi}, $\bfu^*$ is defined $|D_{\bm\xi}a^{\bm\xi}|$-a.e.\ in $\Omega$ for every $\bm\xi\in\Xi$, we can define the Radon measure
\begin{equation}\label{f:udivXi}
(\bfu^*\cdot\DIV\MA)_\Xi:=\sum_{\bm\xi\in\Xi}(\bfu^*\cdot\bm\xi)\,D_{\bm\xi}a^{\bm\xi}
\end{equation}
in $\Omega$. By \eqref{eq:divslic_dist}, if $\bfu^*\in L^1_{\rm loc}(\Omega,|\DIV\MA|;\R^N)$, then $(\bfu^*\cdot\DIV\MA)_\Xi=\bfu^*\cdot\DIV\MA$; in general, however, the integrability of $\bfu^*$ with respect to $|\DIV\MA|$ is not required in \eqref{f:udivXi}.

\begin{proposition}\label{L1DivA2233nuovo}
Let $\MA\in\BV^\infty_{\Xi}(\Omega;\R^{N\times N}_{\rm sym})$. For every $\bfu \in BD(\Omega)$ satisfying \eqref{assumXi}, the assumptions 
\begin{equation}\label{eq:integrabslice1}
\bfu^*\cdot \bm\xi\in L^1(\Omega,|D_{\bm\xi} {a}^{\bm\xi}|)\,\, \mbox{ for every } \,\, \bm\xi\in\Xi,
\end{equation}
and
\begin{equation}
\int_{\bm \xi^\perp}\left(\int_{\Omega^{{\bm\xi}}_y
} |({\hat{u}}_{y }^{{\bm\xi}})^*|(s) \,d|D{a}_y^{\bm\xi}|(s)\right)\,d\Leb{N-1}(y) <+\infty\,\, \mbox{ for every } \,\, \bm\xi\in\Xi\,,
\label{eq:integrabslice}
\end{equation} 
are equivalent. Moreover, under either of the equivalent conditions above, 
we have
\begin{equation}
(\bfu^*\cdot\DIV\MA)_\Xi(\Omega) =
\sum_{{\bm\xi}\in\Xi}\int_{\bm \xi^\perp}\left(\int_{\Omega^{{\bm\xi}}_y
} ({\hat{u}}_{y }^{{\bm\xi}})^*(s) \,dD{a}_y^{\bm\xi}(s)\right)\,d\Leb{N-1}(y)\,.
\label{eq:slicingdiverge2nuovo}
\end{equation}
\end{proposition}
\begin{proof}
The equivalence between \eqref{eq:integrabslice1} and \eqref{eq:integrabslice} is an immediate consequence of \eqref{eq:slicdirectder} combined with \cite[Corollary 2.29]{AFP}, and \eqref{eq:disintegrationf} with $f=|\bfu^*\cdot\xi|$, and Lemma~\ref{t:identification}, which yield  
$$
 \int_{\Omega} |\bfu^*\cdot \bm\xi|\, d|D_{\bm\xi} {a}^{\bm\xi}|
 =\int_{{\bm\xi}^\perp}\left(
\int_{\Omega^{{\bm\xi}}_y
} |
({\hat{u}}_{y }^{{\bm\xi}})^*
|(s) \,d|D{a}_y^{\bm\xi}|(s)\right)\,d\Leb{N-1}(y)\,\, \mbox{ for every } \,\, \bm\xi\in\Xi\,.
$$
By \eqref{f:udivXi}, \eqref{eq:integrabslice1}, \eqref{eq:slicdirectder3} with $v={a}^{\bm\xi}$ and $f=\bfu^*\cdot\xi$, and Lemma~\ref{t:identification},  
we get
\begin{equation*}
\begin{split}
(\bfu^*\cdot\DIV\MA)_\Xi(\Omega)
& = \sum_{{\bm\xi}\in\Xi} \int_{\Omega} (\bfu^*\cdot \bm\xi)\, d(D_{\bm\xi} {a}^{\bm\xi}) \\
& =\sum_{{\bm\xi}\in\Xi}\int_{\bm \xi^\perp}\left(\int_{\Omega^{{\bm\xi}}_y
} ({\hat{u}}_{y }^{{\bm\xi}})^*(s) \,dD{a}_y^{\bm\xi}(s)\right)\,d\Leb{N-1}(y)\,.
\end{split}
\label{eq:divslictris}
\end{equation*}
This concludes the proof of \eqref{eq:slicingdiverge2nuovo}. Applying the same argument to $\chi_B\,\bfu^*\cdot\bm\xi$, we see that \eqref{eq:slicingdiverge2nuovo} also holds with $(\bfu^*\cdot\DIV\MA)_\Xi(B)$ and $B^{\bm\xi}_y$ in place of $(\bfu^*\cdot\DIV\MA)_\Xi(\Omega)$ and $\Omega^{\bm\xi}_y$, for every Borel set $B\subset\Omega$. 
\end{proof}

The following result shows that, under \eqref{assumXi}, the distributional pairing \eqref{eq:pairingBD} makes sense also for unbounded $\bfu$, is a Radon measure, and coincides with the slicing pairing; for bounded $\bfu$ it gives the consistency of the two pairings without assumption \eqref{f:ipoBD} (compare with Theorem~\ref{consistenza}$(iii)$).

\begin{theorem}\label{t:leibnizconsistency}
Let $\MA\in\BV^\infty_{\Xi}(\Omega;\R^{N\times N}_{\rm sym})$ and $\bfu\in BD(\Omega)$ satisfy \eqref{assumXi} and
\begin{equation}\label{f:integrabloc}
\bfu^*\cdot\bm\xi\in L^1_{\rm loc}(\Omega,|D_{\bm\xi}a^{\bm\xi}|)\qquad\text{for every }\bm\xi\in\Xi .
\end{equation}
Then the linear functional
\begin{equation}\label{f:pairingunbdd}
\pscal{(\MA:E\bfu)}{\varphi}:=-\int_\Omega\varphi\,d(\bfu^*\cdot\DIV\MA)_\Xi-\int_\Omega\MA:[\bfu\odot\nabla\varphi]\,dx,
\end{equation}
for $\varphi\in C^\infty_c(\Omega)$,
coincides with $((\MA:E\bfu))_\Xi$; in particular, it is a Radon measure in $\Omega$. Equivalently, since $-\int_\Omega\MA:[\bfu\odot\nabla\varphi]\,dx=\pscal{\Div(\MA\bfu)}{\varphi}$, the divergence of $\MA\bfu$ is a Radon measure in $\Omega$ and the Leibniz formula
\begin{equation}\label{f:leibnizXi}
\Div(\MA\bfu)=(\bfu^*\cdot\DIV\MA)_\Xi+((\MA:E\bfu))_\Xi
\end{equation}
holds in the sense of measures in $\Omega$. If $\bfu^*\in L^1_{\rm loc}(\Omega,|\DIV\MA|;\R^N)$, then $(\bfu^*\cdot\DIV\MA)_\Xi=\bfu^*\cdot\DIV\MA$, so that \eqref{f:pairingunbdd} is the distributional pairing \eqref{eq:pairingBD} and \eqref{f:leibnizXi} reduces to \eqref{f:integrBD}. In particular, if $\bfu\in BD(\Omega)\cap L^\infty(\Omega;\R^N)$ satisfies \eqref{assumXi}, then $((\MA:E\bfu))_\Xi=(\MA:E\bfu)$.
\end{theorem}
\begin{proof}
Let $\varphi\in C^\infty_c(\Omega)$. Since $\MA=\sum_{\bm\xi\in\Xi}a^{\bm\xi}\,\bm\xi\otimes\bm\xi$ and $(\bm\xi\otimes\bm\xi):(\bfu\odot\nabla\varphi)=(\bfu\cdot\bm\xi)(\nabla\varphi\cdot\bm\xi)$, Fubini's theorem gives
\[
\int_\Omega\MA:[\bfu\odot\nabla\varphi]\,dx=\sum_{\bm\xi\in\Xi}\int_{\bm\xi^\perp}\left(\int_{\Omega^{\bm\xi}_y}a^{\bm\xi}_y(s)\,\hat u^{\bm\xi}_y(s)\,(\varphi^{\bm\xi}_y)'(s)\,ds\right)d\Leb{N-1}(y).
\]
For $\Leb{N-1}$-a.e.\ $y$, the functions $a^{\bm\xi}_y\in BV_{\rm loc}(\Omega^{\bm\xi}_y)\cap L^\infty$ and $\hat u^{\bm\xi}_y\in BV_{\rm loc}(\Omega^{\bm\xi}_y)$, so that $\hat u^{\bm\xi}_y$ is locally bounded and, by the definition \eqref{f:pairing} of the pairing for $N=1$ and by Theorem~\ref{t:pairingCD3} for $N=1$,
\[
-\int_{\Omega^{\bm\xi}_y}a^{\bm\xi}_y\,\hat u^{\bm\xi}_y\,(\varphi^{\bm\xi}_y)'\,ds=\int_{\Omega^{\bm\xi}_y}\varphi^{\bm\xi}_y\,(\hat u^{\bm\xi}_y)^*\,dDa^{\bm\xi}_y+\int_{\Omega^{\bm\xi}_y}\varphi^{\bm\xi}_y\,d(a^{\bm\xi}_y,D\hat u^{\bm\xi}_y).
\]
On the other hand, by \eqref{eq:slicdirectder3}, \eqref{f:integrabloc} and Lemma~\ref{t:identification},
\[
\int_\Omega\varphi\,(\bfu^*\cdot\bm\xi)\,d(D_{\bm\xi}a^{\bm\xi})=\int_{\bm\xi^\perp}\left(\int_{\Omega^{\bm\xi}_y}\varphi^{\bm\xi}_y\,(\hat u^{\bm\xi}_y)^*\,dDa^{\bm\xi}_y\right)d\Leb{N-1}(y).
\]
All the integrals above are absolutely convergent with respect to $(y,s)$: the first one by \eqref{f:integrabloc}, \cite[Corollary 2.29]{AFP} and Lemma~\ref{t:identification}, and the second one since $|(a^{\bm\xi}_y,D\hat u^{\bm\xi}_y)|\le\|a^{\bm\xi}\|_{L^\infty}|D\hat u^{\bm\xi}_y|$ and by \eqref{eq:slicefq}. Summing up, the terms containing $Da^{\bm\xi}_y$ cancel, and we obtain
\[
\pscal{(\MA:E\bfu)}{\varphi}=\sum_{\bm\xi\in\Xi}\int_{\bm\xi^\perp}\left(\int_{\Omega^{\bm\xi}_y}\varphi^{\bm\xi}_y\,d(a^{\bm\xi}_y,D\hat u^{\bm\xi}_y)\right)d\Leb{N-1}(y)=\pscal{((\MA:E\bfu))_\Xi}{\varphi}.
\]
The remaining assertions follow from Theorem~\ref{consistenza}$(i)$, from \eqref{eq:divslic_dist}, and from the fact that, if $\bfu$ is bounded, then \eqref{f:integrabloc} holds and, by Remark~\ref{rem:HXi}$(i)$, \eqref{assumXi} implies \eqref{assum}.
\end{proof}

The next proposition provides a Gauss-Green formula by slicing, which will be the key ingredient in the proof of the Gauss-Green formula, see Theorem~\ref{mmmmm} below.

We say that $\MA\in\BV^\infty_{\Xi,{\rm loc}}(\R^N;\R^{N\times N}_{\rm sym})$ if $\MA\in\BV^\infty_{\Xi}(U;\R^{N\times N}_{\rm sym})$ for every bounded open set $U\subset\R^N$.

\begin{proposition}\label{main}
 Let $\MA\in\BV^\infty_{\Xi,{\rm loc}}(\R^N;\R^{N\times N}_{\rm sym})$ and $\bfu\in BD_{\rm loc}(\R^N)$. Let $F\subset \R^N$ be a bounded set of finite perimeter. Assume that 
 \begin{equation}
\int_{\bm \xi^\perp}\left(\int_{(F^1)^{{\bm\xi}}_y\cup(\partial^*F)^{{\bm\xi}}_y
} |({\hat{u}}_{y }^{{\bm\xi}})^*|(s) \,d|D{a}_y^{\bm\xi}|(s)\right)\,d\Leb{N-1}(y) <+\infty\,\, \mbox{ for every } \,\, \bm\xi\in\Xi\,.
\label{eq:integrabslice22}
\end{equation} 
 Then the following Gauss-Green formulas by slicing hold
\begin{align}
&
\sum_{{\bm\xi}\in\Xi}\int_{{\bm\xi}^\perp}\left(
\int_{(F^1)^{{\bm\xi}}_y
} ({\hat{u}}_{y }^{{\bm\xi}})^*(s) \,dD{a}_y^{\bm\xi}(s)\right)\,d\Leb{N-1}(y) + \int_{F^1}\,d((\MA:E\bfu))_\Xi \label{eq:gaussgreen1bisww}\\
& =-\sum_{{\bm\xi}\in\Xi}\int_{{\bm\xi}^\perp}
\left( 
\sum_{s\in(\partial^*F)^{{\bm\xi}}_y} 
({\hat{u}}_{y }^{{\bm\xi}})^+(s) ({a_y^{\bm\xi}})^+(s) \, \nu_{F^{{\bm\xi}}_y}
(s)\,
\right)
\,d\Leb{N-1}(y)\,, \nonumber
\\
& \sum_{{\bm\xi}\in\Xi}\int_{{\bm\xi}^\perp}\left(
\int_{(F^1)^{{\bm\xi}}_y\cup (\partial^*F)^{{\bm\xi}}_y
} ({\hat{u}}_{y }^{{\bm\xi}})^*(s) \,dD{a}_y^{\bm\xi}(s)\right)\,d\Leb{N-1}(y) + \int_{F^1\cup\partial^*F}\,d((\MA:E\bfu))_\Xi \label{eq:gaussgreen2bis}\\
 & =-\sum_{{\bm\xi}\in\Xi}\int_{{\bm\xi}^\perp}
\left(
\sum_{s\in(\partial^*F)^{{\bm\xi}}_y}
({\hat{u}}_{y }^{{\bm\xi}})^-(s) ({a_y^{\bm\xi}})^-(s) \, \nu_{F^{{\bm\xi}}_y}(s)\,
\right)
\,d\Leb{N-1}(y)\,, \nonumber
\end{align}
where
$$
\nu_{F^{\bm\xi}_y}(s) = -\operatorname{sign}\left( {\bm\nu}_F(y + s\bm\xi) \cdot \bm\xi \right),
$$
${\bm\nu}_F$ denotes the measure theoretic outer unit normal to $F$ on $\partial^*F$ (so that $\nu_{F^{\bm\xi}_y}(s)$ is the interior unit normal to $F^{\bm\xi}_y$ at $s$), and $({a_y^{\bm\xi}})^\pm(s)$, $({\hat{u}}_{y }^{{\bm\xi}})^\pm(s)$ are the one-sided limits at $s\in\partial^*(F^{\bm\xi}_y)$ from the side of $F^{\bm\xi}_y$ and from the opposite side, respectively, as in Theorem~\ref{t:GGN=1}.
\end{proposition}
\begin{proof} We only prove \eqref{eq:gaussgreen1bisww}, the proof of the other one being similar.
For every $\bm\xi\in\Xi$ and for $\Leb{N-1}$-a.e. $y\in \bm\xi^\perp$,  
 by Theorem~\ref{t:GGN=1} we get
$$
\int_{(F^1)^{{\bm\xi}}_y} ({\hat{u}}_{y }^{{\bm\xi}})^*(s)\,dD{a}_y^{{\bm\xi}}(s)+\int_
{(F^1)^{{\bm\xi}}_y} d({a}_y^{{\bm\xi}},D{\hat{u}}_{y }^{{\bm\xi}})\,(s)
=-\sum_{s\in (\partial^*F)^{{\bm\xi}}_y}
(a^{{\bm\xi}}_y)^+(s)\,\nu_{F^{{\bm\xi}}_y}(s)\, ({\hat{u}}_{y }^{{\bm\xi}})^+(s),
$$
where we used that, for $\Leb{N-1}$-a.e. $y\in \bm\xi^\perp$, $F^{\bm\xi}_y$ is a finite union of intervals (up to a null set), $\partial^*(F^{\bm\xi}_y)=(\partial^*F)^{{\bm\xi}}_y$ and 
$
(F^1)_y^{\bm{\xi}} =  (F_y^{\bm{\xi}})^1
$. These properties follow from \cite[Theorem 3.108]{AFP} applied to $\chi_F\in BV(\R^N)$, since $J_{\chi_F}=\partial^*F$ and $S_{\chi_F}\setminus J_{\chi_F}$ is $\Haus{N-1}$-negligible, so that for $\Leb{N-1}$-a.e.\ $y$ the line $y+\R\bm\xi$ does not meet it: the points of $(F^1)^{\bm\xi}_y$ are then points of approximate continuity of $(\chi_F)^{\bm\xi}_y$ with value $1$, and the jump points of $(\chi_F)^{\bm\xi}_y$ are exactly the points of $(\partial^*F)^{\bm\xi}_y$.
We will integrate over $\bm\xi^\perp$ and sum over ${\bm\xi}\in\Xi$.
Notice that the finiteness of the integral in $y$ of the right hand side is ensured by \eqref{eq:integrabslice22} and by the estimate $|({a}_y^{{\bm\xi}},D{\hat{u}}_{y }^{{\bm\xi}})|\le\|a^{\bm\xi}\|_{L^\infty}|D{\hat{u}}_{y }^{{\bm\xi}}|$ together with \eqref{eq:slicefq}.
By using \eqref{eq:pairingBDnew} we have
\begin{equation*}\label{eq:pairingBDnew1}
\int_{F^1}\,d((\MA:E\bfu))_\Xi
  =
\sum_{{\bm\xi}\in\Xi}\int_{{\bm\xi}^\perp}
\left(\int_
{(F^1)^{{\bm\xi}}_y} 
d({a}_y^{{\bm\xi}},D{\hat{u}}_{y }^{{\bm\xi}})(s)
\right)\,d\Leb{N-1}(y).
\end{equation*}
This concludes the proof of \eqref{eq:gaussgreen1bisww}. The second identity \eqref{eq:gaussgreen2bis} follows analogously by using the second formula in Theorem~\ref{t:GGN=1}, which gives
$$
\int_{(F^1)^{{\bm\xi}}_y\cup(\partial^*F)^{{\bm\xi}}_y} ({\hat{u}}_{y }^{{\bm\xi}})^*(s)\,dD{a}_y^{{\bm\xi}}(s)+\int_
{(F^1)^{{\bm\xi}}_y\cup(\partial^*F)^{{\bm\xi}}_y} d({a}_y^{{\bm\xi}},D{\hat{u}}_{y }^{{\bm\xi}})\,(s)
=-\sum_{s\in (\partial^*F)^{{\bm\xi}}_y}
(a^{{\bm\xi}}_y)^-(s)\,\nu_{F^{{\bm\xi}}_y}(s)\, ({\hat{u}}_{y }^{{\bm\xi}})^-(s)
$$
for $\Leb{N-1}$-a.e. $y\in \bm\xi^\perp$.
\end{proof}

We now combine the previous results to derive the Gauss-Green formula for possibly unbounded $BD$ functions.

\begin{theorem} \label{mmmmm} Let $\MA\in\BV^\infty_{\Xi,{\rm loc}}(\R^N;\R^{N\times N}_{\rm sym})$ and $\bfu\in BD_{\rm loc}(\R^N)$ satisfy \eqref{assumXi} (with $\Omega=\R^N$). Let $F\subset \R^N$ be a bounded set of finite perimeter. Assume that 
\begin{equation}\label{eq:integrabslice1333}
\bfu^*\cdot \bm\xi\in L^1_{\rm{loc}}(\R^N,|D_{\bm\xi} {a}^{\bm\xi}|)\,\, \mbox{ for every } \,\, \bm\xi\in\Xi,
\qquad
 \bfu^\pm\in L^1_{\Haus{N-1}\res\partial^*F,{\rm{loc}}}(\R^N;\R^N),
  \end{equation}
  where $\bfu^\pm$ are the traces of $\bfu$ on $\partial^* F$ defined in Proposition~\ref{prop:rectifiable}.
Then
\begin{align*}
(\bfu^*\cdot\DIV\MA)_\Xi(F^1)+((\MA:E\bfu))_\Xi(F^1) &=- \int_{\partial^*F} \bfu^+\cdot \Trpv{\MA}{\partial^*F}\, d\Haus{N-1}\,,\\
(\bfu^*\cdot\DIV\MA)_\Xi(F^1\cup \partial^*F)+((\MA:E\bfu))_\Xi(F^1\cup \partial^*F) &=- \int_{\partial^*F} \bfu^-\cdot \Trmv{\MA}{\partial^*F}\, d\Haus{N-1} \,,
\end{align*}
where $(\bfu^*\cdot\DIV\MA)_\Xi$ is the measure defined in \eqref{f:udivXi} with $\Omega=\R^N$, and $\Trpmv{\MA}{\partial^*F}$
are the normal traces of \(\MA\) on \(\partial^* F\) oriented
by the interior unit normal vector. 
\end{theorem}
\begin{proof}

By the second assumption in \eqref{eq:integrabslice1333} and \eqref{eq:tracesl111starmmmm5999} with ${\bf v}=\bfu^\pm$, and by Lemma~\ref{t:BDpointwise}$(iv)$ with $\Sigma=\partial^*F$ (oriented by the interior normal), which identifies, for $\Leb{N-1}$-a.e.\ $y$, the slices of $\bfu^\pm\cdot\bm\xi$ on $(\partial^*F)^{\bm\xi}_y$ with the one-sided limits $(\hat u^{\bm\xi}_y)^\pm$,
we obtain 
\begin{equation*}
\int_{\partial^*F} \bfu^\pm\cdot \Trpmv{\MA}{\partial^*F}\, d\Haus{N-1}= \sum_{{\bm\xi}\in\Xi}\int_{{\bm\xi}^\perp}
\left( 
\sum_{s\in\partial^*F^{{\bm\xi}}_y} 
({\hat{u}}_{y }^{{\bm\xi}})^\pm(s)  ({a_y^{\bm\xi}})^\pm(s)\,\nu_{(\partial^* F)^{{\bm\xi}}_y}(s)
\right)
\,d\Leb{N-1}(y).
\end{equation*}
Since $\partial^*F$ is oriented by the interior unit normal $-{\bm\nu}_F$, by \eqref{eq:tracesl111pm22} we have $\nu_{(\partial^* F)^{{\bm\xi}}_y}(s)=-\operatorname{sign}({\bm\nu}_F(y+s{\bm\xi})\cdot{\bm\xi})=\nu_{F^{{\bm\xi}}_y}(s)$. Hence the right-hand sides of \eqref{eq:gaussgreen1bisww} and \eqref{eq:gaussgreen2bis} coincide with $-\int_{\partial^*F} \bfu^+\cdot \Trpv{\MA}{\partial^*F}\, d\Haus{N-1}$ and $-\int_{\partial^*F} \bfu^-\cdot \Trmv{\MA}{\partial^*F}\, d\Haus{N-1}$, respectively.
Combining these facts with Propositions \ref{L1DivA2233nuovo} and \ref{main}, applied in a ball containing $\overline F$, with the localized form of \eqref{eq:slicingdiverge2nuovo} on the Borel sets $F^1$ and $F^1\cup\partial^*F$ (see the end of the proof of Proposition~\ref{L1DivA2233nuovo}); note that the first assumption in \eqref{eq:integrabslice1333} implies \eqref{eq:integrabslice22}, by the equivalence in Proposition~\ref{L1DivA2233nuovo}, we obtain the two Gauss-Green formulas stated above.
\end{proof}
\begin{remark} \label{strongerthan}
If $F \subset \mathbb{R}^N$ is a bounded open set with Lipschitz boundary, then the second assumption in \eqref{eq:integrabslice1333} is satisfied (see Theorem \ref{Babadjian}), $F^1=F$, $\Haus{N-1}(\partial F\setminus\partial^*F)=0$ and $\bfu^+=\gamma^+(\bfu)$ is the interior trace of $\bfu$ on $\partial F$. Hence, if also \eqref{assumXi} and the first assumption in \eqref{eq:integrabslice1333} hold, then it holds
$$
(\bfu^*\cdot\DIV\MA)_\Xi(F)+((\MA:E\bfu))_\Xi(F)  =- \int_{\partial F} \gamma^+(\bfu)\cdot \Trpv{\MA}{\partial F}\, d\Haus{N-1}.
$$
In particular, let $\MA\in C^1(\R^N;\R^{N\times N}_{\rm sym})$ and $\bfu\in BD_{\rm loc}(\R^N)$. Since $\DIV\MA$ and $D_{\bm\xi}a^{\bm\xi}$ are absolutely continuous with respect to $\Leb{N}$, with continuous densities, \eqref{assumXi} and the first assumption in \eqref{eq:integrabslice1333} hold, and $(\bfu^*\cdot\DIV\MA)_\Xi=\bfu\cdot\DIV\MA\,\Leb{N}$; moreover, $((\MA:E\bfu))_\Xi=\MA:E\bfu$ by Remark~\ref{rem:W11LD}, and $\Trpv{\MA}{\partial F}=-\MA\bm\nu_{F}$. Hence, by \eqref{eq:ident12}, the formula above becomes
$$
\int_{F} \bfu \cdot \DIV\MA\,dx + \int_{F}\MA:\,dE\bfu
= \int_{\partial F} \MA : (\gamma^+(\bfu) \odot \bm\nu_{F}) \,d\Haus{N-1},
$$
that is, \eqref{eq:ident12345} with $\Omega=F$, which is thus recovered as a particular case of Theorem~\ref{mmmmm}. For smooth $\MA$, this also extends \eqref{eq:ident1234} to $BD$ functions that are not necessarily bounded.
\end{remark}

\section{Comparison with the Kohn-Temam pairing in plasticity theory} \label{sec:comparisonTemam}

We refer to \cite[Ch.~II, Sec.~7]{Temam} and \cite{KT}. Let $\Omega$ be a bounded and Lipschitz subset of $\R^N$. Define  
\[
\Sigma(\Omega):=\Bigl\{\MA \in L^2(\Omega;\R^{N\times N}_{\mathrm{sym}}):\ 
\DIV\MA \in L^{N}(\Omega;\R^N),\ \MA^D \in L^\infty(\Omega;\R^{N\times N})\Bigr\},
\]
and
\[
U(\Omega):=\Bigl\{\bfu\in BD(\Omega):\ \Div\bfu\in L^2(\Omega)\Bigr\}.
\]
For $\bfu\in BD(\Omega)$ we denote by $E^D\bfu$ the deviatoric part of $E\bfu$ and, if $\bfu\in LD(\Omega)$, by $e^D(\bfu)$ the deviatoric part of $e(\bfu)$, that is,
\[
E^D\bfu:=E\bfu-\frac1N(\Div\bfu)\,{\bf I},\qquad e^D(\bfu):=e(\bfu)-\frac1N\,{\rm tr}(e(\bfu))\,{\bf I},
\]
where $\Div\bfu={\rm tr}\,E\bfu$ is a Radon measure.
Recall that,  by \eqref{eq:identitydev} for every $\bfu\in H^1(\Omega;\R^N)$,
\begin{equation*}\label{eq:identitydev1}
\MA^D:e^D(\bfu)
=
\MA:e(\bfu)
-\frac1N \tr(e(\bfu))\,\tr \MA.
\end{equation*}
Observe that, for $\bfu\in U(\Omega)$ and $\MA\in\Sigma(\Omega)$, we have $\tr \MA\Div\bfu\in L^1(\Omega)$. Hence, for $\bfu\in U(\Omega)$ and $\MA\in\Sigma(\Omega)$, one can define the following distributions (see \cite[Definitions~7.42-7.43]{Temam}):
\begin{equation}\label{eq:pairingBDTem22}
\langle (\MA:E\bfu),\varphi\rangle
:=
-\int_\Omega \varphi\, \bfu\cdot \DIV\MA\,dx
-\int_\Omega \MA:(\bfu\odot\nabla\varphi)\,dx,
\end{equation}
and
\begin{equation}\label{eq:pairingBDTem2}
\langle (\MA^D:E^D\bfu),\varphi\rangle
:=
-\frac1N\int_\Omega \varphi\,\tr \MA\,\Div\bfu\,dx
+\langle (\MA:E\bfu),\varphi\rangle,
\end{equation}
for every $\varphi\in C_c^\infty(\Omega)$.

It is proved that, whenever $\bfu\in U(\Omega)$ and $\MA\in\Sigma(\Omega)$, the distributions defined in
\eqref{eq:pairingBDTem22}-\eqref{eq:pairingBDTem2}
are bounded Radon measures on $\Omega$
(see \cite[Lemma~7.3]{Temam} and \cite[Theorem~3.2]{KT}).

Now, we will extend the notion of the pairing \eqref{eq:pairingBDTem2} to stresses with divergence measure.
We give two new definitions of $(\MA^D:E^D\bfu)$ by using the slicing approach of the previous section only for the term $(\MA:E\bfu)$; they differ in the way the remaining term is defined.
Indeed, for $\MA\in\BV^\infty_{\Xi}(\Omega;\R^{N\times N}_{\rm sym})$ and $\bfu\in BD(\Omega)$, the product of ${\rm tr}\,\MA$ and $\Div\bfu$ requires some care, since ${\rm tr}\,\MA$ is only defined $\Leb{N}$-a.e., while $\Div\bfu$ may have a singular part with respect to $\Leb{N}$; this already happens in the example of Section~\ref{sec:explicit_example}, where $\Div\bfu$ is concentrated on a line across which ${\rm tr}\,\MA$ jumps. We propose two ways to give a meaning to this term: by means of the precise representative of ${\rm tr}\,\MA$, under a minimal compatibility condition between $\Div\bfu$ and ${\rm tr}\,\MA$ (Subsection~\ref{sec:devprecise}), or, for every $\bfu\in BD(\Omega)$, by means of the Borel representative $\MA^\Xi$ of $\MA$ adapted to the frame $\Xi$ constructed in Section~\ref{sec:pairingunbounded} (Subsection~\ref{sec:devAXi}). The two definitions are compared in Subsection~\ref{sec:devcomparison}. In Subsections~\ref{sec:devprecise}--\ref{sec:devcomparison}, $\Omega$ is a general open subset of $\R^N$, unless otherwise stated; the assumption that $\Omega$ is bounded and Lipschitz is only needed for the comparison with \eqref{eq:pairingBDTem22}--\eqref{eq:pairingBDTem2}.

\subsection{\texorpdfstring{A definition via the precise representative of ${\rm tr}\,\MA$}{A definition via the precise representative of tr A}}\label{sec:devprecise}
Let $\bfu\in BD(\Omega)$ (so that $\Div\bfu$ is a Radon measure) and let
$\MA\in\BV^\infty_{\Xi}(\Omega;\R^{N\times N}_{\rm sym})$.
In order to give a meaning to the integral of ${\rm tr}\,\MA$ with respect to $\Div\bfu$ by means of the precise representative of ${\rm tr}\,\MA$, we assume the compatibility condition
\begin{equation}\label{assumtr}\tag{${\rm H}_{\rm tr}$}
|\Div\bfu|\big(S_{{\rm tr}\MA}\setminus J_{{\rm tr}\MA}\big)=0,
\end{equation}
where $S_{{\rm tr}\MA}$ and $J_{{\rm tr}\MA}$ are the approximate discontinuity set and the jump set of the scalar function ${\rm tr}\,\MA$. Under \eqref{assumtr}, the precise representative $({\rm tr}\,\MA)^*$ (see \eqref{eq:preciserepresentative}) is a Borel function defined $|\Div\bfu|$-a.e.\ in $\Omega$, and $|({\rm tr}\,\MA)^*(x)|\leq\|{\rm tr}\,\MA\|_{L^\infty(U)}$ for every open set $U\subset\Omega$ and every $x\in U\setminus(S_{{\rm tr}\MA}\setminus J_{{\rm tr}\MA})$.
Assumption \eqref{assumtr} is satisfied, for instance, in the following cases:
\begin{enumerate}
\item[(i)] $\Div\bfu\ll\Leb{N}$ (in particular, if $\bfu\in U(\Omega)$), since $\Leb{N}(S_{{\rm tr}\MA})=0$;
\item[(ii)] ${\rm tr}\,\MA\in BV(\Omega)$, since in this case $\Haus{N-1}(S_{{\rm tr}\MA}\setminus J_{{\rm tr}\MA})=0$, and $|\Div\bfu|\leq\sqrt N|E\bfu|$ vanishes on $\Haus{N-1}$-negligible sets.
\end{enumerate}
Under assumption \eqref{assumtr},
we define the following distribution:
for every $\varphi\in C_c^\infty(\Omega)$,
\begin{equation}\label{bbbbb}
\begin{split}
&\pscal{((\MA^D:E^D\bfu))^{*}_\Xi}{\varphi}
 :=\\
& - \frac{1}{N} \int_\Omega \varphi\, ({\rm tr}\,\MA)^*\,d(\Div\bfu) +\sum_{\bm{\xi}\in\Xi}\int_{\bm{\xi}^\perp}
\Bigg(\int_{\Omega^{\bm{\xi}}_y} \varphi_y^{{\bm\xi}}(s)\,
d(a_y^{\bm{\xi}},D\hat{u}_{y}^{\bm{\xi}})
\,(s)
\Bigg)\,d\Leb{N-1}(y).
\end{split}
\end{equation}

\begin{theorem}\label{t:deviatoricprecise}
Let $\MA\in\BV^\infty_{\Xi}(\Omega;\R^{N\times N}_{\rm sym})$ and $\bfu\in BD(\Omega)$ satisfying \eqref{assumtr}. Then  the distribution
$((\MA^D:~E^D\bfu))^{*}_\Xi$ is a Radon measure in $\Omega$, absolutely continuous with respect to \(|E\bfu|\), and there exists a positive constant $\widetilde{C}_\Xi$ such that     
\begin{equation}\label{MA}
|((\MA^D:E^D\bfu))^{*}_\Xi|(B)\leq \widetilde{C}_\Xi\|\MA\|_{L^\infty(U;\R^{N\times N})}|E\bfu|(B) \quad \mbox{for every Borel set $B \subset U \Subset \Omega$}.
\end{equation}
Moreover, if $\Omega$ is a bounded and Lipschitz subset of $\R^N$, for every $\bfu\in U(\Omega)\cap L^\infty(\Omega;\R^N)$  and $\MA\in
\BV^\infty_{\Xi}(\Omega;\R^{N\times N}_{\rm sym})\cap
\Sigma(\Omega)$ 
(and so assumptions \eqref{assum} and \eqref{assumtr} are satisfied) such that \eqref{f:ipoBD} holds, the measure
 $((\MA^D:~E^D\bfu))^{*}_\Xi$ coincides with the pairing defined in \eqref{eq:pairingBDTem2}; i.e.,
\begin{equation}\label{mmmmmm}
((\MA^D:E^D\bfu))^{*}_\Xi =(\MA^D:E^D\bfu),
\end{equation}
in the sense of measures.
Finally, if $\bfu\in C^1(\overline{\Omega};\R^N)$ and $\MA\in\BV^\infty_{\Xi}(\Omega;\R^{N\times N}_{\rm sym})\cap\Sigma(\Omega)$, then 
 \begin{equation}\label{eq:C1}
 ((\MA^D:E^D\bfu))^{*}_\Xi=\big(\MA:e(\bfu)
-\frac{1}{N}\,\tr \MA\, \tr(e(\bfu))\big)\Leb{N}=\MA^D:e^D(\bfu)\,\Leb{N}.
 \end{equation}
\end{theorem}
\begin{proof}
To prove condition \eqref{MA}, let us note that, by \eqref{eq:abscontBD11} and \eqref{eq:traceineq},
\begin{equation*}
\begin{split}
|((\MA^D:E^D\bfu))^{*}_\Xi|(B) & \leq\frac1N\int_B|({\rm tr}\,\MA)^*|\,d\,|\Div\bfu|+ |((\MA:E\bfu))_\Xi|(B) \\
&\leq\frac1N \|{\rm tr}\,\MA\|_{L^\infty(U)}|\Div\bfu|(B)+C_\Xi\|\MA\|_{L^\infty(U;\R^{N\times N}_{\rm sym})} |E\bfu|(B) \\
& \leq (1+C_\Xi)\|\MA\|_{L^\infty(U;\R^{N\times N}_{\rm sym})}|E\bfu|(B),
\end{split}
\end{equation*}
where in the first inequality we used that $({\rm tr}\,\MA)^*$ is defined $|\Div\bfu|$-a.e.\ by \eqref{assumtr}, and in the last one we used \eqref{eq:traceineq} both for $\MA$ and for the polar density of $E\bfu$, which gives $|\Div\bfu|=|{\rm tr}\,E\bfu|\leq\sqrt N|E\bfu|$.
Assertion \eqref{mmmmmm} follows by Theorem~\ref{consistenza}$(iii)$, Remark ~\ref{pipo1} and \eqref{eq:pairingBDTem2}, once we observe that, since $\Div\bfu\in L^2(\Omega)$, we have $|\Div\bfu|\ll\Leb{N}$, so that $({\rm tr}\,\MA)^*={\rm tr}\,\MA$ $|\Div\bfu|$-a.e.\ and the first term in \eqref{bbbbb} coincides with the first term in \eqref{eq:pairingBDTem2}.

On the other hand, if $\bfu\in C^1(\overline{\Omega};\R^N)$, then $\bfu\in LD(\Omega)$, so that $\Div\bfu={\rm tr}(e(\bfu))\,\Leb{N}\ll\Leb{N}$ and \eqref{assumtr} holds, and $((\MA:E\bfu))_\Xi=\MA:e(\bfu)\,\Leb{N}$ by Remark~\ref{rem:W11LD}. Hence, by \eqref{bbbbb} and Theorem~\ref{consistenza}$(i)$, we get
$$
((\MA^D:E^D\bfu))^{*}_\Xi
=\big(\MA:e(\bfu)
-\frac{1}{N}\tr \MA\,\tr(e(\bfu))\big)\Leb{N}.$$
The second equality in \eqref{eq:C1} follows from \eqref{eq:identitydev}.
\end{proof}

\begin{remark}\label{t:remdevBV}
If $\MA\in BV(\Omega;\R^{N\times N}_{\rm sym})\cap L^\infty(\Omega;\R^{N\times N}_{\rm sym})$, then ${\rm tr}\,\MA\in BV(\Omega)$, so that \eqref{assumtr} holds for every $\bfu\in BD(\Omega)$, and $({\rm tr}\,\MA)^*={\rm tr}\,(\MA^*)$ $|E\bfu|$-a.e.\ in $\Omega$. Hence, by the last assertion of Theorem~\ref{consistenza} and \eqref{eq:identitydev} applied with ${\bf B}$ equal to the polar density of $E\bfu$, we get
$$
((\MA^D:E^D\bfu))^{*}_\Xi=\MA^*:E\bfu-\frac1N\,{\rm tr}\,(\MA^*)\,\Div\bfu=(\MA^*)^D:E^D\bfu\,,
$$
\noindent and, in particular, in this case $((\MA^D:E^D\bfu))^{*}_\Xi$ does not depend on the frame $\Xi$.
\end{remark}

\subsection{A definition via a representative of \texorpdfstring{$\MA$}{A} adapted to the frame}\label{sec:devAXi}
Alternatively, the trace term can be defined for every $\bfu\in BD(\Omega)$, without assuming \eqref{assumtr}, by means of the Borel representative $\MA^\Xi$ of $\MA$ defined in \eqref{f:AXi}, which is adapted to the frame $\Xi$ and reproduces the slicing pairing, by Lemma~\ref{t:AXi}$(iii)$.

For $\MA\in\BV^\infty_{\Xi}(\Omega;\R^{N\times N}_{\rm sym})$, with the representatives of the coefficients $a^{\bm\xi}$ fixed as in \eqref{f:goodrep}, and $\bfu\in BD(\Omega)$ we define
\begin{equation}\label{f:deviatoricslicing}
((\MA^D:E^D\bfu))_\Xi:=(\MA^\Xi)^D:E^D\bfu=((\MA:E\bfu))_\Xi-\frac1N\,({\rm tr}\,\MA^\Xi)\,\Div\bfu .
\end{equation}
The second equality follows from \eqref{f:representationXi} and \eqref{eq:identitydev}: indeed, writing $E\bfu=\bm\sigma\,|E\bfu|$, where $\bm\sigma$ is the polar density of $E\bfu$, we have $\Div\bfu={\rm tr}\,\bm\sigma\,|E\bfu|$ and $E^D\bfu=\bm\sigma^D|E\bfu|$, so that $(\MA^\Xi)^D:E^D\bfu=\big(\MA^\Xi:\bm\sigma-\frac1N\,{\rm tr}\,\MA^\Xi\,{\rm tr}\,\bm\sigma\big)|E\bfu|$.
While the slicing pairing $((\MA:E\bfu))_\Xi=\MA^\Xi:E\bfu$ does not depend on the choice of the representatives of the coefficients $a^{\bm\xi}$ (see the comment after Lemma~\ref{t:AXi}), when $\Div\bfu$ has a nontrivial singular part the trace term in \eqref{f:deviatoricslicing}, and hence $((\MA^D:E^D\bfu))_\Xi$, may depend on this choice (see Remark~\ref{t:remdevcomparison}$(ii)$ below).

\begin{theorem}\label{t:deviatoric}
Let $\MA\in\BV^\infty_{\Xi}(\Omega;\R^{N\times N}_{\rm sym})$ and $\bfu\in BD(\Omega)$. Then:
\begin{itemize}
\item[$(i)$] $((\MA^D:E^D\bfu))_\Xi$ is a Radon measure in $\Omega$, absolutely continuous with respect to $|E^D\bfu|$, and
\begin{equation}\label{f:deviatoricestimate}
|((\MA^D:E^D\bfu))_\Xi|(B)\leq C_\Xi\|\MA\|_{L^\infty(U;\R^{N\times N})}\,|E^D\bfu|(B)\le C_\Xi\|\MA\|_{L^\infty(U;\R^{N\times N})}\,|E\bfu|(B)
\end{equation}
for every Borel set $B\subset U\Subset\Omega$;
\item[$(ii)$] if $\Div\bfu\ll\Leb{N}$, then $((\MA^D:E^D\bfu))_\Xi=((\MA:E\bfu))_\Xi-\frac1N\,{\rm tr}\,\MA\,\Div\bfu=((\MA^D:E^D\bfu))^*_\Xi$;
\item[$(iii)$] if $\Omega$ is a bounded and Lipschitz subset of $\R^N$, $\bfu\in U(\Omega)\cap L^\infty(\Omega;\R^N)$ and $\MA\in\BV^\infty_{\Xi}(\Omega;\R^{N\times N}_{\rm sym})\cap\Sigma(\Omega)$ are such that \eqref{f:ipoBD} holds, then
\[
((\MA^D:E^D\bfu))_\Xi =(\MA^D:E^D\bfu),
\]
where the right-hand side is defined in \eqref{eq:pairingBDTem2};
\item[$(iv)$] if $\bfu\in LD(\Omega)$ (in particular, if $\Omega$ is bounded and $\bfu\in C^1(\overline\Omega;\R^N)$), then $((\MA^D:E^D\bfu))_\Xi=\MA^D:e^D(\bfu)\,\Leb{N}$.
\end{itemize}
\end{theorem}
\begin{proof}
$(i)$ By \eqref{f:deviatoricslicing}, $((\MA^D:E^D\bfu))_\Xi=(\MA^\Xi)^D:\bm\sigma^D\,|E\bfu|$, where $(\MA^\Xi)^D:\bm\sigma^D$ is a Borel function which is bounded on every $U\Subset\Omega$, by Lemma~\ref{t:AXi}$(i)$; hence it is a Radon measure. Since ${\bf M}\mapsto{\bf M}^D$ is the orthogonal projection onto the trace-free matrices, $|(\MA^\Xi)^D|\le|\MA^\Xi|\le C_\Xi\|\MA\|_{L^\infty(U;\R^{N\times N})}$ on $U$ and $|\bm\sigma^D|\le|\bm\sigma|=1$. Therefore
\[
|((\MA^D:E^D\bfu))_\Xi|(B)\le\int_B|(\MA^\Xi)^D|\,|\bm\sigma^D|\,d|E\bfu|\le C_\Xi\|\MA\|_{L^\infty(U;\R^{N\times N})}\,|E^D\bfu|(B),
\]
since $|E^D\bfu|=|\bm\sigma^D|\,|E\bfu|\le|E\bfu|$.

$(ii)$ If $\Div\bfu\ll\Leb{N}$, then ${\rm tr}\,\MA^\Xi\,\Div\bfu={\rm tr}\,\MA\,\Div\bfu$, since $\MA^\Xi=\MA$ $\Leb{N}$-a.e. Moreover, \eqref{assumtr} holds and $({\rm tr}\,\MA)^*={\rm tr}\,\MA$ $\Leb{N}$-a.e., so that also the second equality follows from \eqref{bbbbb} and Theorem~\ref{consistenza}$(i)$.

$(iii)$ Since $\Div\bfu\in L^2(\Omega)$, we have $\Div\bfu\ll\Leb{N}$; hence the conclusion follows from $(ii)$ and Theorem~\ref{t:deviatoricprecise}.

$(iv)$ By Remark~\ref{rem:W11LD}, $((\MA:E\bfu))_\Xi=\MA:e(\bfu)\,\Leb{N}$, and $\Div\bfu={\rm tr}\,e(\bfu)\,\Leb{N}$. Hence, by $(ii)$ and \eqref{eq:identitydev},
$((\MA^D:E^D\bfu))_\Xi=\big(\MA:e(\bfu)-\frac1N\,{\rm tr}\,\MA\,{\rm tr}\,e(\bfu)\big)\Leb{N}=\MA^D:e^D(\bfu)\,\Leb{N}$.
\end{proof}

\subsection{Comparison of the two definitions}\label{sec:devcomparison}

\begin{remark}\label{t:remdevcomparison}
$(i)$ \emph{Relation between the two definitions.} Let $\MA\in\BV^\infty_{\Xi}(\Omega;\R^{N\times N}_{\rm sym})$ and $\bfu\in BD(\Omega)$ satisfy \eqref{assumtr}. By \eqref{bbbbb}, Theorem~\ref{consistenza}$(i)$ and \eqref{f:deviatoricslicing},
\[
((\MA^D:E^D\bfu))^*_\Xi-((\MA^D:E^D\bfu))_\Xi=\frac1N\big({\rm tr}\,\MA^\Xi-({\rm tr}\,\MA)^*\big)\Div\bfu=\frac1N\big({\rm tr}\,\MA^\Xi-({\rm tr}\,\MA)^*\big)\Div^s\bfu ,
\]
where $\Div^s\bfu$ is the singular part of $\Div\bfu$ with respect to $\Leb{N}$, since ${\rm tr}\,\MA^\Xi=({\rm tr}\,\MA)^*={\rm tr}\,\MA$ $\Leb{N}$-a.e. Hence the two definitions coincide if and only if ${\rm tr}\,\MA^\Xi=({\rm tr}\,\MA)^*$ $|\Div^s\bfu|$-a.e.; in particular, they coincide whenever $\Div\bfu\ll\Leb{N}$, and therefore in the Kohn--Temam setting of Theorems~\ref{t:deviatoricprecise} and \ref{t:deviatoric}$(iii)$. They may coincide also when $\Div^s\bfu\neq0$, as in the example of Section~\ref{sec:explicit_example} (see the end of that section).

$(ii)$ \emph{Dependence on the representative.} The function $\MA^\Xi$ depends on the Borel representatives of the coefficients $a^{\bm\xi}$ fixed in \eqref{f:goodrep}, which are determined only up to $\Leb{N}$-negligible sets. A modification of $a^{\bm\xi}$ on an $\Leb{N}$-negligible set changes $\bar a^{\bm\xi}$ only on the union of the lines $y+\R\bm\xi$ with $y$ in an $\Leb{N-1}$-negligible subset of $\bm\xi^\perp$; by \eqref{eq:slicefq}, such a set is negligible for $E\bfu\,\bm\xi\cdot\bm\xi$, and this is why $((\MA:E\bfu))_\Xi=\MA^\Xi:E\bfu$ does not depend on these choices. On the other hand, since $\{{\bf B}^{\bm\eta}\}_{\bm\eta\in\Xi}$ is the dual basis of $\{\bm\eta\otimes\bm\eta\}_{\bm\eta\in\Xi}$, we have $E\bfu=\sum_{\bm\eta\in\Xi}(E\bfu\,\bm\eta\cdot\bm\eta)\,{\bf B}^{\bm\eta}$, hence
\[
\Div\bfu={\rm tr}\,E\bfu=\sum_{\bm\eta\in\Xi}{\rm tr}({\bf B}^{\bm\eta})\,(E\bfu\,\bm\eta\cdot\bm\eta),
\]
and the measures $E\bfu\,\bm\eta\cdot\bm\eta$ with $\bm\eta\neq\bm\xi$ may charge lines parallel to $\bm\xi$. As a consequence, when $\Div^s\bfu\neq0$, the measure \eqref{f:deviatoricslicing} may depend on these choices. For instance, let $N=2$, let $\Xi$ be the standard frame of Section~\ref{sec:pairingbysliceBDbdd}, $\Omega=(-1,1)^2$,
\[
\MA(x):={\rm sign}(x_1)\,{\bf e}_2\otimes{\bf e}_2 ,\qquad \bfu(x):=H(x_1)\,{\bf e}_1 ,
\]
where $H$ is the Heaviside function. Then $\MA\in BV(\Omega;\R^{2\times2}_{\rm sym})\cap L^\infty(\Omega;\R^{2\times2}_{\rm sym})$, $a^{{\bf e}_1}=a^{({\bf e}_1+{\bf e}_2)/\sqrt2}=0$, $a^{{\bf e}_2}={\rm sign}(x_1)$, $E\bfu=\Haus{1}\res\{x_1=0\}\,{\bf e}_1\otimes{\bf e}_1$ and $\Div\bfu=\Haus{1}\res\{x_1=0\}$. For every $t\in[-1,1]$, the Borel function equal to ${\rm sign}(x_1)$ for $x_1\neq0$ and to $t$ on $\{x_1=0\}$ is a representative of $a^{{\bf e}_2}$ satisfying \eqref{f:goodrep}; with this choice, $\bar a^{{\bf e}_2}=t$ on the line $\{x_1=0\}$, which is parallel to ${\bf e}_2$. Since $((\MA:E\bfu))_\Xi=\MA^*:E\bfu=0$, \eqref{f:deviatoricslicing} gives
\[
((\MA^D:E^D\bfu))_\Xi=-\frac t2\,\Haus{1}\res\{x_1=0\},
\]
which depends on $t$. In contrast, $({\rm tr}\,\MA)^*=0$ on $\{x_1=0\}$, and \eqref{bbbbb} gives $((\MA^D:E^D\bfu))^*_\Xi=0=(\MA^*)^D:E^D\bfu$, in accordance with Remark~\ref{t:remdevBV}. Thus, when $\Div^s\bfu\neq0$, definition \eqref{f:deviatoricslicing} is meaningful only once the representatives of the coefficients $a^{\bm\xi}$ have been fixed.

$(iii)$ \emph{Summary.} Definition \eqref{f:deviatoricslicing} is available for every $\bfu\in BD(\Omega)$, without assuming \eqref{assumtr}, and satisfies the estimate \eqref{f:deviatoricestimate} in terms of $|E^D\bfu|$, but in general it depends on the choice of the representative $\MA^\Xi$ on the support of $\Div^s\bfu$. Definition \eqref{bbbbb} requires \eqref{assumtr} and satisfies only the estimate \eqref{MA} in terms of $|E\bfu|$, but it depends only on the equivalence class of $\MA$, its trace term does not depend on $\Xi$, and it reduces to $(\MA^*)^D:E^D\bfu$ for every $\MA\in BV(\Omega;\R^{N\times N}_{\rm sym})\cap L^\infty(\Omega;\R^{N\times N}_{\rm sym})$ (Remark~\ref{t:remdevBV}). The two definitions coincide whenever $\Div\bfu\ll\Leb{N}$.
\end{remark}

\section{An example: Mode I opening of an interface carrying a normal force}
\label{sec:explicit_example}

To illustrate the physical meaning of the pairing and how it overcomes the limitations of the classical Kohn-Temam theory, we consider a simple two-dimensional example describing a displacement jump across an interface carrying a concentrated normal force, as in cohesive (Barenblatt-type) cracks or adhesive layers (see \cite{Bar}).

Let $\Omega = (-1, 1)^2 \subset \mathbb{R}^2$ be the reference domain, divided into two subdomains by the interface $\Gamma = \{x \in \Omega : x_2 = 0\}$, with ${\bm\nu}_\Gamma={\bf e}_2$. We define a displacement field $\mathbf{u} \in BD(\Omega)\cap L^\infty(\Omega;\R^2)$ representing a normal opening with amplitude $\gamma > 0$
\begin{equation*}
\mathbf{u}(x) = \begin{pmatrix} 0 \\ \gamma H(x_2) \end{pmatrix},
\end{equation*}
where $H$ denotes the Heaviside function. Note that $\bfu \in SBV(\Omega;\R^2)$, and
\begin{equation*}
E\bfu = [\bfu^+-\bfu^-] \odot {\bm\nu}_\Gamma \Haus{1}\res \Gamma = \left(\gamma  \Haus{1}\res \Gamma\right)({\bf e}_2\otimes {\bf e}_2)\,.
\end{equation*}
The jump of $\bfu$ across the interface is purely normal (Mode I opening, see \cite{bourdin2008}), since it is aligned with the normal direction to the interface and no tangential slip occurs. 
 We consider a piecewise constant stress tensor, defined as
\begin{equation*}
\MA(x) =
\begin{pmatrix}
0 & 0 \\
0 & \sigma(x_2)
\end{pmatrix},
\qquad
\sigma(x_2)=
\begin{cases}
\sigma^+ & \text{if } x_2>0,\\
\sigma^- & \text{if } x_2<0,
\end{cases}
\end{equation*}
with $\sigma^\pm \in \mathbb{R}$. This tensor describes a normal traction state with no shear components. The jump $\sigma^+-\sigma^-$ of the normal traction across $\Gamma$ corresponds to a normal force per unit length concentrated on $\Gamma$, as in cohesive or adhesive interfaces. We point out that in Griffith's model the crack faces are traction-free, i.e., $\sigma^\pm=0$, while if $\sigma^+=\sigma^-$ no force acts on $\Gamma$, so that $\DIV\MA=0$ and the Kohn--Temam theory applies.  Note that $\MA\in BV(\Omega;\R^{2 \times 2}_{\rm sym})\cap L^\infty(\Omega; \R^{2 \times 2}_{\rm sym})$.

In this setting, the divergence of $\MA$ is a vector-valued measure concentrated on the interface $\Gamma$, namely
\begin{equation*}
\DIV \MA = \begin{pmatrix} 0 \\ (\sigma^+ - \sigma^-) \Haus{1} \res \Gamma \end{pmatrix}.
\end{equation*}
The singularity of $\DIV\MA$ with respect to $\Leb{2}$ prevents the application of the classical Kohn-Temam pairing theory, which consequently cannot be used to evaluate the mechanical work across the interface. The pairing $(\MA:E\bfu)$ defined in Section~\ref{sec:pairingBDbounded} overcomes this limitation by naturally capturing the interfacial contribution. Since $\bfu\in SBV(\Omega;\R^2)\cap L^\infty(\Omega;\R^2)$, this situation is also covered by the pairing between divergence-measure tensor fields and $BV$ functions introduced in \cite{DCSTensor}, which coincides with $(\MA:E\bfu)$ for symmetric $\MA$; the purpose of this example is to illustrate, in a simple explicit case, the interfacial work, the Gauss--Green formula of Theorem~\ref{main1} and the deviatoric pairings of Section~\ref{sec:comparisonTemam}.

Since $\MA\in BV(\Omega;\R^{2 \times 2}_{\rm sym})$ and \eqref{assum} holds, since $\bfu\in BV(\Omega;\R^2)$ (see Remark~\ref{rem:Hstar}$(i)$), by Remark \ref{rem:BDpairjump} we have
\begin{equation*}
(\MA : E\mathbf{u}) = \MA^*: E \bfu = \frac{\sigma^+ + \sigma^-}{2} \, \gamma \, \Haus{1} \res \Gamma.
\end{equation*}

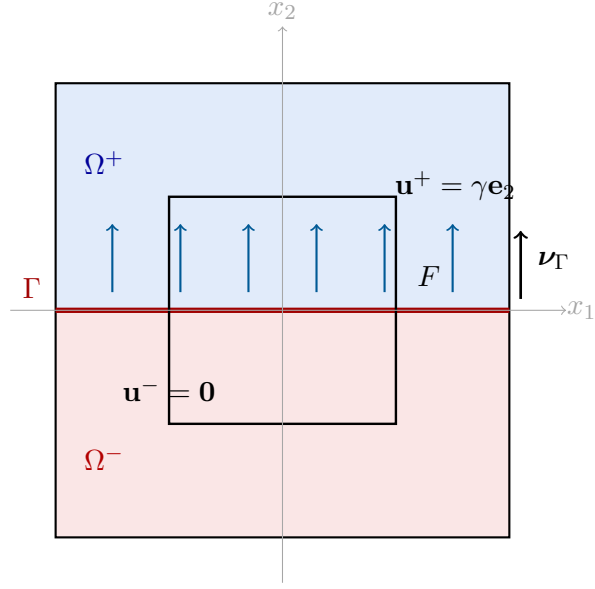
\begin{figure}[h]
\centering
\begin{tikzpicture}[scale=3]

\definecolor{upper}{RGB}{225,235,250}
\definecolor{lower}{RGB}{250,230,230}
\definecolor{crack}{RGB}{160,0,0}
\definecolor{disp}{RGB}{0,90,150}

\fill[upper] (-1,0) rectangle (1,1);
\fill[lower] (-1,-1) rectangle (1,0);

\draw[line width=0.8pt] (-1,-1) rectangle (1,1);

\draw[line width=2pt, crack] (-1,0) -- (1,0);

\draw[line width=0.9pt] (-0.5,-0.5) rectangle (0.5,0.5);

\node[anchor=west]
at (0.55,0.15)
{$F$};

\node[anchor=west, text=blue!60!black]
at (-0.92,0.65)
{$\Omega^{+}$};

\node[anchor=west, text=red!70!black]
at (-0.92,-0.65)
{$\Omega^{-}$};

\node[anchor=east, text=crack]
at (-1.02,0.10)
{$\Gamma$};

\draw[->,thin,gray!70]
(-1.2,0)--(1.25,0);

\draw[->,thin,gray!70]
(0,-1.2)--(0,1.25);

\node[gray!70]
at (1.32,0)
{$x_1$};

\node[gray!70]
at (0,1.32)
{$x_2$};

\draw[->, line width=1pt, black]
(1.05,0.05) -- (1.05,0.35);

\node[anchor=west]
at (1.08,0.22)
{$\boldsymbol{\nu}_{\Gamma}$};


\foreach \x in {-0.75,-0.45,-0.15,0.15,0.45,0.75}
{
\draw[->,line width=0.8pt,disp]
(\x,0.08) -- (\x,0.38);
}

\node[anchor=west]
at (0.45,0.55)
{$\mathbf u^{+}=\gamma{\mathbf e}_2$};

\node[anchor=west]
at (-0.75,-0.35)
{$\mathbf u^{-}=\mathbf 0$};

\end{tikzpicture}

\caption{The reference configuration $\Omega$ with the interface $\Gamma$ and the Lipschitz subdomain
$F$ crossing the interface.}
\label{fig:figura}
\end{figure}

This identity provides a consistent definition of the interfacial work associated with the discontinuity of the displacement field. Moreover, it shows that the generalized Gauss--Green formula (Theorem~\ref{main1}) remains valid even in the presence of singular divergences, thereby ensuring global energy balance on singular sets where classical divergence integrability fails.

Indeed, for the Lipschitz subdomain $F := (-\frac{1}{2}, \frac{1}{2})^2\subset\Omega$ (see Figure \ref{fig:figura}), the generalized Gauss--Green formula reads
\begin{equation}
\int_{F} \mathbf{u}^* \cdot d\DIV\MA + \int_{F}\,d(\MA:E\mathbf{u}) = \int_{\partial F} \mathbf{u}^+\cdot (\MA^+{\bm\nu}_{F})\, d\mathcal{H}^{1}.
\end{equation}
Now, evaluating both sides over the domain $F$ yields a perfect identity. Specifically, the left-hand side combines the singular divergence and the interfacial work contributions along the crack line $\Gamma$:
\begin{equation}
\int_{F} \mathbf{u}^* \cdot d\DIV\MA + \int_{F}\,d(\MA:E\mathbf{u}) = \int_{F \cap \Gamma} \frac{\gamma}{2} (\sigma^+ - \sigma^-) \, d\mathcal{H}^{1} + \int_{F \cap \Gamma} \frac{\sigma^+ + \sigma^-}{2} \, \gamma \, d\mathcal{H}^{1} = \gamma \sigma^+,
\end{equation}
which perfectly matches the right-hand side boundary integral, where a non-vanishing traction trace $\mathbf{u}^+ \cdot (\MA^+ {\bm\nu}_{F}) = \gamma \sigma^+$ occurs exclusively along the upper edge of $F$. Indeed, on this portion of the boundary one has 
$\bm{\nu}_{F}={\bf e}_2$, 
$\mathbf{u}=\gamma{\bf e}_2$, and
$\MA^+\boldsymbol{\nu}_{F}=\sigma^+{\bf e}_2$, while the contributions from the lower and the lateral edges of $F$ vanish. 

Finally, we compute the two deviatoric pairings introduced in Section~\ref{sec:comparisonTemam}. Here $\Div\bfu={\rm tr}\,E\bfu=\gamma\,\Haus{1}\res\Gamma$ is singular with respect to $\Leb{2}$ and ${\rm tr}\,\MA=\sigma(x_2)$ jumps across $\Gamma$, so that the Kohn--Temam deviatoric pairing \eqref{eq:pairingBDTem2} is not defined. Since ${\rm tr}\,\MA\in BV(\Omega)$, assumption \eqref{assumtr} holds and $({\rm tr}\,\MA)^*={\rm tr}\,(\MA^*)=\frac{\sigma^++\sigma^-}{2}$ on $\Gamma$; hence, by Remark~\ref{t:remdevBV},
\[
((\MA^D:E^D\bfu))^*_\Xi=\MA^*:E\bfu-\frac12\,{\rm tr}\,(\MA^*)\,\Div\bfu=\frac{\sigma^++\sigma^-}{4}\,\gamma\,\Haus{1}\res\Gamma
\]
for every frame $\Xi$. Let now $\Xi$ be the standard frame of Section~\ref{sec:pairingbysliceBDbdd}. Then $a^{{\bf e}_1}=a^{({\bf e}_1+{\bf e}_2)/\sqrt2}=0$ and $a^{{\bf e}_2}=\sigma(x_2)$. By \eqref{f:goodrep}, the representatives of $a^{{\bf e}_1}$ and $a^{({\bf e}_1+{\bf e}_2)/\sqrt2}$ vanish everywhere, while the one-sided limits of $a^{{\bf e}_2}$ along ${\bf e}_2$ at the points of $\Gamma$ are $\sigma^\pm$, whatever the values of the representative on $\Gamma$. Hence ${\rm tr}\,\MA^\Xi=\bar a^{{\bf e}_2}=\frac{\sigma^++\sigma^-}{2}$ on $\Gamma$ for every choice of the representatives, and, since $((\MA:E\bfu))_\Xi=\MA^*:E\bfu$ by Theorem~\ref{consistenza}, definition \eqref{f:deviatoricslicing} gives
\[
((\MA^D:E^D\bfu))_\Xi=\frac{\sigma^++\sigma^-}{2}\,\gamma\,\Haus{1}\res\Gamma-\frac12\,\frac{\sigma^++\sigma^-}{2}\,\gamma\,\Haus{1}\res\Gamma=((\MA^D:E^D\bfu))^*_\Xi .
\]
Thus, in this example, the two definitions coincide, although $\Div^s\bfu\neq0$.

\section*{Acknowledgements}
The authors are members of  the Istituto Nazionale di Alta Matematica (INdAM), Gruppo Nazionale per l'Analisi Matematica, la Probabilità e le loro Applicazioni (GNAMPA), 
and were partially supported by the INdAM--GNAMPA 2026 Project \textit{Problemi Geometrici e Variazionali nelle PDE: Regolarità, Pairings e Analisi Isoperimetrica}, codice CUP E53C25002010001. 

\section*{Data availability}
We do not analyse or generate any datasets, because our work proceeds within a theoretical and mathematical approach. One can obtain the relevant materials from the references below.

\section*{Conflict of interest}

We declare that there is no Conflict of interest involved.

\def\cprime{$'$}

\end{document}